\documentclass[11pt,leqno,oneside,letterpaper]{amsart}
\usepackage[letterpaper,left=1.75truein, top=1truein]{geometry}
\usepackage{amsmath,amsfonts,amssymb,amscd,amsthm,amsbsy,epsf,graphics}

\usepackage{color}

\usepackage[colorlinks,linkcolor=blue,citecolor=blue]{hyperref} 
\usepackage{epsfig} 
\usepackage{tikz-cd}
\usepackage{bm}

\usepackage{graphicx}
\DeclareGraphicsExtensions{.pdf}

\newtheorem{thm}{Theorem}[section]
\newtheorem{cor}[thm]{Corollary}
\newtheorem{lemma}[thm]{Lemma}
\newtheorem{prop}[thm]{Proposition}
\newtheorem{conj}[thm]{Conjecture}

\newtheorem{que}[thm]{Question}

\newtheorem{claim}[thm]{Claim}

\newtheorem{problem}[thm]{Problem}
 
\newtheorem{remark}[thm]{Remark}

\theoremstyle{definition}

\newtheoremstyle{cases}
  {12pt plus 6 pt}
  {2pt}
  {\bfseries}   
  {}
  {\bfseries}
  {.}
  {.5em}
  {}
\theoremstyle{cases}

\newtheorem{case}{Case}
\newtheorem{subcase}{Subcase}

\numberwithin{subcase}{case} 
\numberwithin{subsubcase}{subcase}
\numberwithin{equation}{subsection} 

\begin{document}

\title{On definite strongly quasipositive links and L-space branched covers}

\author{Michel Boileau}  
\thanks{Michel Boileau was partially supported by ANR projects 12-BS01-0003-01 and 12-BS01-0004-01.}
\address{Aix Marseille Univ, CNRS, Centrale Marseille, I2M, Marseille, France, 39, rue F. Joliot Curie, 13453 Marseille Cedex 13}
\email{michel.boileau@cmi.univ-mrs.fr }
 
\author[Steven Boyer]{Steven Boyer}
\thanks{Steven Boyer was partially supported by NSERC grant RGPIN 9446-2013}
\address{D\'epartement de Math\'ematiques, Universit\'e du Qu\'ebec \`a Montr\'eal, 201 avenue du Pr\'esident-Kennedy, Montr\'eal, QC H2X 3Y7.}
\email{boyer.steven@uqam.ca}
\urladdr{http://www.cirget.uqam.ca/boyer/boyer.html}

\author[Cameron McA. Gordon]{Cameron McA. Gordon}
\thanks{Cameron Gordon was partially supported by NSF grant DMS-130902.}
\address{Department of Mathematics, University of Texas at Austin, 1 University Station, Austin, TX 78712, USA.}
\email{gordon@math.utexas.edu}
\urladdr{http://www.ma.utexas.edu/text/webpages/gordon.html}

\begin{abstract} 
We investigate the problem of characterising the family of strongly quasipositive links which have definite symmetrised Seifert forms and 
apply our results to the problem of determining 
when such a link can have an L-space cyclic branched cover. In particular, we show that if $\delta_n = \sigma_1 \sigma_2 \ldots \sigma_{n-1}$ 
is the dual Garside element and  
$b = \delta_n^k P \in B_n$ is a strongly quasipositive braid whose braid closure $\widehat b$ is definite, then $k \geq 2$ implies that $\widehat b$ is one of the torus links 
$T(2, q), T(3,4), T(3,5)$ or pretzel links $P(-2, 2, m), P(-2,3,4)$. Applying \cite[Theorem 1.1]{BBG} we deduce 
that if one of the standard cyclic branched covers of $\widehat b$ is an L-space, 
then $\widehat b$ is one 
of these links. We show by example that there are strongly quasipositive braids $\delta_n P$ whose closures are definite but 
not one of these torus or pretzel links. We also determine 
the family of definite strongly quasipositive $3$-braids and show that their closures coincide with the family of strongly quasipositive 
$3$-braids with an L-space branched cover.  \\

\end{abstract}

\maketitle

\begin{center}
\today 
\end{center}

\section{Introduction} \label{sec: introduction} 

We assume throughout that links are oriented and contained in the $3$-sphere. To each link $L$ and integer $n \geq 2$, we associate the canonical $n$-fold cyclic cover $\Sigma_n(L) \to S^3$ branched over $L$. 

We are interested in links that are fibred and strongly quasipositive. By theorems of Giroux and Rudolph (see \cite[102.1]{Ru2}) these are precisely the links that can be obtained from the unknot by plumbing and deplumbing positive Hopf bands. Recall that an {\it L-space} is a rational homology $3$-sphere such that $\mbox{dim}\; \widehat{HF}(M; \mathbb Z/2)) = |H_1(M; \mathbb Z)|$ \cite{OS1}, and an {\it L-space knot} is a knot with a non-trivial L-space Dehn surgery. L-space knots are fibred \cite{Ni1}, strongly quasipositive \cite[Theorem 1.2]{He}, and prime \cite{Krc}. We ask:

\begin{que}\label{que:lspaces}
{\rm For which fibred strongly quasipositive links $L$ is some $\Sigma_n(L)$ an L-space?}
\end{que}
Examples are provided by the following. 

Call a link $L$ {\it simply laced arborescent} if it is the boundary of an oriented surface obtained by plumbing positive Hopf bands according to one of 
the trees $\Gamma$ determined by the simply laced Dynkin diagrams $\Gamma = A_m (m \geq 1), D_m (m \geq 4), E_6, E_7, E_8$. We denote $L$ by $L(\Gamma)$. It is known that  

\indent \hspace{5.5cm} $L(A_m) = T(2, m+1)$

\indent \hspace{5.5cm} $L(D_m) = P(-2, 2, m-2)$

\indent \hspace{5.5cm} $L(E_6) = P(-2,3,3) = T(3,4)$

\indent \hspace{5.5cm} $L(E_7) = P(-2,3,4)$

\indent \hspace{5.5cm} $L(E_8) = P(-2,3,5) = T(3,5)$

where $T(p, q)$ is the $(p,q)$ torus link and $P(p, q, r)$ the $(p,q,r)$ pretzel link. For such a link $L$, $\Sigma_2(L)$ has finite fundamental group and is therefore an L-space \cite{OS1}. 

\begin{conj} 
\label{conj: branched lspace implies simply laced arborescent}
If $L$ is a prime, fibred, strongly quasipositive link for which some $\Sigma_n(L)$ is an L-space, then $L$ is simply laced arborescent. 
\end{conj}
\vspace{-.3cm} 
We say that a link $L$ of $m$ components is 
{\it definite} if $\vert \sigma(L) \vert = 2g(L) + (m-1)$, where $\sigma(L)$ is the signature of $L$ and $g(L)$ is its genus, and {\it indefinite} otherwise. One of the main results in our earlier paper  \cite{BBG} is 

\begin{thm} {\rm (\cite[Theorem 1.1(1)]{BBG})}
\label{thm: bbg definite}
A strongly quasipositive link $L$ for which some $\Sigma_n(L)$ is an L-space is definite.  
\end{thm}

For some subclasses of prime, fibred, strongly quasipositive links this immediately leads to a proof of Conjecture \ref{conj: branched lspace implies simply laced arborescent}. 
For example, the simply laced arborescent links are all definite positive braid links, i.e. closures of positive braids, and  Baader has shown:

\begin{thm} \label{thm: baader} {\rm (\cite[Theorem 2]{Baa})}
Let $L$ be a prime positive braid link. Then $L$ is simply laced arborescent if and only if it is definite. 
\end{thm}

Thus Conjecture \ref{conj: branched lspace implies simply laced arborescent} holds for positive braid links. In the same way, it 
was shown in \cite{BBG} that the conjecture holds if $L$ is prime and either a divide knot, a fibred strongly quasipositive knot which is 
either alternating or Montesinos, or an arborescent knot which bounds a surface obtained by plumbing positive Hopf bands along a tree. See 
\cite[Corollary 1.5 and Proposition 9.3]{BBG} and the remarks which follow the proof of the former. However, Theorem \ref{thm: bbg definite} is not sufficient to 
prove Conjecture \ref{conj: branched lspace implies simply laced arborescent} in general; there are prime, definite, fibred strongly quasipositive links that are not simply laced arborescent (see \cite{Mis} and Theorem \ref{thm: p odd}). The examples in Theorem \ref{thm: p odd} are basket links; see \cite{Ru1}.

A useful point of view concerning strongly quasipositive links is obtained through the consideration of Birman-Ko-Lee (BKL) positive braids \cite{BKL}. Here, the authors introduce a 
presentation for the braid group $B_n$ with generators the strongly quasipositive braids $a_{rs}$, $1 \leq r < s \leq n$, given by 
\begin{equation} \label{exprn for ars} 
a_{rs} = (\sigma_r \sigma_{r+1} \ldots \sigma_{s-2}) \sigma_{s-1}  (\sigma_r \sigma_{r+1} \ldots \sigma_{s-2})^{-1}
\end{equation}
Figure \ref{fig: steve 1} depicts the associated geometric braid.  

An element of $B_n$ is called {\it BKL-positive} if it can be expressed as a word in positive powers of the generators $a_{rs}$. The family of 
BKL-positive elements in $B_n$ coincides with the family of strongly quasipositive $n$-braids. The BKL-positive element 
\begin{equation} \label{def: deltan}
\delta_n = \sigma_1 \sigma_2 \ldots \sigma_{n-1} \in B_n,
\end{equation}
called the {\it dual Garside element}, plays a particularly important role below. 

Let $L$ be a strongly quasipositive link, so $L$ is the closure of a BKL-positive braid, and define the {\it{BKL-exponent}} of $L$ to be     
\begin{equation} \label{def: kl} 
k(L) = \max \{k : L \hbox{ is the closure of } \delta_n^k P \hbox{ where $n \geq 2, k \geq 0$, $P \in B_n$ is BKL-positive}\}
\end{equation}  
It is clear that $k(L) \geq 0$ and we show in Lemma \ref{lemma: k(L) finite}  that $ k(L) < \infty$.  
Consideration of the table in \S \ref{subsec: symm seifert forms of baskets} 
shows that $k(L) \geq 2$ when $L$ is simply laced arborescent.

The BKL-exponent $k(L)$ can be used to characterise the simply laced arborescent 
links amongst prime strongly quasipositive links.

\begin{thm} \label{thm: def baskets}
Let $L$ be a prime strongly quasipositive link. Then $L$ is simply laced arborescent if and only if it is definite and  $k(L) \geq 2$. 
\end{thm}

The condition $k(L) \geq 2$ on the 
BKL-exponent cannot be relaxed as there are prime strongly quasipositive definite links with $k(L) = 1$. See Theorem \ref{thm: p odd}. 

Theorems \ref{thm: bbg definite} and \ref{thm: def baskets} give a complete answer to Question \ref{que:lspaces} for prime strongly quasipositive links with BKL-exponent $k(L) \geq 2$.

\begin{cor}\label{cor:lspace cover} Let $L$ be a prime strongly quasipositive link with BKL-exponent $k(L) \geq 2$. Then $\Sigma_n(L)$ is an L-space for some $n \geq 2$  if and only if  $L$ is simply laced arborescent.
\end{cor} 

It  is worth noting that by \cite[Theorem 4.2]{Ban}, a strongly quasipositive link $L$ with $k(L) \geq 1$ is fibred. (Banfield also showed in \cite[Theorem 5.2]{Ban} 
that the family of strongly quasipositive links $L$ with $k(L) \geq 1$ coincides with the family of basket links; see \cite{Ru1} and the discussion below.) 
In particular, the strongly quasipositive links arising in Theorem \ref{thm: def baskets} and Corollary \ref{cor:lspace cover} are fibred. This reduces 
Conjecture \ref{conj: branched lspace implies simply laced arborescent} to the following statement.

\begin{conj} 
\label{conj: branched lspace implies simply laced arborescent 2}
If $L$ is a prime, fibred, strongly quasipositive link with BKL-exponent $k(L) \leq 1$, then no $\Sigma_n(L)$ is an L-space. 
\end{conj}

We completely determine the definite strongly quasipositive links of braid index $3$ or less, a family which includes the simply laced arborescent links (cf. the 
table in \S \ref{subsec: symm seifert forms of baskets}).  

\begin{thm}
\label{thm: definite 3-braids}
Suppose that $L$ is a non-split, non-trivial, strongly quasipositive link of braid index $2$ or $3$. 

$(1)$ If $L$ is prime, then the following statements are equivalent.
\vspace{-.2cm} 
\begin{itemize}
\item[{\rm (a)}] $L$ is definite;

\vspace{.2cm} \item[{\rm (b)}] $\Sigma_2(L)$ is an L-space;

\vspace{.2cm} \item[{\rm (c)}] $L$ is simply laced arborescent or a Montesinos link $M(1; 1/p,1/q,1/r)$ for some positive integers $p, q, r$.
\end{itemize}

$(2)$ If $L$ is definite, then it is fibred if and only if it is conjugate to a positive braid. 

\end{thm}
Part (2) of Theorem \ref{thm: definite 3-braids} 
combines with Theorem \ref{thm: baader} and \cite[Theorem 4.2] {Ban} to prove Theorem \ref{thm: def baskets} in the case of strongly quasipositive links 
of braid index $2$ or $3$. Part (1) of Theorem \ref{thm: definite 3-braids} resolves Question \ref{que:lspaces} for non-split strongly quasipositive links of braid index 
$2$ or $3$, even in the non-fibred case. Further, since 
the Montesinos links $M(1; 1/p,1/q,1/r)$ are not fibred (\cite[Theorem 1.1]{Ni2}, \cite[Theorem 3.3]{Stoi}), Theorem \ref{thm: bbg definite} combines with  
part (1) of the theorem to imply the following corollary. 

\begin{cor}
Conjecture \ref{conj: branched lspace implies simply laced arborescent} holds for fibred strongly quasipositive links of braid index $2$ or $3$. 
\qed
\end{cor}

We noted above that L-space knots are prime, fibred, and strongly quasipositive, so Theorem \ref{thm: bbg definite} and Theorem \ref{thm: definite 3-braids} also imply

\begin{cor}
If $K$ is a non-trivial L-space knot of braid index $2$ or $3$ and some $\Sigma_n(K)$ is an L-space, then $K$ is one of the torus knots $T(2,m)$ 
where $m \geq 3$ is odd, $T(3,4)$, or $T(3,5)$. 
\qed 
\end{cor}

Our next corollary follows immediately from Theorem \ref{thm: bbg definite} and the equivalence of (a) and (b) in part (1) of Theorem \ref{thm: definite 3-braids}. 

\begin{cor}
\label{cor: n implies 2} 
Suppose that $L$ is a strongly quasipositive link of braid index $2$ or $3$. If $\Sigma_n(L)$ is an L-space for some $n \geq 2$, then $\Sigma_2(L)$ is an L-space.
\qed
\end{cor}

Experimental evidence suggests that if $L$ is a link for which $\Sigma_n(L)$ is an L-space for some $n \geq 2$, then $\Sigma_r(L)$ is an L-space for each 
$2 \leq r \leq n$. In \S \ref{sec: lspace branched covers 3-braids} we verify this in all but three cases of strongly quasipositive links of braid index $2$ or $3$. 

\begin{prop} 
\label{prop: 1 < r < n}
Suppose that $L$ is a strongly quasipositive link of braid index $2$ or $3$ for which $\Sigma_n(L)$ is an L-space for some $n \geq 2$. If $L$ is not an appropriately oriented version of one of the links $6_2^2, 6_3^2$ or $7_1^3$, then $\Sigma_r(L)$ is an L-space for each $2 \leq r \leq n$. 
\end{prop}

See \S \ref{subsubsec: proof of proposition} for a discussion of the exceptional cases $L = 6_2^2, 6_3^2, 7_1^3$ and in particular 
Remark \ref{rmk: to do} for a description of what remains to be done to deal with these open cases.

By a {\it basket} we mean a positive Hopf plumbed basket in the sense of \cite{Ru1}. These are the surfaces $F$ in $S^3$ constructed by successively plumbing 
some number of positive Hopf bands onto a disk. The boundary of $F$ is a {\it basket link} $L$, which is fibred with fiber $F$. Also, $F$ is a quasipositive surface, 
and so $L$ is strongly quasipositive (\cite{Ru1}).  By  Banfield (\cite[Theorem 5.2]{Ban}  basket links coincide with strongly quasipositive links $L$ with  exponent $k(L) \geq 1$.

Our next result produces examples of prime definite basket links $L$ which are not  simply laced arborescent. We do this by considering basket links from the point of view of plumbing diagrams and studying the family of {\it cyclic basket links} $L(C_m, p)$ associated to an $m$-cycle $C_m$ incidence graph.

\begin{thm} 
\label{thm: p odd}
Let $m \geq 3$ and $p$ be integers with $p$ odd and $0 < p < m$. 

$(1)$ $L(C_m, p)$ is prime, fibred and definite. 

$(2)$ $L(C_m, p)$ is simply laced arborescent if and only if $p = 1$. 
 
\end{thm}

It follows from Theorem \ref{thm: def baskets} and \cite[Theorem 5.2]{Ban} that for $p > 1$ odd, the 
BKL-exponent of $L(C_m, p)$, is $1$. Our next result 
shows that even though these $L(C_m, p)$ are fibred, definite and not simply laced arborescent, they do not 
provide counterexamples to Conjecture \ref{conj: branched lspace implies simply laced arborescent 2}. 

\begin{thm} 
\label{thm: intro not a counterexample}
Suppose that $p$ is odd.

$(1)$ If $p > 1$ and $n \geq 3$, then $\Sigma_n(L(C_m, p))$ is not an L-space.

$(2)$ If $p > 1$, then $\Sigma_2(L(C_m, p))$ admits a co-oriented taut foliation and hence is not an L-space.

$(3)$ $\Sigma_n(L(C_m, 1))$ is an L-space if and only if either $n = 2$ or $n = m = 3$.

\end{thm}

By Lemma \ref{lemma: definite iff p odd} the links $L(C_m,p)$ with $p$ odd are the only definite basket links whose incidence graph is an $m$-cycle, $m \ge 3$. Hence Theorems \ref{thm: bbg definite}, \ref{thm: p odd} and \ref{thm: intro not a counterexample} give the following corollary. 

\begin{cor} 
\label{cor: no contradiction} 
If $L$ is a basket link whose incidence graph is an $m$-cycle, $m \geq 3$, then $\Sigma_n(L)$ is an L-space for some $n \geq 2$ if and only if $L$ is simply laced arborescent.
\qed
\end{cor}

This leads us to the following problem.

\begin{problem}
{\rm Determine the definite basket links and which of those have L-space branched cyclic covers.}
\end{problem} 

In \S \ref{sec: BKL braids and baskets} we discuss background material on basket links, their BKL braid representations, and their symmetrised Seifert forms. 
In \S \ref{sec: 3-braids} we classify definite strongly quasipositive $3$-braids leading to a proof of Theorem \ref{thm: definite 3-braids} 
and that of Theorem \ref{thm: def baskets} 
for strongly quasipositive links of braid index $3$ or less. Section \ref{sec: seifert forms of baskets} 
describes the Seifert form of a basket link in terms of a BKL-positive braid representation 
$b = \delta_n P$. This is then used to determine conditions guaranteeing that its symmetrised Seifert form is 
indefinite in \S \ref{sec: indefinite bkl pos words} and conditions guaranteeing it is congruent to $E_6, E_7$, or $E_8$ in \S \ref{sec: e6, e7, e8}. Sections \ref{sec: n = 4} 
and \ref{sec: n = 5} prove Theorem \ref{thm: def baskets} 
for strongly quasipositive links of braid index $4$ and $n \geq 5$ respectively. 
In \S \ref{sec: cycle baskets} we introduce cyclic basket links and prove Theorem \ref{thm: p odd} and 
Theorem \ref{thm: intro not a counterexample}. Finally, in \S \ref{sec: lspace branched covers 3-braids} we determine the
pairs $(b, n)$ where $b$ is a strongly quasipositive $3$-braid for which $\Sigma_n(\widehat b)$ is an 
L-space except for a finite number of cases. The third table in \S \ref{subsec: calculating alex polys} lists what is known to us 
and what remains open. Proposition \ref{prop: 1 < r < n} follows from this analysis.

{\bf Acknowledgements}. This paper originated during the authors' visits to the winter-spring 2017 thematic semester {\it Homology theories in low-dimensional topology} held at the Isaac Newton Institute for the Mathematical Sciences Cambridge (funded by EPSRC grant no. EP/K032208/1). They gratefully acknowledge their debt to this institute. They also thank Idrissa Ba for preparing the paper's figures.

\section{BKL-positive braids and basket links}
\label{sec: BKL braids and baskets}
Figure \ref{fig: braid conventions} illustrates our convention for relating topological and algebraic braids, and for the orientations on the components of braid closures. These differ from those of Murasugi \cite{Mu1} by replacing $\sigma_i$ by $\sigma_i^{-1}$ and those of  Birman-Ko-Lee \cite{BKL} and Baldwin \cite{Bld} by replacing $\sigma_i$ by $\sigma_{n-i}$. 

\begin{figure}[!ht]
\centering
 \includegraphics[scale=0.7]{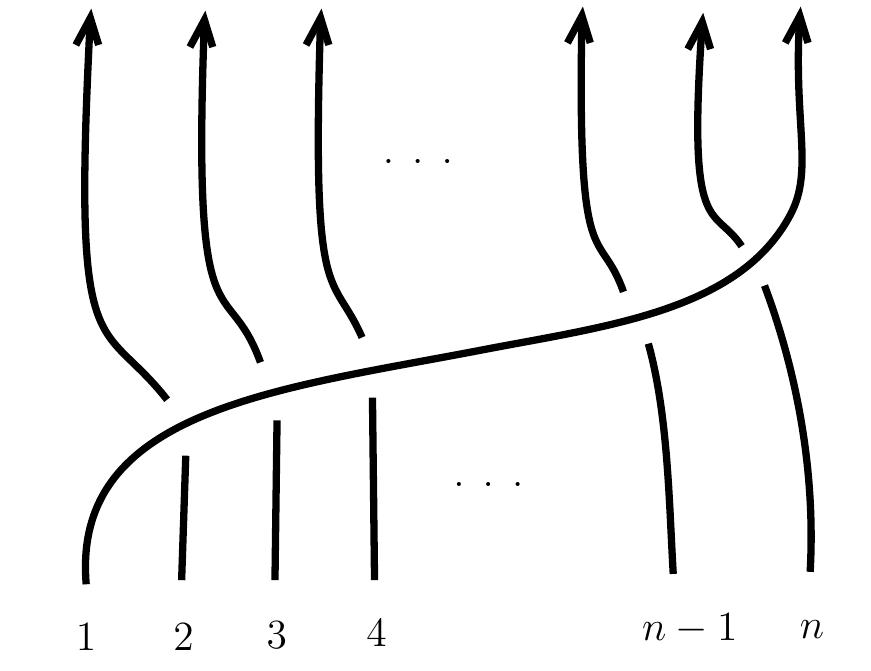} 
\caption{$\delta_n = \sigma_1 \sigma_2 \cdots \sigma_{n-1}$} 
\label{fig: braid conventions}
\end{figure} 

\subsection{BKL-positive braids}

The geometric braid corresponding to the Birman-Ko-Lee generator $a_{rs}$ of $B_n$ (cf. (\ref{exprn for ars})) is depicted in Figure \ref{fig: steve 1}. 
\begin{figure}[!ht]
\centering
 \includegraphics[scale=0.7]{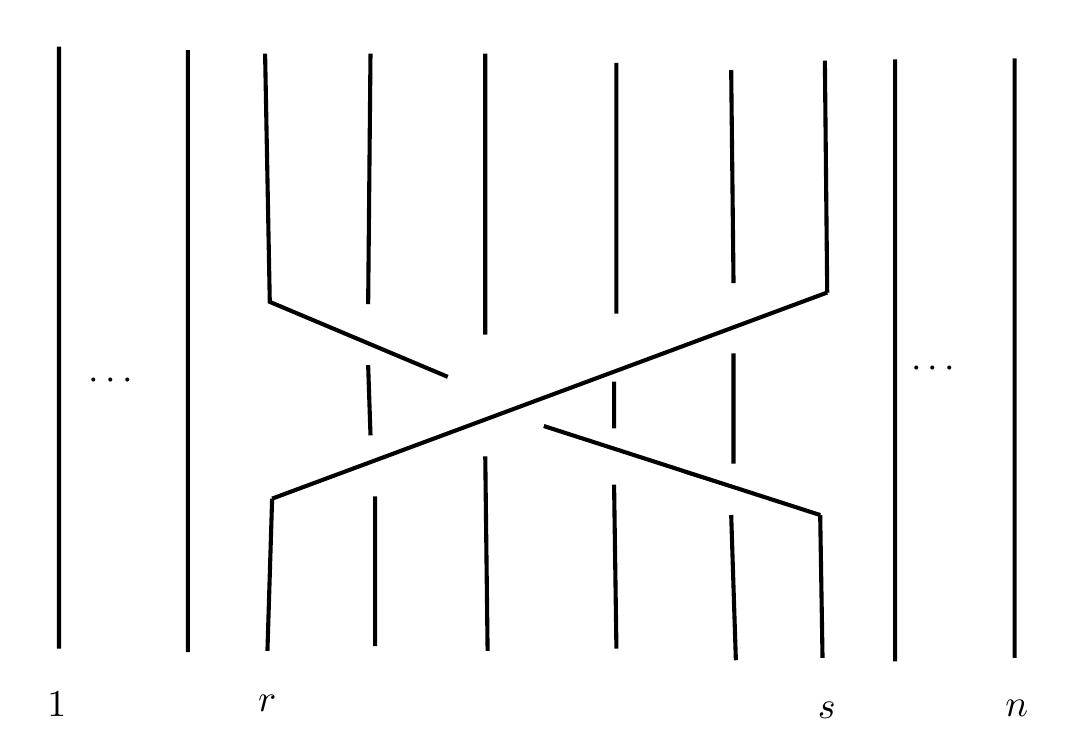}
\caption{$a_{rs}$} 
\label{fig: steve 1}
\end{figure} 
It is clear that $a_{rs} = \sigma_r$ when $s = r+1$. The classic braid relations show that $a_{rs} = (\sigma_r \sigma_{r+1} \ldots \sigma_{s-2}) \sigma_{s-1}  (\sigma_r \sigma_{r+1} \ldots \sigma_{s-2})^{-1}$ has an alternate expression: 
\begin{equation} \label{alt exprn for ars} 
a_{rs} =  (\sigma_{r+1} \sigma_{r+2} \ldots \sigma_{s-1})^{-1} \sigma_{r}  (\sigma_{r+1} \sigma_{r+2} \ldots \sigma_{s-1})
\end{equation} 
Recall the dual Garside element $\delta_n = \sigma_1 \sigma_2 \ldots \sigma_{n-1} \in B_n$ from the introduction. The $n^{th}$ power of $\delta_n$ generates the centre of $B_n$ and if we define $a_{r+1, n+1}$\footnote{For clarity, we will introduce a comma between the two parameters of the BKL generators from time to time.} to be $a_{1, r+1}$, then the reader will verify that for all $1 \leq r < s \leq n$ we have 
\begin{equation} 
\label{conj by delta}
\delta_n a_{rs} \delta_n^{-1} = a_{r+1, s+1} 
\end{equation} 

\begin{remark} 
\label{rem: delta rotn} 
{\rm Identity (\ref{conj by delta}) has a geometric interpretation. Think of the $n$ strands of the trivial braid in $B_n$ as $I$-factors of $S^1 \times I$ based at the $n^{th}$ roots of unity. A general element of $B_n$ is obtained by adding half twists to these numbered strands in the usual way. If $s \leq n-1$, the algebraic identity $\delta_n a_{rs} \delta_n^{-1} = a_{r+1, s+1}$ states that the braid element obtained by conjugating $a_{rs}$ by $\delta_n$ is the geometric braid obtained by rotating the geometric braid $a_{rs}$ around the $S^1$-factor of $S^1 \times I$ by $\frac{2 \pi}{n}$. This remains true even when $s = n$, for in this case if we follow the rotation by an isotopy of the half-twisted band corresponding to $a_{rs}$ through the point at $\infty$ of a plane transverse to $S^1 \times I$, we obtain $a_{1, r+1}$.}
\end{remark} 

The relations in the Birman-Ko-Lee presentation are 
\begin{eqnarray} 
\label{eqn: commuting generators}
a_{rs} a_{tu} = a_{tu} a_{rs} \;\; \mbox{ if } \;\; (t-r)(t-s)(u-r)(u-s) > 0
\end{eqnarray} 
and
\begin{eqnarray} 
\label{eqn: non-commuting generators}
a_{rs} a_{st} = a_{rt} a_{rs} = a_{st} a_{rt}  \;\; \mbox{ if } \;\; 1 \leq r < s < t \leq n
\end{eqnarray} 
We say that two BKL generators $a_{rs}$ and $a_{tu}$ are {\it linked} if either $r < t < s < u$ or $t < r < u < s$. Relation (\ref{eqn: commuting generators}) states that $a_{rs}$ commutes with $a_{tu}$ if $\{r, s\} \cap \{t, u\} = \emptyset$ and they are not linked. The converse holds; if $a_{rs}$ and $a_{tu}$ commute, then $\{r, s\} \cap \{t, u\} = \emptyset$ and they are not linked (\cite{BKL}).

We say that a BKL-positive word $P$ {\it covers} $\sigma_i$ if there is a letter $a_{rs}$ of $P$ for which $r \leq i < s$. 
The reader will verify that if a BKL-positive word $P$ covers $\sigma_i$ and $P' \in B_n$ is a BKL-positive rewriting of 
$P$, then $P'$ covers $\sigma_i$. 

The {\it span} of a BKL-positive letter $a_{rs}$ is $s - r$. 

\subsection{The quasipositive Seifert surface of a basket link}
Let $b \in B_n$ to be a braid of the form
$$b = \delta_n P$$ 
where $P$ is a BKL-positive braid. Since $b$ is a strongly quasipositive braid it determines a quasipositive surface $F(b)$ with oriented boundary $\widehat b$ obtained by attaching negatively twisted bands, one for each letter of $b$, to the disjoint union of $n$ disks in the usual way. See Figure \ref{fig: steve 2}. 

\begin{figure}[!ht] 
\centering
 \includegraphics[scale=0.7]{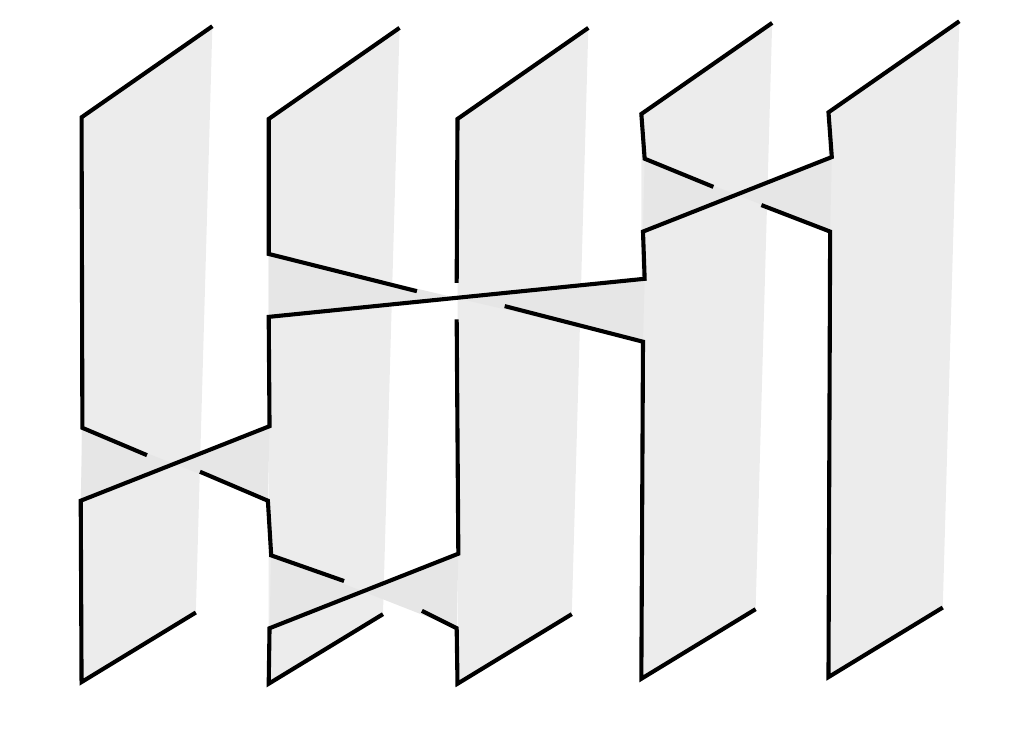}
\caption{Quasipositive surface} 
\label{fig: steve 2}
\end{figure} 

The assumption that $b = \delta_n P$ implies that $F(b)$ is connected. Hence it minimizes the first Betti number of the Seifert surfaces for $\widehat b$. It is readily calculated that $b_1(F(b)) = l_{BKL}(b) - (n-1)$ where $l_{BKL}(b)$ denotes the number of generators $a_{rs}$, counted with multiplicity, which occur in $b$. 

Recall the BKL-exponent $k(L)$ we defined for strongly quasipositive links $L$ (cf. (\ref{def: kl})).  

\begin{lemma} \label{lemma: k(L) finite} 
For $L$ a strongly quasipositive link, $0 \leq k(L) < \infty$.
\end{lemma}

\begin{proof} 
Let $L$ be a strongly quasipositive link, and let $b \in B_n$ be a braid of the form $\delta_n^k P$ where $k \geq 0, n \geq 2$, $P$ is BKL-positive, and $\widehat b = L$. If every such representation of $L$ has $k = 0$ then $k(L) = 0$ and we are done. So suppose that $L = \widehat b$ is as above with $k \geq 1$. Then $F = F(b)$ is a fibre surface for $L$ (\cite{Ban}) and hence independent of the choice of such a representation. From above, the first Betti number $b_1(F)$ of $F$ is given by $l_{BKL}(b) - (n-1) = (k-1)(n-1) + l_{BKL}(P)$, from which we see that $k$ is bounded above by $b_1(F) + 1$. 
\end{proof}

\subsection{Reducing braid indices} 
\label{subsec: reducing indices}
Let $L$ be a basket link for which $k(L) \geq 2$ and set 
$$n(L) = \min\{n : L = \widehat b \hbox{ where } b = \delta_{n}^k P, n \geq 1, k \geq 2, \hbox{ and } P \in B_{n} \hbox{ is BKL-positive}\}$$

\begin{lemma} \label{lemma: not minimal} 
If $b = \delta_n^2 P$ where $P$ does not cover $\sigma_1$ or $\sigma_{n-1}$, then $n(\widehat b) < n$.   
\end{lemma}

\begin{proof}
Suppose that $P$ does not cover $\sigma_1$. Since
\begin{eqnarray} 
\delta_n^2  & = & (\sigma_1 \sigma_2 \ldots \sigma_{n-1})(\sigma_1 \sigma_2 \ldots \sigma_{n-1}) \nonumber \\
& = & (\sigma_1 \sigma_2 \sigma_1)(\sigma_3 \ldots \sigma_{n-1})(\sigma_2 \ldots \sigma_{n-1}) \nonumber \\
& = & (\sigma_2 \sigma_1 \sigma_2)(\sigma_3 \ldots \sigma_{n-1})(\sigma_2 \ldots \sigma_{n-1}) \nonumber \\
& = & (\sigma_2 \sigma_1) (\sigma_2  \ldots \sigma_{n-1})(\sigma_2 \ldots \sigma_{n-1}), \nonumber
\end{eqnarray} 
we have $b = (\sigma_2 \sigma_1) (\sigma_2  \ldots \sigma_{n-1})(\sigma_2 \ldots \sigma_{n-1}) P$, in which $\sigma_1$ occurs as a letter exactly once. Then $\widehat b = \widehat b'$ where $b' = (\sigma_2  \ldots \sigma_{n-1})^2 P \sigma_2$, which completes the proof. 

A similar argument deals with the case that $P$ does not cover $\sigma_{n-1}$.
\end{proof}

\begin{lemma} 
\label{lemma: innermost commute}
Suppose that $b = \delta_n^2 P$ where $P$ is BKL-positive and contains a letter $a_{rs}$ which commutes with all other letters of $P$ and that there are 
no letters $a_{tu}$ of $P$ such that $r < t < u < s$. Then $n(\widehat b) < n$.
\end{lemma}

\begin{proof}
We use ``$\sim$" to denote ``conjugate to" in what follows.

Write $P = P_1 a_{rs} P_2$ where $P_1, P_2$ are BKL-positive. By (\ref{conj by delta}), conjugating $b$ by $\delta_n^{1-r}$ yields 
$$b' = \delta_n^2 P_1' a_{1s'} P_2' = P_1'' \delta_n^2 a_{1s'} P_2' \sim \delta_n^2 a_{1s'} P_2' P_1''$$
where $P_1', P_2', P_1''$ are BKL-positive words and there are no letters $a_{tu}$ of $P_2' P_1''$ such that $1 < t < u < s$.
Hence, without loss of generality $r = 1$. Then $P = a_{1s}^m P'$ where for each letter $a_{tu}$ of $P'$ we have $t, u \in \{s+1, s+2, \ldots, n\}$. 
 
If $s = 2$, then 
\begin{eqnarray} 
b = \delta_n^2 \sigma_1^m P' = \delta_n^2 P' \sigma_1^m& \sim & (\sigma_1^{m+1} \sigma_2 \sigma_3 \ldots \sigma_{n-1}) (\sigma_1 \sigma_2 \ldots \sigma_{n-1}) P' \nonumber \\ 
& = & (\sigma_1^{m+1} \sigma_2 \sigma_1)(\sigma_3 \ldots \sigma_{n-1}) (\sigma_2 \ldots \sigma_{n-1}) P' \nonumber \\ 
& = & (\sigma_2 \sigma_1 \sigma_2^{m+1})(\sigma_3 \ldots \sigma_{n-1}) (\sigma_2 \ldots \sigma_{n-1}) P', \nonumber
\end{eqnarray} 
which has the same closure as 
$$(\sigma_2^{m+2})(\sigma_3 \ldots \sigma_{n-1}) (\sigma_2 \ldots \sigma_{n-1}) P' \sim (\sigma_2\sigma_3 \ldots \sigma_{n-1}) (\sigma_2 \ldots \sigma_{n-1}) P' \sigma_2^{m+1}$$ 
This implies the conclusion of the lemma. 

If $s \geq 3$, then $\sigma_1$ commutes with $P'$ so as $a_{1s} = \sigma_1 a_{2s} \sigma_1^{-1}$, we have 
\begin{eqnarray} 
b = \delta_n^2 \sigma_1 a_{2s}^m \sigma_1^{-1}P' = \delta_n^2 \sigma_1 a_{2s}^m P' \sigma_1^{-1} & \sim & (\sigma_2 \sigma_3 \ldots \sigma_{n-1}) (\sigma_1 \sigma_2  \ldots \sigma_{n-1}) \sigma_1 a_{2s}^j P' \nonumber \\ 
& = & (\sigma_2 \sigma_3 \ldots \sigma_{n-1}) ((\sigma_1 \sigma_2 \sigma_1) \sigma_3 \ldots \sigma_{n-1}) a_{2s}^m P'  \nonumber \\ 
& = & (\sigma_2 \sigma_3 \ldots \sigma_{n-1}) ((\sigma_2 \sigma_1 \sigma_2) \sigma_3  \ldots \sigma_{n-1}) a_{2s}^m P' \nonumber 
\end{eqnarray} 
which has the same braid closure as 
\begin{eqnarray} 
(\sigma_2 \sigma_3 \sigma_4 \ldots \sigma_{n-1}) (\sigma_2^2 \sigma_3 \ldots \sigma_{n-1}) P' & = & ((\sigma_2 \sigma_3 \sigma_2) \sigma_4 \ldots \sigma_{n-1}) (\sigma_2 \ldots \sigma_{n-1}) P' \nonumber \\ 
&=& (\sigma_3 \sigma_2 \sigma_3 \sigma_4 \ldots \sigma_{n-1}) (\sigma_2 \ldots \sigma_{n-1}) P' \nonumber \\ 
&\sim&  (\sigma_2 \sigma_3  \ldots \sigma_{n-1}) (\sigma_2 \ldots \sigma_{n-1}) P' \sigma_3, \nonumber
\end{eqnarray} 
which implies the conclusion of the lemma. 
\end{proof}

\subsection{The symmetrised Seifert forms of basket links} 
\label{subsec: symm seifert forms of baskets} 

Symmetrised Seifert forms are even symmetric bilinear forms. In particular, if $b = \delta_n P$ (as above) is a basket link, the symmetrised Seifert form of the quasipositive surface $F(b)$ is an even symmetric bilinear form which we denote by 
$$\mathcal{F}(b): H_1(F(b)) \times H_1(F(b)) \to \mathbb Z$$
We say that $b$ is {\it definite} if $\mathcal{F}(b)$ is definite. Note that with respect to standard orientation conventions, if $b$ is definite then $\mathcal{F}(b)$ is negative definite\footnote{This holds more generally for the symmetrised Seifert form of the closure of any BKL positive braid.} and in this case we will see in \S \ref{sec: seifert forms of baskets} that $\mathcal{F}(b)$ is a root lattice. See Remark \ref{rem: root lattice}. In particular, it is congruent to an orthogonal sum of the simply laced arborescent forms $A_m, D_m, E_6, E_7$, and $E_8$\footnote{We consider the negative definite versions of these forms in this paper.}. Here are some simple examples of such braids whose associated forms are simply laced arborescent (cf. \cite{Baa}). Braids in the same row are equal modulo rewriting, conjugation, and (de)stabilisation. 

{\small 
\begin{center}
\begin{tabular}{|c||c|c|c|c|c|c|c|} \hline 
Form & $B_2$ & $B_3$ & $B_3$ & $B_4$ & $B_5$&  Braid Closure \\  \hline \hline 
$A_m$ ($m \geq 1$) & $\delta_2^{m+1}$ &$\delta_3 \sigma_1^m$ & $\sigma_1^{m+1} \sigma_2$ &  &&   $T(2,m+1)$ \\  \hline 
$D_m$ ($m \geq 4$)& &$\delta_3^3 \sigma_1^{m-4}$  &  $\sigma_1^{m-2} \sigma_2 \sigma_1^2 \sigma_2$ &  && $P(-2,2,m-2)$ ($= T(3,3)$ when $m = 4$) \\  \hline 
$E_6$ & &$\delta_3^4$ &  $\sigma_1^3 \sigma_2 \sigma_1^3 \sigma_2$ & $\delta_4^3$ &  &   $P(-2,3,3) = T(3,4)$  \\  \hline 
$E_7$ & &$\delta_3^4 \sigma_1$ &  $\sigma_1^4 \sigma_2 \sigma_1^3 \sigma_2$  &&  & $P(-2,3,4)$    \\  \hline 
$E_8$ & &$\delta_3^5$ &  $\sigma_1^5 \sigma_2 \sigma_1^3 \sigma_2$ & & $\delta_5^3$ &    $P(-2,3,5) = T(3,5)$   \\  \hline 
\end{tabular}
\end{center}}

Given two BKL-positive words $b, b' \in B_n$, we say that {\it $b$ contains $b'$ as a subword} if there are BKL-positive words $a, c \in B_n$ such that $b = a b' c$. 

More generally, we say that {\it $b$ contains $b'$}, written $b \supseteq b'$, if $b'$ can be obtained from $b$ by deleting some BKL-positive letters. If $b' \subseteq b$, then $F(b') \subseteq F(b)$ and the inclusion-induced homomorphism $H_1(F(b')) \to H_1(F(b))$ is injective with image a direct summand of $H_1(F(b))$. Consequently, $\mathcal{F}(b')$ is a primitive sublattice of $\mathcal{F}(b)$. 

For $P$ a BKL-positive word, set 
$$r(P) = \max\{s : \delta_n^s \subseteq P\}$$ 
Then if $b = \delta_n^k P$, $b \supseteq \delta_n^{k + r(P)}$ and therefore $\mathcal{F}(b)$ contains $\mathcal{F}(\delta_n^{k + r(P)})$ as a sublattice. 
The braid closure of $\delta_n^{k + r(P)}$ is the $(k+r(P), n)$ torus link $T(k+r(P),n)$ and from the table above we see that $\mathcal{F}(T(k+r(P), n))$ is definite in certain cases: 
$\mathcal{F}(T(2,n)) \cong A_{n-1}$, $\mathcal{F}(T(3,3)) \cong D_4$, $\mathcal{F}(T(3,4)) \cong E_6$, and $\mathcal{F}(T(3,5)) \cong E_8$. 
The converse is also true. 

\begin{lemma} \label{restrictions on krn} {\rm (\cite[Theorem 1]{Baa})} $\;$ 
$\delta_n^{k}$ is definite if and only if $\{k, n\}$ is $\{2, n\}, \{3,3\}, \{3,4\}$, or $\{3, 5\}$. Consequently, if $n \geq 3$ and $b = \delta_n^k P$ is BKL-positive and definite, then $k + r(P) \leq 5$. 
\qed
\end{lemma}

\section{Definite strongly quasipositive $3$-braids} 
\label{sec: 3-braids} 

The goal of this section is to prove Theorem \ref{thm: definite 3-braids}. 

\subsection{Minimal representatives of strongly quasipositive $3$-braids} 
\label{subsec: minimal rep}

Recall that strongly quasipositive $3$-braids $P$ are reduced products of words of the form $\sigma_1^p, a_{13}^q, \sigma_2^r$ where $p, q, r > 0$. We define the {\it syllable length} of a particular strongly quasipositive expression for $P$ to be the number of such words in its expression.

Given a strongly quasipositive $3$-braid $b$, choose a braid $\delta_3^k P$ from among its strongly quasipositive conjugates and their strongly quasipositive  rewritings, for which $k$ is maximal and for such maximal $k$, for which $P$ has minimal syllable length. Then $P$ does not contain the subwords $\sigma_1\sigma_2, \sigma_2 a_{13}, a_{13} \sigma_1$ as each of these equals $\delta_3$ which can be moved to the left through $P$ to yield a strongly quasipositive rewriting $b = \delta_3^{k+1} P'$, contrary to our choices. If $P \ne 1$, we can assume that its first letter is $\sigma_1$ after conjugating by an appropriate power of $\delta_3$. Then there are words $w_1, w_2, \ldots, w_s$ ($s \geq 0$) of the form 
$$w_i = \sigma_1^{p_i} a_{13}^{q_i} \sigma_2^{r_i}$$
with $p_i, q_i, r_i > 0$ for each $i$, such that $P = w_1w_2 \ldots w_s P'$ where $P' \in \{1, \sigma_1^{p}, \sigma_1^{p} a_{13}^{q}\}$ for some $p, q \geq 1$. 

We call an expression for $b$ of the sort just described a {\it minimal} representative of its conjugacy class.

\begin{lemma} 
\label{lemma: minimal form}
Let $\delta_3^k P$ be a minimal representative for a strongly quasipositive braid $b$ where $P = w_1w_2 \ldots w_s P'$ with $P' \in \{1, \sigma_1^{p}, \sigma_1^{p} a_{13}^{q}\}$ as above. 

$(1)$ Suppose that $s = 0$. Then $b$ is either $\delta_3^k$, or $\delta_3^k \sigma_1^p$ for some $p \geq 1$, or $\delta_3^k \sigma_1^p a_{13}^q$ where $k \equiv 1$ {\rm (mod $3$)} and $p, q \geq 1$. In all cases, $b$ is conjugate to a positive braid. 

$(2)$ Suppose that $s > 0$. Then 
$$P = \left\{ 
\begin{array}{ll}
 w_1w_2 \ldots w_s   & \hbox{if } k  \equiv 0 \hbox{ {\rm (mod $3$)}} \\
w_1w_2 \ldots w_s \sigma_1^{p_{s+1}} a_{13}^{q_{s+1}} & \hbox{if } k \equiv 1 \hbox{ {\rm (mod $3$)}}  \\
w_1w_2 \ldots w_s\sigma_1^{p_{s+1}}  & \hbox{if } k \equiv 2 \hbox{ {\rm (mod $3$)}}
\end{array} \right.$$
\end{lemma}

\begin{proof}
First suppose that $s = 0$. Then $b = \delta_3^k P$ where $P \in \{1, \sigma_1^{p}, \sigma_1^{p} a_{13}^{q}\}$. If $P$ is either $1$ or $\sigma_1^p$, we are done. Assume then that $P = \sigma_1^{p} a_{13}^{q}$. If $k \equiv 0$ (mod $3$), then $b \sim (a_{13}^q \delta_3^k) \sigma_1^p = (\delta_3^k a_{13}^q) \sigma_1^{p} = \delta_3^k \sigma_1 \sigma_2^q \sigma_1^{p -1} = \delta_3^{k+1} \sigma_2^{q-1} \sigma_1^{p -1}$, which contradicts the maximality of $k$. If $k \equiv 2$ (mod $3$), then $b \sim (a_{13}^q \delta_3^k) \sigma_1^p = (\delta_3^k \sigma_1^{q}) \sigma_1^{p} = \delta_3^k \sigma_1^{p + q}$, which contradicts the minimality of the syllable length of $P$. Thus, $k \equiv 1$ (mod $3$). Finally observe that $\delta_3 (\delta_3^k \sigma_1^{p} a_{13}^{q}) \delta_3^{-1}$ is the positive braid $\delta_3^k \sigma_2^{p} \sigma_1^{q}$. 

Next suppose that $s> 0$. 

If $k \equiv 0$ (mod $3$) and $P' = \sigma_1^{p} a_{13}^{q}$ for some $p, q \geq 1$, then 
$$\delta_3^k P \sim \delta_3^k \sigma_1^{p} a_{13}^{q} w_1w_2 \ldots w_s = \delta_3^k \sigma_1^{p} \delta_3 \sigma_2^{q-1} w_1' w_2 \ldots w_s 
= \delta_3^{k+1} a_{13}^p \sigma_2^{q-1}w_1' w_2 \ldots w_s,$$
where $w_1' = \sigma_1^{p_1 - 1} a_{13}^{q_1} \sigma_2^{r_1}$, which contradicts the maximality of $k$. 
On the other hand, if $k \equiv 0$ (mod $3$) and $P' = \sigma_1^{p}$ then 
$$\delta_3^k P \sim \sigma_1^{p} \delta_3^k w_1w_2 \ldots w_s = \delta_3^k \sigma_1^{p} w_1w_2 \ldots w_s = \delta_3^k w_1' w_2 \ldots w_s $$
where $w_1' = \sigma_1^{p_1 + p} a_{13}^{q_1} \sigma_2^{r_1}$, which contradicts the minimality of the syllable length of $P$. Thus $P' = 1$.

If $k \equiv 1$ (mod $3$) and $P' = \sigma_1^{p}$ for some $p \geq 1$ then 
\begin{eqnarray} 
\delta_3^k P \sim \sigma_1^{p} \delta_3^k w_1w_2 \ldots w_s = \delta_3^k a_{13}^{p} w_1w_2 \ldots w_s & = & \delta_3^{k+1} \sigma_2^{p-1} \sigma_1^{p_1-1} a_{13}^{q_1} \sigma_2^{r_1}  w_2 \ldots w_s, \nonumber
\end{eqnarray} 
which contradicts the maximality of $k$. On the other hand, if $k \equiv 1$ (mod $3$) and $P' = 1$ then 
$$\delta_3^k P \sim \sigma_2^{r_s} \delta_3^k w_1 \ldots w_{s-1} \sigma_1^{p_s} a_{13}^{q_s} = \delta_3^k \sigma_1^{r_s} w_1w_2 \ldots w_{s-1} \sigma_1^{p_s} a_{13}^{q_s} = \delta_3^k w_1' w_2 \ldots w_{s-1} \sigma_1^{p_s} a_{13}^{q_s}$$ 
where $w_1' = \sigma_1^{r_s + p_1} a_{13}^{q_1} \sigma_2^{r_1}$, which contradicts the minimality of the syllable length of $P$. Thus $P' = \sigma_1^{p} a_{13}^{q}$ for some $p, q \geq 1$. 

Finally, if $k \equiv 2$ (mod $3$) and $P' = 1$ then 
\begin{eqnarray} 
\delta_3^k P \sim \sigma_2^{r_s} \delta_3^k w_1w_2 \ldots w_{s-1} \sigma_1^{p_s} a_{13}^{q_s} & = & \delta_3^k a_{13}^{r_s} w_1 \ldots w_{s-1} \sigma_1^{p_s} a_{13}^{q_s}  \nonumber \\
& = & \delta_3^{k+1} \sigma_2^{r_s-1} \sigma_1^{p_1-1} a_{13}^{q_1} \sigma_2^{r_1}  w_2 \ldots w_{s-1} \sigma_1^{p_s} a_{13}^{q_s}, \nonumber
\end{eqnarray} 
which contradicts the maximality of $k$. On the other hand, if $k \equiv 2$ (mod $3$) and $P' = \sigma_1^{p} a_{13}^{q}$ for some $p, q \geq 1$ then 
$$\delta_3^k P \sim a_{13}^q \delta_3^k w_1 \ldots w_{s} \sigma_1^{p} = \delta_3^k \sigma_1^q w_1w_2 \ldots w_{s}\sigma_1^{p}  = \delta_3^k w_1' w_2 \ldots w_{s} \sigma_1^{p}$$ 
where $w_1' = \sigma_1^{q + p_1} a_{13}^{q_1} \sigma_2^{r_1}$, which contradicts the minimality of the syllable length of $P$. Thus $P' = \sigma_1^p$ for some $p \geq 1$.  
\end{proof}

 \subsection{The Murasugi normal form of strongly quasipositive $3$-braids} 

In order to calculate the signature of the closure of $\delta_3^k P$, we determine the Murasugi normal form for its conjugacy class. 

\begin{lemma} 
\label{lemma: simple rewrite}
Suppose that $p,q, r$ are three positive integers. 

$(1)$ $\sigma_1^p = \sigma_2^{-1} (\sigma_1^{-1} \sigma_2^{p-1} \delta_3) \sigma_1$.

$(2)$  $\sigma_1^p a_{13}^q = \sigma_2^{-1} (\sigma_1^{-1} \sigma_2^{p-1} \sigma_1^{-1} \sigma_2^{q-1}\delta_3^2) a_{13}$

$(3)$ $\sigma_1^p a_{13}^q \sigma_2^r =  \sigma_2^{-1} (\delta_3^3 \sigma_1^{-1} \sigma_2^{p-1} \sigma_1^{-1} \sigma_2^{q-1} \sigma_1^{-1} \sigma_2^{r-1})\sigma_2$. 

\end{lemma} 

\begin{proof} We have, 
$$\sigma_1^p = \sigma_2^{-1} \sigma_2 \sigma_1^p  = \sigma_2^{-1} (\sigma_2 \delta_3^{-1} \sigma_2^{p} \delta_3 )  = \sigma_2^{-1} (\sigma_1^{-1} \sigma_2^{p-1} \delta_3 \sigma_1)  = \sigma_2^{-1} (\sigma_1^{-1} \sigma_2^{p-1} \delta_3) \sigma_1,$$ 
and 
$$\sigma_1^p a_{13}^q = \sigma_2^{-1} (\sigma_2 \sigma_1^p a_{13}^q)  =  \sigma_2^{-1} (\sigma_2 \delta_3^{-1} \sigma_2^{p} \delta_3^{-1} \sigma_2^q\delta_3^2) = \sigma_2^{-1} (\sigma_1^{-1} \sigma_2^{p-1} \sigma_1^{-1} \sigma_2^{q-1}\delta_3^2) a_{13},$$ 
and finally 
\begin{eqnarray}
\sigma_1^p a_{13}^q \sigma_2^r = \sigma_2^{-1} (\sigma_2 \sigma_1^p a_{13}^q \sigma_2^{r-1})\sigma_2 & = & \sigma_2^{-1} (\delta_3 \sigma_1 a_{13}^p \sigma_2^q \delta_3^{-1} \sigma_2^{r-1})\sigma_2 \nonumber  \\ 
&  = &  \sigma_2^{-1} (\delta_3^2 a_{13} \sigma_2^p \delta_3 ^{-1} \sigma_2^q \delta_3^{-1} \sigma_2^{r-1})\sigma_2 \nonumber  \\
& = & \sigma_2^{-1} (\delta_3^3 \sigma_2 \delta_3^{-1} \sigma_2^p \delta_3^{-1} \sigma_2^q \delta_3^{-1} \sigma_2^{r-1})\sigma_2  \nonumber   \\ 
& = & \sigma_2^{-1} (\delta_3^3 \sigma_1^{-1} \sigma_2^{p-1} \sigma_1^{-1} \sigma_2^{q-1} \sigma_1^{-1} \sigma_2^{r-1})\sigma_2  \nonumber  
\end{eqnarray}
\end{proof}

\begin{prop} 
\label{prop: rewrite}
Let $\delta_3^k P$ be a minimal representative for a strongly quasipositive $3$-braid $b$ where $P = w_1w_2 \ldots w_s P'$ with $P' \in \{1, \sigma_1^{p}, \sigma_1^{p} a_{13}^{q}\}$ as above. Suppose that $s > 0$. 

$(1)$ If $k = 3r$, then $\displaystyle b \sim \delta_3^{3(r+s)} \prod_{i=1}^s \sigma_1^{-1} \sigma_2^{p_i-1} \sigma_1^{-1}\sigma_2^{q_i-1} \sigma_1^{-1} \sigma_2^{r_i-1}$;

$(2)$ If $k = 3r + 1$, then $\displaystyle b \sim \delta_3^{3(r +s+1)} \big(\prod_{i=1}^s \sigma_1^{-1} \sigma_2^{p_i-1} \sigma_1^{-1}\sigma_2^{q_i-1} \sigma_1^{-1} \sigma_2^{r_i-1}\big) \sigma_1^{-1}  \sigma_2^{p_{s+1}-1}  \sigma_1^{-1}  \sigma_2^{q_{s+1}-1}$; 

$(3)$ If $k = 3r + 2$, then $\displaystyle b \sim \delta_3^{3(r + s+1)} \big(\prod_{i=1}^s \sigma_1^{-1} \sigma_2^{p_i-1} \sigma_1^{-1}\sigma_2^{q_i-1} \sigma_1^{-1} \sigma_2^{r_i-1}\big) \sigma_1^{-1}  \sigma_2^{p_{s+1} - 1}$.

\end{prop} 

\begin{proof}
Suppose that $k = 3r$. Then Lemma \ref{lemma: minimal form}(2) and Lemma \ref{lemma: simple rewrite}(3) show that 
\vspace{-.2cm}
\begin{eqnarray} 
b =  \delta_3^k \prod_{i=1}^s \sigma_1^{p_i} a_{13}^{q_i} \sigma_2^{r_i} & = & \delta_3^{3r} \prod_{i=1}^s \sigma_2^{-1} \big(\delta_3^3 \sigma_1^{-1} \sigma_2^{p_i-1}  \sigma_1^{-1} \sigma_2^{q_i-1}  \sigma_1^{-1} \sigma_2^{r_i-1}\big) \sigma_2 \nonumber \\ 
& = & \sigma_2^{-1} \big(  \delta_3^{3(r+s)}\prod_{i=1}^s \sigma_1^{-1} \sigma_2^{p_i-1}  \sigma_1^{-1} \sigma_2^{q_i-1}  \sigma_1^{-1} \sigma_2^{r_i-1}\big) \sigma_2, \nonumber 
\end{eqnarray}
which is Assertion (1). 

Suppose that $k = 3r + 1$. Lemma \ref{lemma: minimal form}(2) and Lemma \ref{lemma: simple rewrite}(2) show that
\vspace{-.2cm}
\begin{eqnarray} 
b & = & \delta_3^{3r+1} \big(\prod_{i=1}^s \sigma_1^{p_i} a_{13}^{q_i} \sigma_2^{r_i}\big)\sigma_1^{p_{s+1}} a_{13}^{q_{s+1}} \nonumber \\
& = & \delta_3^{3r+1}  \big(\sigma_2^{-1} \big( \delta_3^{3s} \prod_{i=1}^s \sigma_1^{-1} \sigma_2^{p_i-1}  \sigma_1^{-1} \sigma_2^{q_i-1}  \sigma_1^{-1} \sigma_2^{r_i-1}\big) \sigma_2\big) \big(\sigma_2^{-1}   \big(\sigma_1^{-1}  \sigma_2^{p_{s+1}-1}  \sigma_1^{-1}  \sigma_2^{q_{s+1}-1} \delta_3^2 \big) a_{13}  \big)\nonumber \\ 
& = & a_{13}^{-1} \big(\delta_3^{3r + 3s+1} \prod_{i=1}^s \sigma_1^{-1} \sigma_2^{p_i-1}  \sigma_1^{-1} \sigma_2^{q_i-1}  \sigma_1^{-1} \sigma_2^{r_i-1}\big)\big(\sigma_1^{-1}  \sigma_2^{p_{s+1}-1}  \sigma_1^{-1}  \sigma_2^{q_{s+1}-1} \delta_3^2 \big) a_{13} \nonumber \\ 
& = & (\delta_3^2 a_{13})^{-1} \delta_3^{3(r + s+1)} \big( \prod_{i=1}^s \sigma_1^{-1} \sigma_2^{p_i-1}  \sigma_1^{-1} \sigma_2^{q_i-1}  \sigma_1^{-1} \sigma_2^{r_i-1}\big)\big(\sigma_1^{-1}  \sigma_2^{p_{s+1}-1}  \sigma_1^{-1}  \sigma_2^{q_{s+1}-1} \big)  (\delta_3^2 a_{13})\nonumber  
\end{eqnarray}
This is Assertion (2). 

Finally, suppose that $k = 3r + 2$. Lemma \ref{lemma: minimal form}(2) and Lemma \ref{lemma: simple rewrite}(1) show that
\vspace{-.2cm}
\begin{eqnarray} 
b & = & \delta_3^{3r + 2} \big(\prod_{i=1}^s \sigma_1^{p_i} a_{13}^{q_i} \sigma_2^{r_i}\big)\sigma_1^{p_{s+1}} \nonumber \\
& = & \delta_3^{3r+2}  \big(\sigma_2^{-1} \big( \delta_3^{3s} \prod_{i=1}^s \sigma_1^{-1} \sigma_2^{p_i-1}  \sigma_1^{-1} \sigma_2^{q_i-1}  \sigma_1^{-1} \sigma_2^{r_i-1}\big) \sigma_2\big) \big(\sigma_2^{-1}   \big(\sigma_1^{-1}  \sigma_2^{p_{s+1}-1}  \delta_3 \big) \sigma_1  \big)\nonumber \\ 
& = & \sigma_1^{-1} \big(\delta_3^{3r + 3s+2} \prod_{i=1}^s \sigma_1^{-1} \sigma_2^{p_i-1}  \sigma_1^{-1} \sigma_2^{q_i-1}  \sigma_1^{-1} \sigma_2^{r_i-1}\big)\big(\sigma_1^{-1}  \sigma_2^{p_{s+1}-1}  \delta_3 \big) \sigma_1\nonumber \\ 
& = & (\delta_3 \sigma_1)^{-1} \delta_3^{3(r + s+1)} \big( \prod_{i=1}^s \sigma_1^{-1} \sigma_2^{p_i-1}  \sigma_1^{-1} \sigma_2^{q_i-1}  \sigma_1^{-1} \sigma_2^{r_i-1}\big)\big(\sigma_1^{-1}  \sigma_2^{p_{s+1}-1} \big) (\delta_3 \sigma_1)  \nonumber  
\end{eqnarray}
This completes the proof. 
\end{proof}

\subsection{Definite strongly quasipositive $3$-braids} 

Murasugi \cite{Mu1} and Erle \cite{Erle} have determined the signatures of the closures of $3$-braids. 

\begin{prop}  
\label{prop: signatures} 
Suppose that 
$b = \delta_3^{3d} \sigma_1^{-1} \sigma_2^{a_1} \sigma_1^{-1} \sigma_2^{a_2} \cdots \sigma_1^{-1} \sigma_2^{a_n}$ where each $a_j \geq 0$. Then there is an 
integer $\epsilon(d) \in \{-1, 0, 1\}$ congruent to $d$ (mod $2$) such that   
$$\sigma(\widehat b) = \left\{ \begin{array}{ll} n - 4d  - 1 + \epsilon(d) & \mbox{ if each $a_j = 0$} \\ n - 4d - \sum_{i=1}^n a_i & \mbox{ if some $a_j \ne 0$}   \end{array} \right.$$  
\qed 
\end{prop}

\begin{prop} 
\label{prop: when definite}
Let $\delta_3^k P$ be a minimal representative for a strongly quasipositive $3$-braid $b$ where $P = w_1w_2 \ldots w_s P'$ with $P' \in \{1, \sigma_1^{p}, \sigma_1^{p} a_{13}^{q}\}$ as above. 

$(1)$ If $P = 1$, then $\widehat b = T(3, k)$. In particular, $\widehat b$ is definite if and only if $b = \delta_3^k$ where $k \leq 5$.

$(2)$ If $P \ne 1$ and $s = 0$, then $b$ is definite if and only if either 

$(a)$ $b = \delta_3^k \sigma_1^p$ for some $p \geq 1$ where $k \leq 3$ or $k = 4$ and $p \in \{1, 2\}$.

$(b)$ $b = \delta_3 \sigma_1^p a_{13}^q$ where $p, q \geq 1$. In this case, $\widehat b = T(2, p+1) \# T(2, q+1)$. 

$(3)$ If $P \ne 1, s > 0$, and $p_i = q_i = r_i = 1$ for all values of $i$, then $b$ is definite if and only if $b = \sigma_1 a_{13} \sigma_2$.

$(4)$ If $k = 3r, s> 0$, and $\max\{p_i, q_i, r_i\} \geq 2$, then $b$ is definite if and only if $b = \sigma_1^{p_1} a_{13}^{q_1} \sigma_2^{r_1}$.  

$(5)$ If $k = 3r + 1, s >0$, and $\max\{p_i, q_i, r_i\} \geq 2$, then $b$ is indefinite. 

$(6)$ If $k = 3r + 2, s> 0$, and $\max\{p_i, q_i, r_i\} \geq 2$, then $b$ is indefinite. 
\end{prop} 
 \begin{proof}
(1) This is well known. 
 
(2) Lemma \ref{lemma: minimal form}(1) implies that $b$ is either $\delta_3^k \sigma_1^p$ for some $p \geq 1$ or $\delta_3^k \sigma_1^p a_{13}^q$ where $k \equiv 1$ {\rm (mod $3$)} and $p, q \geq 1$. Suppose that $b$ is definite. If $b = \delta_3^k \sigma_1^p$ for some $p \geq 1$, then $k \leq 5$ as $b \supset \delta_3^k$. One easily verifies that $b$ is always definite if $k = 0$. Further,
$$b \sim \left\{
\begin{array}{ll}
\sigma_1^{p+1} \sigma_2 & \mbox{ if } k = 1 \\
\sigma_1^{p+3} \sigma_2 & \mbox{ if } k = 2 \\
\sigma_1^{p+2} \sigma_2 \sigma_1^2 \sigma_2 & \mbox{ if } k = 3 \\
\sigma_1^{p+3} \sigma_2 \sigma_1^3 \sigma_2 & \mbox{ if } k = 4 \\
\sigma_1^{p+2} \sigma_2 \sigma_1^3 \sigma_2\sigma_1^{2} \sigma_2 & \mbox{ if } k = 5 \\
\end{array} \right.$$
The symmetrised Seifert form of $\widehat b$ is congruent to $A_{p}, A_{p+2}, D_{p+4}$ when $k = 1,2,3$, so is always definite. When $k = 4$, it is $E_7, E_8$ when $p = 1,2$ (cf. \cite[page 351]{Baa}), and therefore definite, and is indefinite for $p \geq 3$ (cf. fifth paragraph of \cite[page 356]{Baa}). The form is always indefinite when $k = 5$, as is noted in the second paragraph of \cite[page 356]{Baa}. Thus $b = \delta_3^k \sigma_1^p$ is definite if and only if $k \leq 3$ or $k = 4$ and $p \in \{1, 2\}$.

Next suppose that $b = \delta_3^k \sigma_1^p a_{13}^q$ where $k \equiv 1$ {\rm (mod $3$)} and $p, q \geq 1$. Write $k = 3r +1$ where $r$ is a non-negative integer. If $r \geq 2$, then $b \supset \delta_3^6$, so is indefinite. If $r = 0$, then $b = \delta_3 \sigma_1^p a_{13}^q \sim \delta_3 \sigma_2^p \sigma_1^q \sim \sigma_1^{q+1} \sigma_2^{p+1}$. Then $\widehat b = T(2, q+1) \# T(2, p+1)$, so is definite. If $r =1$, then 
\begin{eqnarray} 
b = \delta_3^4 \sigma_1^p a_{13}^q  \sim  \delta_3^3 \sigma_1^p \delta_3 \sigma_2^q =  \delta_3^3 \sigma_1^{p+1} \sigma_2^{q+1} & \sim & \delta_3^2 \sigma_1^{p+1} (\sigma_2^{q+1} \sigma_1 \sigma_2) \nonumber  \\ 
& = & \delta_3^2 \sigma_1^{p+1} (\sigma_1 \sigma_2 \sigma_1^{q+1})  \nonumber  \\
& = & \delta_3^2 \sigma_1^{p+2} \sigma_2 \sigma_1^{q+1} \nonumber  \\ 
& \sim &  \sigma_1^{p+2} \sigma_2 \sigma_1^{q+2} \sigma_2 \sigma_1 \sigma_2 \nonumber \\
& =&  \sigma_1^{p+2} \sigma_2 \sigma_1^{q+3} \sigma_2 \sigma_1   \nonumber \\
& \sim &  \sigma_1^{p+3} \sigma_2 \sigma_1^{q+3} \sigma_2  \nonumber 
\end{eqnarray}
Since $p+3, q+3 \geq 4$, $b$ is indefinite by the fifth paragraph of \cite[page 356]{Baa}. Thus when $b = \delta_3^k \sigma_1^p a_{13}^q$ where $k \equiv 1$ {\rm (mod $3$)} and $p, q \geq 1$, it is definite if and only if $k = 1$. 

(3) Suppose that $P \ne 1, s > 0$, and $p_i = q_i = r_i = 1$ for all values of $i$. Then Lemma \ref{lemma: minimal form}(2) implies that 
\begin{eqnarray} 
b_1(F(b)) & = & \left\{ \begin{array}{ll}  6r + 3s - 2 & \mbox{ if } k = 3r \\ 6r + 3s  + 2 & \mbox{ if } k = 3r + 1 \\ 6r + 3s + 3 & \mbox{ if } k = 3r + 2 \end{array} \right. \nonumber 
\end{eqnarray}  
On the other hand, it follows from Proposition \ref{prop: rewrite} that $b \sim \delta_3^{3d(b)} \sigma_1^{-e(b)}$ where 
$$d(b) = \left\{ \begin{array}{ll}  s + r & \mbox{ if } k = 3r \\ s + r + 1 & \mbox{ otherwise} \end{array} \right. \; \; \mbox{ and } \;\; 
e(b) = \left\{ \begin{array}{ll}  3s & \mbox{ if } k = 3r \\ 3s + 2 & \mbox{ if } k = 3r + 1 \\ 3s + 1 & \mbox{ if } k = 3r + 2 \end{array} \right.$$
Then Proposition \ref{prop: signatures} shows that 
\begin{eqnarray} 
\sigma(\widehat b) = e(b) - 4d(b) - 1 + \epsilon(d(b)) & = & \left\{ \begin{array}{ll} 3s - 4(s+r)  - 1 + \epsilon(s+r) & \mbox{ if } k = 3r \\ 
(3s + 2) - 4(s+r+1) - 1  + \epsilon(s+r+ 1) & \mbox{ if } k = 3r + 1 \\ (3s + 1) - 4(s+r+1) - 1 + \epsilon(s+r+ 1) & \mbox{ if } k = 3r + 2 \end{array} \right. \nonumber \\
& = & \left\{ \begin{array}{ll}  -(s + 4r) - 1 + \epsilon(s+r) & \mbox{ if } k = 3r \\ -(s + 4r) - 3 + \epsilon(s+r+ 1) & \mbox{ if } k = 3r + 1 \\ -(s + 4r) - 4 + \epsilon(s+r+ 1) & \mbox{ if } k = 3r + 2 \end{array} \right. \nonumber 
\end{eqnarray}  
Then $\widehat b$ is definite if and only if 
\begin{eqnarray} 
\label{eqn: r, s}
0 = b_1(F(b)) + \sigma(\widehat b) =  \left\{ \begin{array}{ll}  2(r+s) - 3 + \epsilon(s+r) & \mbox{ if } k = 3r \\ 2(r+s) - 1 + \epsilon(s+r+1) & \mbox{ if } k = 3r + 1 \\ 2(r+s) - 1 + \epsilon(s+r+ 1)& \mbox{ if } k = 3r + 2 \end{array} \right. 
\end{eqnarray}  
Now if $b \supset \delta_3^k$ is definite, we have $k \leq 5$. We assume this below. Then writing $k = 3r + t$ where $t = 0, 1, 2$, we have $r = 0, 1$. 

First suppose that $r \not \equiv s$ (mod $2$). Since $\epsilon(d) \equiv d$ (mod $2$), (\ref{eqn: r, s}) implies that $k = 3r \in \{0, 3\}$ and $2(r+s) = 3 \pm 1 \in \{2, 4\}$. But $r + s$ is odd by assumption, so $r + s = 1$. Since $s \geq 1$ it follows that $r = 0, s = 1$. Thus under our assumptions, when $r \not \equiv s$ (mod $2$), $\widehat b$ is definite if and only if $b = \sigma_1 a_{13} \sigma_2$.

Next suppose that $r \equiv s$ (mod $2$). In this case (\ref{eqn: r, s}) implies that $k \in \{1, 2, 4, 5\}$ and
$$2(r+s) = 1 - \epsilon(s+r+1)$$
Since $r + s$ is even, it follows that $r+s = 0$. But this is impossible as $r \geq 0$ and $s \geq 1$. Thus this case does not arise. 

\begin{remark}
\label{rmk: only s = 1}
{\rm It follows from (2) that the braid $(\sigma_1 a_{13} \sigma_2)^2$ is indefinite. Thus so are the braids considered in (4), (5) and (6) whenever $s > 1$.}
\end{remark}

(4) Suppose that $k = 3r, s> 0$, and $\max\{p_i, q_i, r_i\} \geq 2$. We know that $b$ is indefinite if $s > 1$ by Remark \ref{rmk: only s = 1}, so by Lemma \ref{lemma: minimal form}(2) suppose that $b = \sigma_1^{p_1} a_{13}^{q_1} \sigma_2^{r_1}$ where $\max\{p_1, q_1, r_1\} \geq 2$. Then 
$$b_1(F(b)) = -2 + 6r  + (p_1 + q_1 + r_1)$$ 
On the other hand, $b \sim \delta_3^{3(1 + r)} \sigma_1^{-1} \sigma_2^{p_1-1} \sigma_1^{-1}\sigma_2^{q_1-1} \sigma_1^{-1} \sigma_2^{r_1-1}$ by Proposition \ref{prop: rewrite} so it follows from Proposition \ref{prop: signatures} that 
$$\sigma(\widehat b) = 2 -4r  - (p + q + r)$$ 
Then $\widehat b$ is definite if and only if $0 = b_1(F(b)) + \sigma(\widehat b) = 2(r+1) -2 = 2r$. This occurs if and only if $r = 0$. Equivalently, $b = \sigma_1^{p_1} a_{13}^{q_1} \sigma_2^{r_1}$.

(5) Suppose that $k = 3r + 1, s >0$, and $\max\{p_i, q_i, r_i\} \geq 2$. Since $b$ is indefinite for $s > 1$, one can suppose by Lemma \ref{lemma: minimal form}(2) that $b = \sigma_1^{p_1} a_{13}^{q_1} \sigma_2^{r_1} \sigma_1^{p_2} a_{13}^{q_2}$. As $b \supset \sigma_1^{p_1} a_{13}^{q_1} \sigma_2^{r_1} \sigma_1^{p_2}$, part (5) is a consequence of part (6).

(6) Suppose that $k = 3r + 2, s >0$, and $\max\{p_i, q_i, r_i\} \geq 2$. As in the previous two cases, by Lemma \ref{lemma: minimal form}(2) we can suppose that $b = \sigma_1^{p_1} a_{13}^{q_1} \sigma_2^{r_1} \sigma_1^{p_2}$. In this case  
$$b_1(F(b)) = 6r + 4 + (p_1 + q_1 + r_1 + p_{2})$$ 
On the other hand, $b \sim \delta_3^{3(r + 2)} \sigma_1^{-1} \sigma_2^{p_1-1} \sigma_1^{-1}\sigma_2^{q_1-1} \sigma_1^{-1} \sigma_2^{r_1-1} \sigma_1^{-1}  \sigma_2^{p_{2}-1}$ by Proposition \ref{prop: rewrite} so it follows from Proposition \ref{prop: signatures} that 
$$\sigma(\widehat b) = - 4r - (p_1 + q_1 + r_1 + p_{2})$$ 
Then $\widehat b$ is definite if and only if $0 = b_1(F(b)) + \sigma(\widehat b) = 2(r+2)$. Since $r \geq 0$, this is impossible. Thus $\widehat b$ is indefinite. 
 \end{proof}

 \begin{cor} 
 \label{cor: definiteness, fibredness, and forms} 
 Let $b = \delta_3^k P$ be as above represent a non-trivial, non-split, prime link. Then $b$ is definite if and only if $b$ is conjugate to one of 
 
 $(1)$ $\sigma_1^p a_{13}^q \sigma_2^r$ where $p, q, r \geq 1$. In this case $\widehat b$ is the Montesinos link $M(1; 1/p,1/q,1/r)$ and is not a fibred link. 
  
 $(2)$ $\delta_3 \sigma_1^m$ where $m \geq 1$. In this case $\widehat b$ is the torus link $T(2,m+1)$ and is fibred. Further, $\mathcal{F}(b) \cong A_m$. 
 
 $(3)$ $\delta_3^3 \sigma_1^{m-4}$ where $m \geq 4$. In this case $\widehat b$ is the pretzel link $P(-2,2,m-2)$ and is fibred. Further, $\mathcal{F}(b) \cong D_m$. 
 
 $(4)$ $\delta_3^4$. In this case $\widehat b$ is the torus knot $T(3,4)$ and is fibred. Further, $\mathcal{F}(b) \cong E_6$. 
 
 $(5)$ $\delta_3^4 \sigma_1$. In this case $\widehat b$  is the pretzel link $P(-2,3,4)$ and is fibred. Further, $\mathcal{F}(b) \cong E_7$. 
 
 $(6)$ $\delta_3^5$. In this case $\widehat b$ is the torus knot $T(3,5)$ and is fibred. Further, $\mathcal{F}(b) \cong E_8$. 
 
 \end{cor}
 
 \begin{proof}
It is easy to verify that the closures of the braids listed in the corollary are as stated. 
Further $\widehat b$ is not fibred in case (1) and fibred otherwise (\cite[Theorem 1.1]{Ni2}, \cite[Theorem 3.3]{Stoi}). The symmetrised Seifert forms 
of the associated quasipositive surfaces are listed in \S \ref{subsec: symm seifert forms of baskets}.

To complete the proof of the corollary, we need to show that the corollary gives a complete list of conjugacy class representatives of definite strongly quasipositive 
$3$-braids whose closures are non-trivial, non-split, and prime. 
 
Proposition \ref{prop: when definite} provides a complete list of conjugacy class representatives of definite strongly quasipositive $3$-braids:
 \vspace{-.2cm} 
 \begin{itemize}

\item $\sigma_1^p a_{13}^q \sigma_2^r$ where $p, q, r \geq 1$; 

\vspace{.2cm} \item $\delta_3^k$ where $k \leq 5$; 

\vspace{.2cm} \item $\delta_3^k \sigma_1^p$ where $k \leq 3$ and $p \geq 1$ or $k = 4$ and $p \in \{1, 2\}$; 

\vspace{.2cm} \item $\delta_3 \sigma_1^p a_{13}^q$ where $p, q \geq 1$. 

\end{itemize}
Thus each of the braids listed in the corollary is definite. Suppose that $b$ is one of the definite braids listed in the proposition whose closure is non-trivial, non-split, and prime. Since $\widehat b$ is prime, the case that $b \sim \delta_3 \sigma_1^p a_{13}^q$ is excluded:
$$b \sim \sigma_1^{p+1} \sigma_2^{q+1} \Rightarrow \widehat b = T(p+1, 2) \# T(q+1, 2)$$ 
We also exclude the braids $1$ and $\delta_3$ by the non-split and non-triviality conditions. It remains to examine the braids $\delta_3^2 \sigma_1^p$ where $p \geq 0$ and $\delta_3^4 \sigma_1^2$. But if $b = \delta_3^2 \sigma_1^p$ where $p \geq 0$, then $b = \sigma_1 \sigma_2 \sigma_1 \sigma_2 \sigma_1^p \sim \sigma_1^{p+1} (\sigma_2 \sigma_1 \sigma_2) = \sigma_1^{p+1} (\sigma_1 \sigma_2 \sigma_1) = \sigma_1^{p+2} \sigma_2 \sigma_1 \sim  \sigma_1^{p+3} \sigma_2 =  \sigma_1^{p+2} \delta_3 \sim \delta_3 \sigma_1^{p+2}$, which is listed in (2). And if $b = \delta_3^4 \sigma_1^2$, then $b \sim \sigma_1 \delta_3^4 \sigma_1 = \delta_3^4 a_{13} \sigma_1 = \delta_3^4 \delta_3 = \delta_3^5$, which is (6). 
 \end{proof}
 
 \subsection{The proof of Theorem \ref{thm: definite 3-braids}}
 \label{subsec: proof of thm on 3-braids}
 
Suppose that $L$ is a non-split, non-trivial, strongly quasipositive link of braid index $3$. Stoimenow has shown that strongly quasipositive links of 
braid index $3$ are the closures of strongly quasipositive $3$-braids (\cite[Theorem 1.1]{Stoi}), so without loss of generality we can suppose that 
$L = \widehat b$ where $b = \delta_3^k P$ is a minimal representative of a conjugacy class of strongly quasipositive $3$-braids.

We know that $\widehat b$ is definite if and only if it is one of the links listed in (1) 
through (6) of Corollary \ref{cor: definiteness, fibredness, and forms}, which shows the equivalence of statements (a) and (c) of part (1) of the theorem. 
In each of the six cases $\Sigma_2(\widehat b)$ is an L-space: the links in (1) are alternating (\cite[Proposition 3.3]{OS2}) while 
those in (2), (3), (4), (5), and (6) have $2$-fold branched covers with finite fundamental groups. Conversely, Theorem \ref{thm: bbg definite} shows that 
$b$ is definite if $\Sigma_2(\widehat b)$ is an L-space. Thus (a) and (b) of part (1) of the theorem are equivalent, which completes the proof of (1). 

Positive braids are fibred \cite{Sta}, which is one direction of (2). For the other direction, suppose that $\widehat b$ is fibred. 
Then the minimal representative of $b$ has $k \geq 1$ (\cite[Theorem 1.1]{Ni2}, \cite[Theorem 3.3]{Stoi}). If it is also definite, 
Proposition \ref{prop: when definite} implies that  it is conjugate to a braid of the form $\delta_3^k \sigma_1^{m}$, for some 
$k \geq 1$ and $m \geq 0$, or $\delta_3  \sigma_1^{p_1} a_{13}^{q_1} =  \sigma_2^{p_1} \sigma_1^{q_1} \delta_3$, each of 
which is a positive braid. This proves (2).

\section{Seifert forms of basket links}
\label{sec: seifert forms of baskets}

Suppose that $b = \delta_n P$ is a BKL-positive word. Each letter $a_{rs}$ of $P$ determines a primitive element $\alpha$ of $H_1(F(b))$ as depicted in Figure \ref{fig: alphars}. 

\begin{figure}[!ht]
\centering
 \includegraphics[scale=0.7]{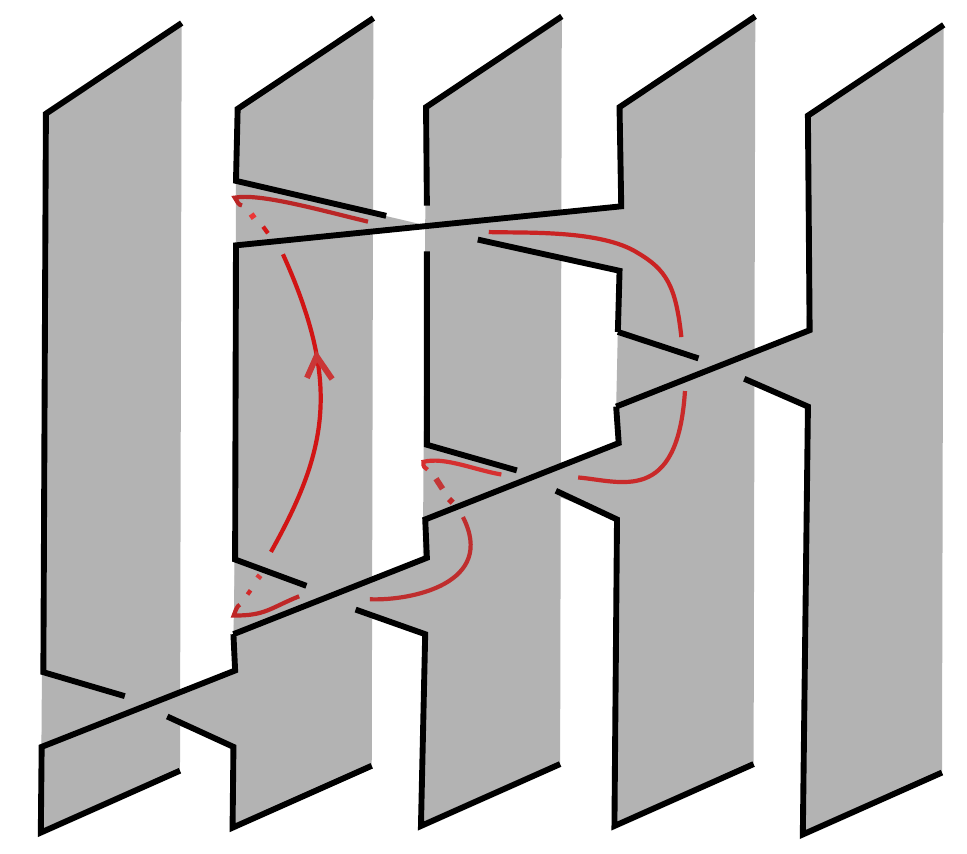} 
\caption{The class associated to $a_{rs}$} 
\label{fig: alphars}
\end{figure} 
Here, a representative cycle for $\alpha$ consists of an arc which passes over the band in $F(b)$ corresponding to $a_{rs}$ and then descends to pass over the bands corresponding to the letters $\sigma_{r}, \sigma_{r+1}, \ldots, \sigma_{s-1}$ of the initial $\delta_n$ factor of $b$. 
The set of all such $\alpha$ forms a basis for $H_1(F(b))$. 

\begin{remark}  
\label{rem: alpha to alpha'}
{\rm Recall that $\delta_n b \delta_n^{-1} = \delta_n P'$ where a letter $a_{rs}$ of $P$ is converted to $a_{r+1, s+1}$ if $s \leq n-1$ and $a_{1, r+1}$ if $s = n$ (cf. Identity (\ref{conj by delta})). We noted in Remark \ref{rem: delta rotn} that the geometric braid $b' = \delta_n P'$ can be obtained by rotating the geometric braid of $\delta_n P$ through an angle of $\frac{2 \pi}{n}$. It's easy to see that $F(b')$ can also be obtained from $F(b)$ by a rotation of $\frac{2 \pi}{n}$. This rotation takes the class $\alpha$ associated to $a_{rs}$ described above to the class associated to $a_{r+1, s+1}$ if $s \leq n-1$ and the {\it negative} of the class associated to $a_{1, r+1}$ if $s = n$.
}
\end{remark}

We are interested in the Seifert form on $H_1(F(b))$ and as such we need to calculate the linking numbers $lk(\alpha^+, \beta)$ where $\alpha, \beta \in H_1(F(b))$ are classes corresponding to letters of $P$. Here, $\alpha^+$ is represented by the cycle obtained by pushing a representative cycle for $\alpha$ to the positive side of  $F(b)$. Given our orientation convention for braid closures (cf. Figure \ref{fig: braid conventions}), the side of $F(b)$ which is shaded in Figure \ref{fig: alphars} is its negative side.

\begin{lemma} 
\label{lemma: linking info} 
Suppose that $b = \delta_n P$ is a BKL-positive word and that $P = a_1 a_2 \ldots a_m$ where each $a_i$ is one of the BKL generators. Let 
$\alpha_1, \alpha_2, \ldots, \alpha_m$ be the associated basis elements of $H_1(F(b))$. Fix $1 \leq i, j \leq m$ and suppose that $\alpha_i$ corresponds to the 
letter $a_{rs}$ while $\alpha_j$ corresponds to $a_{tu}$. Then
$$lk(\alpha_i^+, \alpha_j) = \left\{
\begin{array}{rl}
-1 & \hbox{if either } i = j  \mbox{ or }   i > j \mbox{ and } r \leq  t < s \leq u \\ 
1 & \hbox{if }  i > j \mbox{ and } t < r \leq u < s \\ 
 0 & \hbox{otherwise}  
\end{array} \right.$$
\end{lemma} 

\begin{proof}
If $r < s < t < u$, consideration of the cycles representing $\alpha_i$ and $\alpha_j$ shows that they are separated by a $2$-sphere in $S^3$. Thus $lk(\alpha_i^+, \alpha_j) = 0$. It then follows that the same holds whenever $t < r < s < u$, or $t < u < r < s$, or $r < t < u < s$. Indeed, in each of these cases we can simultaneously conjugate $a_{rs}$ and $a_{tu}$ by a power of $\delta_n$ to obtain new letters $a_{r's'}$ and $a_{t'u'}$ where $r' < s' < t' < u'$. Remark \ref{rem: alpha to alpha'} shows that this conjugation changes linking numbers at most up to sign, so $lk(\alpha_i^+, \alpha_j) = 0$. 

Next, if $t < u = r < s$ the cycles representing $\alpha_i^+$ and $\alpha_j$ are separated by a $2$-sphere; this is obvious if $i < j$ and is easy to see by considering the cycles representing $\alpha_i$ and $\alpha_j$ if $j < i$. Thus $lk(\alpha_i^+, \alpha_j) = 0$. Remark \ref{rem: alpha to alpha'} then implies, as in the previous paragraph, that $lk(\alpha_i^+, \alpha_j) = 0$ whenever $r < t < s = u$ or $r = t < s < u < s$. 

In the case that $r < s = t < u$, it is easy to see that 
$$lk(\alpha_i^+, \alpha_j) = 
\left\{ \begin{array}{ll} 0 & \mbox{ if } i < j \\ 1 & \mbox{ if } j < i \end{array} \right.$$
Remark \ref{rem: alpha to alpha'} then implies that 
$$lk(\alpha_i^+, \alpha_j) =  \left\{ \begin{array}{rl} 0 & \mbox{ if } i < j \\ -1 & \mbox{ if } j < i \end{array} \right.$$
when $r < t < s = u$ or $r = t < s < u < s$ then  

The final three cases to consider are when $i = j$, or $i \ne j$ and $r < t < s < u$, or $i \ne j$ and $t < r < s < u$. Before dealing with these cases we need to make an observation. 

Let $D_i$ be the $2$-disk in $F(b)$ with boundary the $i^{th}$ component of the trivial $n$-braid and suppose that $b$ contains a word $a_{r_0 r_1} a_{r_1 r_2} \cdots a_{r_{k-1} r_k}$. Let $A$ be a smooth arc in the interior of $F(b)$ obtained by concatenating a core of the half-twisted band in $F(b)$ corresponding to $a_{r_0 r_1}$, an arc properly embedded in $D_{r_1}$, a core of the band corresponding to $a_{r_1 r_2}$, an arc properly embedded in $D_{r_2}$, etc., ending with a core of the $1$-handle corresponding to $a_{r_{k-1} r_k}$. Thinking of $A$ as a properly embedded arc in the union $X$ of $D_{r_0}, D_{r_1}, \ldots, D_{r_k}$ and the bands corresponding to $a_{r_0 r_1}, a_{r_1 r_2}, \ldots a_{r_{k-1} r_k}$, the reader will verify that $A$ has a tubular neighbourhood in $X$ which is isotopic (rel $X \cap (D_{r_0} \cup D_{r_k}$)) to a half-twisted band corresponding to $a_{r_0 r_k}$. It follows that the boundary of a tubular neighbourhood of the cycle corresponding to $a_{rs}$ is a Hopf band whose components have linking number is $-1$ when they are like-oriented. Hence, 
$$lk(\alpha_i^+, \alpha_i) = -1,$$
which is the case $i = j$.

Suppose that $i \ne j$ and $r < t < s < u$. From the previous paragraph we see that if we replace $a_{tu}$ by $a_{ts} a_{su}$ and let $\alpha_j(1), \alpha_j(2)$ correspond to $a_{ts}$ and $a_{su}$ respectively, then 
$$lk(\alpha_i^+, \alpha_j) = lk(\alpha_i^+, \alpha_j(1)) + lk(\alpha_i^+, \alpha_j(2))$$    
Consideration of the cases $r < t < s = u$ and $r < s = t < u$, which were handled above, implies that $lk(\alpha_i^+, \alpha_j(1)) = 0$ is zero while $lk(\alpha_i^+, \alpha_j(2)) = 
\left\{ \begin{array}{ll} 0 & \mbox{ if } i < j \\ 1 & \mbox{ if } j < i \end{array} \right.$. Thus 
$$lk(\alpha_i^+, \alpha_j) = 
\left\{ \begin{array}{ll} 0 & \mbox{ if } i < j \\ 1 & \mbox{ if } j < i \end{array} \right.$$
Finally, an application of Remark \ref{rem: alpha to alpha'} shows that if $i \ne j$ and $t < r < s < u$, then 
$$lk(\alpha_i^+, \alpha_j) = 
\left\{ \begin{array}{rl} 0 & \mbox{ if } i < j \\ -1 & \mbox{ if } j < i \end{array} \right.$$
\end{proof} 

\begin{remark} 
\label{rem: root lattice}
{\rm The lemma shows that $H_1(F(b))$ is generated by elements $\alpha$ for which $\mathcal{F}(b)(\alpha, \alpha) = -2$. Hence, when $\mathcal{F}(b)$ is definite, it is a root lattice.}
\end{remark}

\section{Some indefinite BKL-positive words}
\label{sec: indefinite bkl pos words}

\begin{lemma} 
\label{lemma: bound on spans} 
Suppose that $b = \delta_n^k P$ is a BKL-positive $n$-braid.

$(1)$ If $P$ contains a letter of span $l$ where $4 \leq l \leq n-4$, then $b$ is indefinite. 

$(2)$ If $P$ contains a letter of span $3$ or $n-3$, then 
\vspace{-.2cm} 
\begin{itemize}
\item[{\rm (a)}] $b$ is indefinite if $n \geq 9$.

\vspace{.2cm} \item[{\rm (b)}] $\mathcal{F}(b)$ contains a primitive $E_n$ sublattice for $n = 6, 7, 8$.

\end{itemize}

$(3)$ If $P$ contains the square of a letter of span $2$ or $n-2$, then 
\vspace{-.2cm} 
\begin{itemize}
\item[{\rm (a)}] $b$ is indefinite if $n \geq 8$.

\vspace{.2cm} \item[{\rm (b)}] $\mathcal{F}(b)$ contains a primitive $E_{n+1}$ sublattice for $n = 5, 6, 7$.

\end{itemize}

\end{lemma}

\begin{proof}
After conjugating by an appropriate power of $\delta_n$ we can suppose that $b \supseteq \delta_n^2 a_{1, l+1}$. Figure \ref{fig: span3} depicts a schematic of the braid $\delta_n^2 a_{1, l+1}$.

\begin{figure}[!ht] 
\centering
 \includegraphics[scale=0.7]{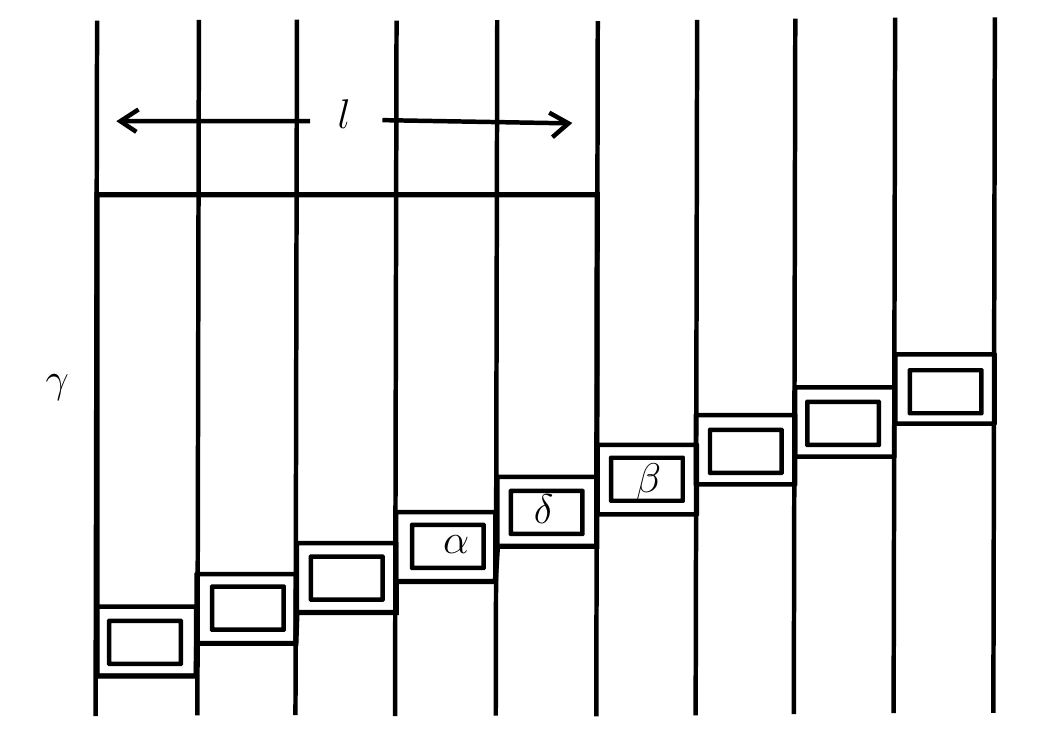} 
\caption{A schematic of the braid $\delta_n^2 a_{1, l+1}$} 
\label{fig: span3}
\end{figure} 
Lemma \ref{lemma: linking info} shows that if $\eta, \zeta$ are any two of the classes shown in Figure \ref{fig: span3} then $\mathcal{F}(\eta, \zeta) = -2$ if $\eta = \zeta$ and is $0$ or $\pm 1$ otherwise. In fact, the classes define a sublattice of $\mathcal{F}(b)$ corresponding to the tree in Figure \ref{fig: e6 tree}. 
\begin{figure}[!ht] 
\centering
  \includegraphics[scale=0.7]{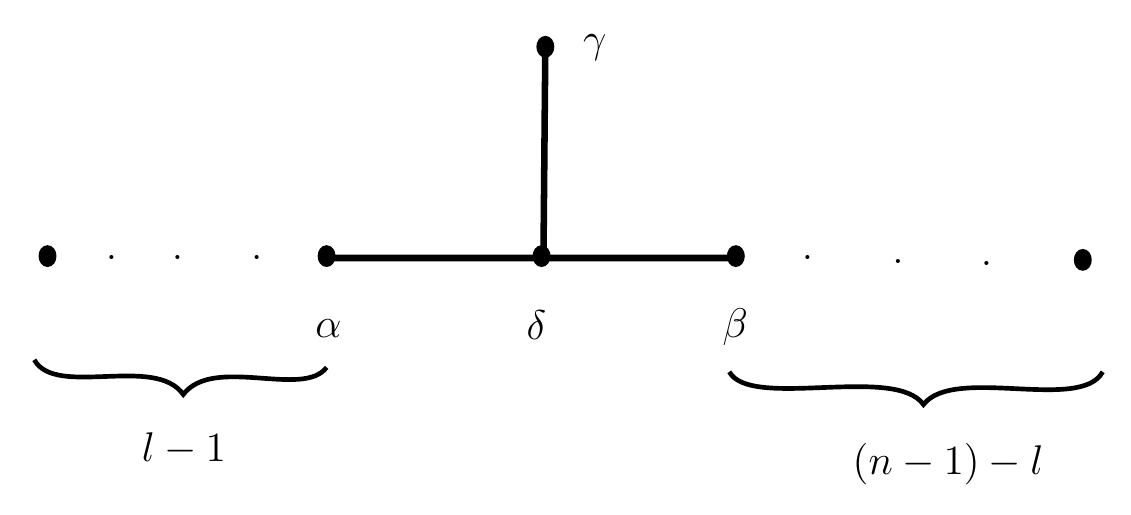}
\caption{} 
\label{fig: e6 tree}
\end{figure} 
The only trees with vertices of weight $-2$ and edges of weight $\pm 1$ whose associated forms are definite correspond to the Dynkin diagrams associated to $A_m, D_m, E_6, E_7$, or $E_8$ (see \cite[pages 61-62]{HNK}). It follows that $\mathcal{F}(b)$ is indefinite if $4 \leq l \leq n-4$ or $l \in \{3, n-3\}$ and $n \geq 9$. If $l \in \{3, n-3\}$ where $n = 6,7$, or $8$, the tree corresponds to a primitive $E_n$ sublattice of $\mathcal{F}(b)$. This completes the proof of parts (1) and (2) of the lemma.  

A similar argument deals with part (3) where we can suppose that $P$ contains $a_{13}^2$. The graph associated to the classes depicted in Figure \ref{fig: square}   
\begin{figure}[!ht] 
\centering
 \includegraphics[scale=0.7]{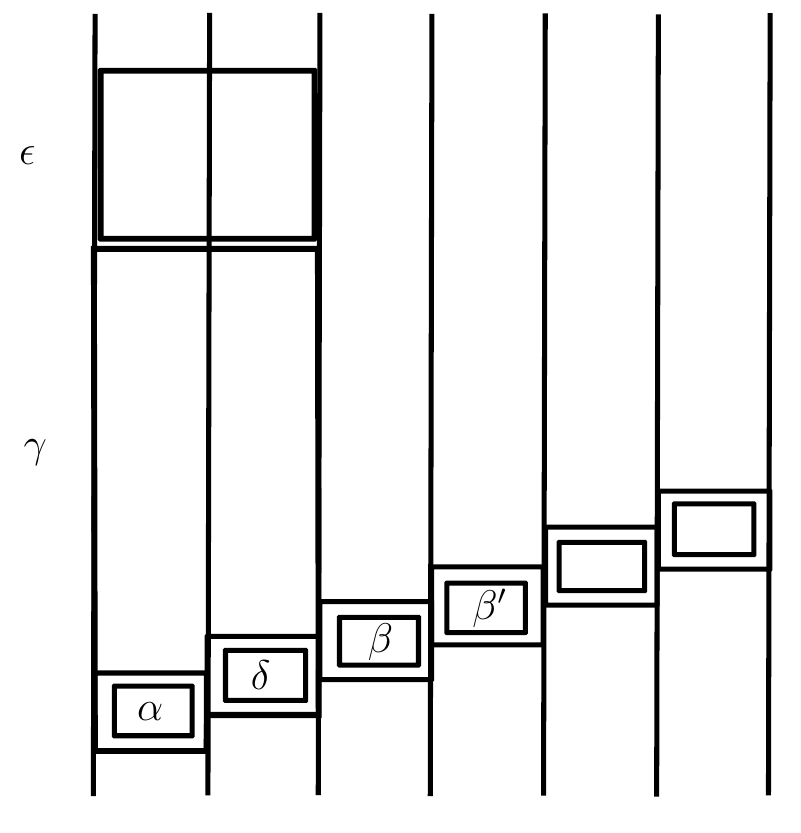} 
\caption{A schematic of the braid $\delta_n^2 a_{13}^2$} 
\label{fig: square}
\end{figure} 
determines a sublattice of $\mathcal{F}(b)$ corresponding to the tree in Figure \ref{fig: square tree}. 
\begin{figure}[!ht] 
\centering
  \includegraphics[scale=0.7]{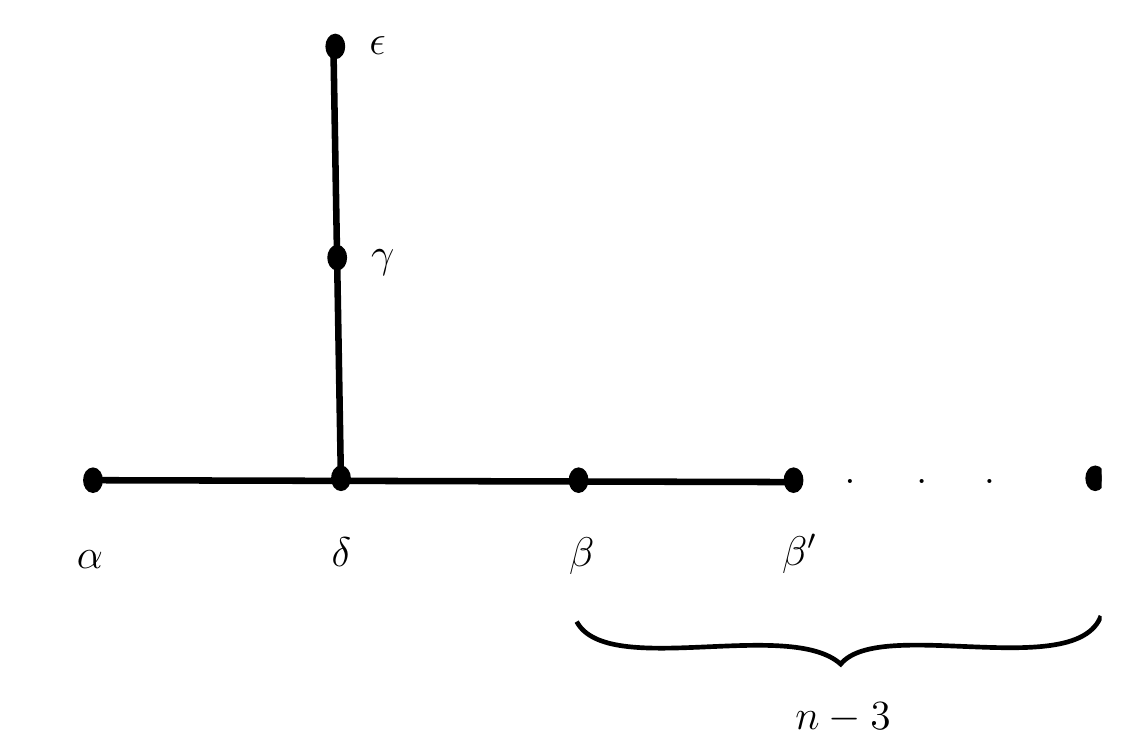}
\caption{} 
\label{fig: square tree}
\end{figure} 
It follows that $b$ is indefinite if $n \geq 8$ and that $\mathcal{F}(b)$ contains a primitive $E_{n+1}$ sublattice for $n = 5, 6, 7$.
\end{proof}

Recall that two BKL generators $a_{rs}$ and $a_{tu}$ are {\it linked} if either $r < t < s < u$ or $t < r < u < s$. 

\begin{lemma} 
\label{lemma: no linking} 
Suppose that $b = \delta_n^2 P$ is a BKL-positive $n$-braid where $P$ contains a pair of linked letters. Then $b$ is indefinite. 
\end{lemma}

\begin{proof}
Suppose that $P$ contains a product $a_{rs} a_{tu}$ of linked letters. We can assume that $r < t < s < u$. Otherwise we conjugate $b$ by $\delta_n^{n-u+1}$. 

Let $\alpha, \beta_1, \beta_2, \ldots, \beta_{u-s}, \gamma, \theta$ be the elements of $H_1(F(b))$ corresponding respectively to the letters $\sigma_{s-1}, \sigma_s, \sigma_{s+1}, \ldots, \sigma_{u-1}$ in the second $\delta_n$, and $a_{rs}$ and $a_{st}$. (See Figure \ref{fig: alphars}.) Set $\beta = \beta_1 + \cdots + \beta_{u-s}$

From Lemma \ref{lemma: linking info} the restriction of the Seifert form of $F(b)$ to the span of $\alpha, \beta, \gamma, \theta$ has matrix 
$$\left(\begin{matrix} 
-1 & \;\;\; 0 & \;\;\; 0 & \;\;\; 0 \\
\;\;\; 1 & -1 & \;\;\; 0 & \;\;\; 0 \\
-1 & \;\;\; 0 & -1 & \;\;\; 0 \\
\;\;\; 0 & -1 & \;\;\; 1 & -1
\end{matrix} \right) $$
Hence $\mathcal{F}(b)|_{\langle \alpha, \beta, \gamma, \theta \rangle}$ is represented by the matrix 
$$\left(\begin{matrix} 
-2 & \;\;\; 1 & -1 & \;\;\; 0 \\
\;\;\; 1 & -2 & \;\;\; 0 & -1 \\
-1 & \;\;\; 0 & -2 & \;\;\; 1 \\
\;\;\; 0 & -1 & \;\;\; 1 & -2
\end{matrix} \right)$$
If $c_i$ is the $i^{th}$ column of this matrix, then $c_1 + c_2 - c_3 - c_4 = 0$, so $\mathcal{F}(b)$ indefinite.
\end{proof} 

\begin{lemma}
\label{lemma: when commuting implies indefinite} 
Suppose that $b = \delta_n^k P$ is a BKL-positive braid where $P$ contains a pair of commuting letters $a_{rs}, a_{tu}$ each of whose spans are bounded between $2$ and $n-2$. Then $b$ is indefinite. 
\end{lemma}

\begin{proof}
Up to conjugation by a power of $\delta_n$, we can suppose that $1 \leq r < s < t < u \leq n$. Consider Figure \ref{fig: commute} which represents a braid contained in $\delta_n^2P$.
\begin{figure}[!ht] 
\centering
  \includegraphics[scale=0.7]{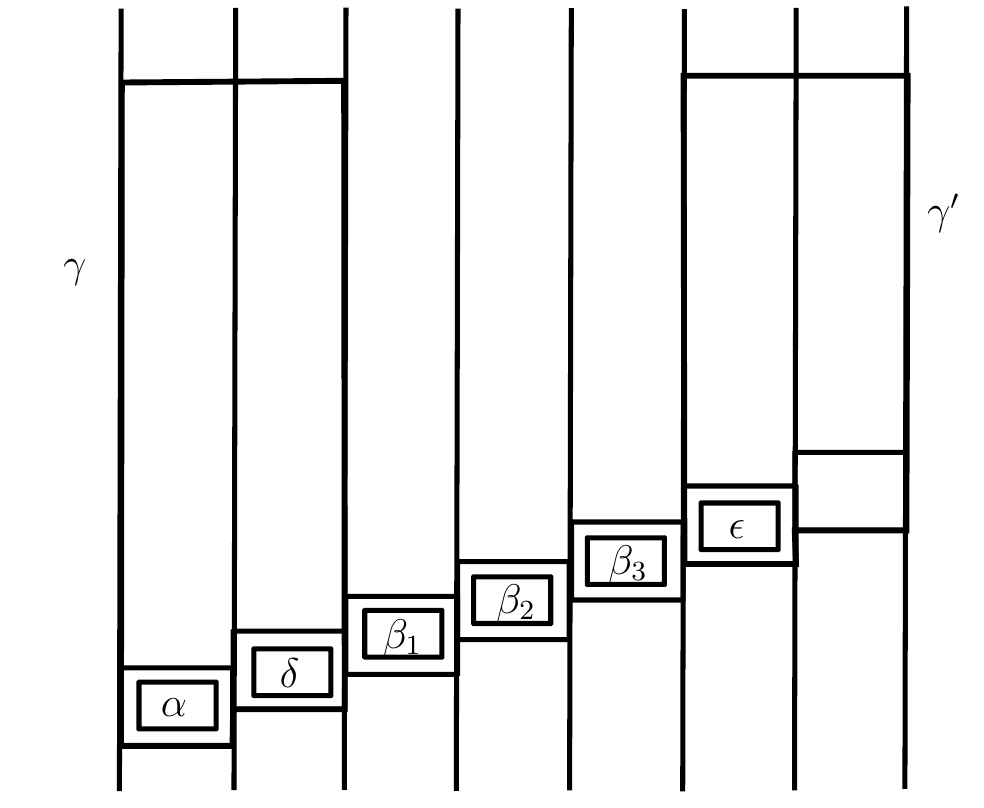}
\caption{} 
\label{fig: commute}
\end{figure} 
The associated graph is a tree with two vertices of valency $3$, as depicted in Figure \ref{fig: classes for 2 vertices valency 3 tree}, 
\begin{figure}[!ht] 
\centering
   \includegraphics[scale=0.7]{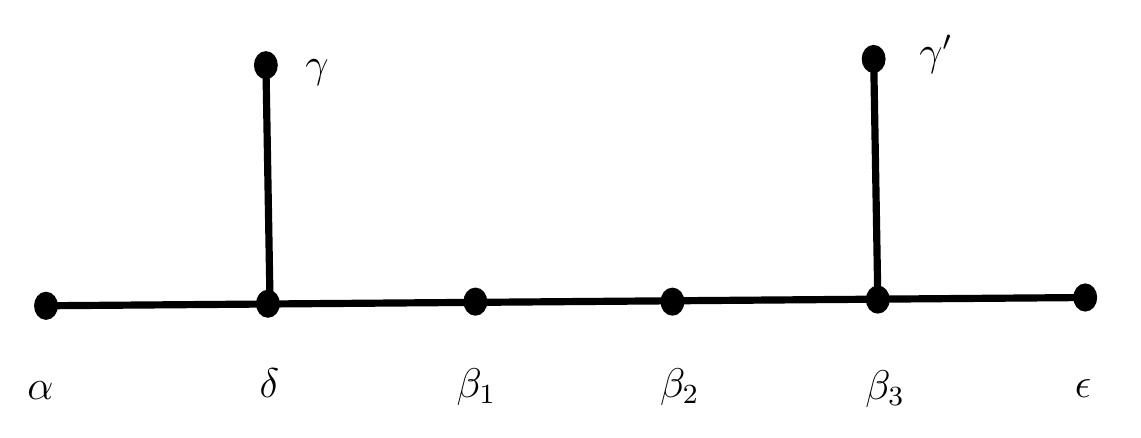}
\caption{} 
\label{fig: classes for 2 vertices valency 3 tree}
\end{figure}
which implies that $\mathcal{F}(b)$ is indefinite. 
\end{proof}

A {\it $2$-step descending staircase} is a product 
$$a_{st} a_{rs}$$
where $1 \leq r < s < t \leq n$. The letters $a_{rs}$ and $a_{st}$ are the {\it steps} of the staircase. 

\begin{lemma} 
 \label{lemma: 2-step descending stairs of length two or more indefinite} 
Suppose that $b = \delta_n^k P$ is a BKL-positive braid where $P$ contains a descending $2$-step staircase $a_{st} a_{rs}$ whose steps are of span $2$ or more. Then $b$ is indefinite.
\end{lemma}

\begin{proof}
Define classes in $H_1(F(b))$ as follows
\vspace{-.2cm} 
\begin{itemize}

\item $\theta_1$ corresponds to the class associated to $\sigma_{s-2}$;

\vspace{.2cm} \item $\theta_0$ corresponds to the class associated to $\sigma_{s-1}$; 

\vspace{.2cm} \item $\theta_2$ corresponds to the class associated to $\sigma_{s}$; 

\vspace{.2cm} \item $\theta_3$ corresponds to the negative of the class associated to $a_{rs}$; 

\vspace{.2cm} \item $\theta_4$ corresponds to the class associated to $a_{st}$;

\end{itemize}
It follows from Lemma \ref{lemma: linking info} that $\mathcal{F}(b)(\theta_i, \theta_i) = -2$ for each $i$, $\mathcal{F}(b)(\theta_i, \theta_j) \in \{0, 1\}$ for $i \ne j$, and the associated graph is a tree with a vertex of valency $4$. Hence, the restriction of $\mathcal{F}(b)$ to it is indefinite.
\end{proof}

\begin{lemma}
\label{lemma: contains sigma1} 
If $b = \delta_n^2 P$ is a BKL-positive definite $n$-braid and $n(\widehat b) = n$, then $\sigma_1$ is contained in some conjugate of $P$ by a power of $\delta_n$.
\end{lemma}

\begin{proof}
Lemma \ref{lemma: not minimal} implies that each conjugate of $P$ by some $\delta_n^k$ contains a letter of the form $a_{1s'}$. Let $s$ be the largest such $s'$. 
Without loss of generality we can suppose that $P \supseteq a_{1s}$. If $s < n$, Lemma \ref{lemma: not minimal} implies that $P$ also contains a letter of the 
form $a_{rn}$. Then Lemma \ref{lemma: no linking} shows that $s \leq r$. But then $\delta_n P \delta_n^{-1} \supseteq \delta_n a_{rn} \delta_n^{-1} = a_{1, r+1}$, 
which is a contradiction since $s \leq r < r+1 \leq n$. Thus $s = n$ so that $P \supseteq a_{1n}$. Then  $\delta_n P \delta_n^{-1} \supseteq \sigma_1$, which is what we set out to prove.
\end{proof} 

\section{BKL-positive braids with forms $E_6, E_7$, and $E_8$} 
\label{sec: e6, e7, e8}

In this section we identify certain situations under which the symmetrised Seifert form of a basket is either $E_6, E_7$, or $E_8$. In this case,
\begin{equation} 
\label{eqn: l(P) bound} 
(n - 1) + l_{BKL}(P) = \mbox{{\rm rank}} \; H_1(F(b)) = \left\{ 
\begin{array}{ll}
6 & \mbox{if } \mathcal{F}(b) \cong E_6 \\ 
7 & \mbox{if } \mathcal{F}(b) \cong E_7 \\ 
8 & \mbox{if } \mathcal{F}(b) \cong E_8 
\end{array} \right.
\end{equation}

\subsection{Definite extensions of $E_6, E_7$, and $E_8$} 

For convenience we use the positive definite versions of $E_6, E_7$, and $E_8$ in the next lemma. 

\begin{lemma} \label{lemma: e* case}
Suppose that $A^* = \left(\begin{matrix}   A & R^T  \\ 
R & 2  \end{matrix}\right)$ is a symmetric bilinear positive definite form on $\mathbb Z^{d}$ where $R = (0, 0, \ldots, 0, 1)$. Then,

$(1)$ $A \cong E_6$ implies that $A^* \cong E_7$.

$(2)$ $A \cong E_7$ implies that $A^* \cong E_8$.

$(3)$ $A \not \cong E_8$.
\end{lemma}

\begin{proof}
We remark that $E_8$ is the unique even symmetric, bilinear positive definite unimodular lattice of rank $8$ or less; $E_7$ is the unique non-diagonalisable symmetric, bilinear positive definite lattice of determinant $2$ and rank $7$ or less; and $E_6$ is the unique non-diagonalisable symmetric, bilinear positive definite lattice of determinant $3$ and rank $6$ or less. See the last paragraph of the introduction of \cite{Gri} for a discussion of these statements and further references. 

Write $A = \left(\begin{matrix}   B & *  \\ * & *  \end{matrix}\right)$ where $B$ is a symmetric bilinear form on $\mathbb Z^{d-2}$ and observe that as $A^*$ is positive definite, $0 < \det(B)$ and $0 < \det(A^*)$. On the other hand, 
$$\det(A^*)= 2 \det(A) - \det(B)$$ 
and therefore 
$$0 < \det(B) < 2\det(A)$$ 

If $A \cong E_6$, then $\det(A) = 3$, so that $0 < \det(B) < 6$. On the other hand, $B$ is an even form so it can be written $B = C + C^T$ where $C$ is a $5 \times 5$ matrix with integer coefficients. Then as $C - C^T$ is skew-symmetric, $0 = \det(C - C^T) \equiv \det(C + C^T) \hbox{ {\rm(mod $2$)}} = \det(B)$. Thus $\det(B) \in \{2, 4\}$. If $\det(B) = 2$, then $B$ is a definite diagonal form (see the first paragraph of the proof) and hence has determinant at least $32$, a contradiction. Thus $\det(B) = 4$ and therefore the determinant of the even form $A^*$ is $2 \det(A) - \det(B) = 2$. Another application of the remarks of the first paragraph implies that $A^* \cong E_7$.

If $A \cong E_7$ then $\det(A) = 2$, so that $0 < \det(B) < 4$. On the other hand, $B$ is an even form of rank $6$ and as there are no even unimodular lattices of rank less than $8$, $\det(B) \in \{2, 3\}$. If $B$ was a diagonalisable form its determinant would be at least $64$, thus by the remarks in the first paragraph it must be congruent to $E_6$ and therefore have determinant $3$. It follows that $\det(A^*) = 1$ and therefore $A^*$ is an even symmetric unimodular form on $\mathbb Z^8$. Consequently, $A^* \cong E_8$. 

Finally suppose that $A \cong E_8$. Then $\det(A) = 1$ so that $0 < \det(B) < 2$. Thus $\det(B) = 1$, implying that $B$ is an even symmetric unimodular form of rank $7$, a contradiction. This completes the proof. 
\end{proof}

\begin{lemma} 
\label{lemma: e_n}
Let $b_0 \subset b = \delta_n P$ be BKL-positive $n$-braids where $l_{BKL}(b) - l_{BKL}(b_0) = 1$ and $F(b_0)$ is connected. Let $a_{rs}$ be the letter of $b$ not contained in $b_0$ and suppose that $\alpha \in H_1(F(b))$ corresponds to $a_{rs}$. Suppose as well that there is a class 
$\beta \in H_1(F(b_0))$ such that $\mathcal{F}(b)(\alpha, \beta) = \pm1$. 

$(1)$ If $\mathcal{F}(b_0) \cong E_6$ and $b$ is definite, then $\mathcal{F}(b) \cong E_7$.

$(2)$ If $\mathcal{F}(b_0) \cong E_7$ and $b$ is definite, then $\mathcal{F}(b) \cong E_8$.

$(3)$ If $\mathcal{F}(b_0) \cong E_8$, then $b$ is indefinte.

\end{lemma}

\begin{proof}
Up to replacing $\alpha$ by its negative, we can suppose that $\mathcal{F}(b)(\alpha, \beta) = -1$. Since $\beta$ is primitive, there is a basis $\{\gamma_1, \gamma_2, \ldots, \gamma_r\}$ of $H_1(F(b_0))$ where $\gamma_r = \beta$. After replacing $\gamma_i$ by $\gamma_i + \mathcal{F}(b)(\gamma_i, \alpha) \gamma_r$ for $1 \leq i \leq r-1$ we can suppose that 
$$\mathcal{F}(b)(\gamma_i, \alpha) = \left\{ \begin{array}{rl} 0 & \mbox{ if } 1 \leq i \leq r -1 \\ -1 & \mbox{ if } i = r \end{array} \right.$$
The lemma now follows from Lemma \ref{lemma: e* case} after replacing $\mathcal{F}(b)$ by $-\mathcal{F}(b)$.  
\end{proof}

\begin{cor}
\label{cor: contains deltan2}
Suppose that $b = \delta_n^2 P \in B_n$ is a definite BKL-positive braid and $P \supset P_0$ such that $\mathcal{F}(\delta_n^2 P_0) \cong E_6, E_7$ or $E_8$. Then $\mathcal{F}(b) \cong E_6, E_7$ or $E_8$. 
\end{cor}

\begin{proof}
There is nothing to prove if $P_0 = P$ so suppose otherwise and let $P_1 \subseteq P$ be obtained from $P_0$ by reinserting a letter $a_{rs}$ of $P \setminus P_0$. Denote by $\alpha$ the associated class in $H_1(F(\delta_n^2 P_1))$. For each $1 \leq i \leq n-1$, $H_1(F(\delta_n^2 P_0))$ contains the class associated to the letter $\sigma_i$ of the second $\delta_n$ factor of $b = \delta_n^2 P$ and some such class $\beta$ satisfies $\mathcal{F}(b)(\alpha, \beta) = \pm1$ by Lemma \ref{lemma: linking info}. Lemma \ref{lemma: e_n} then implies that $\mathcal{F}(\delta_n^2 P_1) \cong E_7$, or $E_8$. The conclusion of the lemma now follows by induction on the number of elements in $P \setminus P_0$. 
\end{proof}

\subsection{Some BKL-positive braids with form $E_6, E_7$, or $E_8$ }
\label{subsec: other e configs}

\begin{lemma}
\label{lemma: small span}
Suppose that  $n \geq 6$ and that $b = \delta_n^2 P$ is a BKL-positive definite $n$-braid. If $P$ contains a letter of span $3$ or $n-3$, then $n \leq 8$ and $\mathcal{F}(b) \cong E_6, E_7$, or $E_8$.
\end{lemma} 

\begin{proof}
Conjugating by a suitable power of $\delta_n$ we may assume that $P$ contains $a_{14}$. According to Lemma \ref{lemma: bound on spans}(2), $n \leq 8$ and $\mathcal{F}(b)$ contains an $E_m$ sublattice for $m = 6, 7$, or $8$. Inspection of the proof of this lemma shows that $P$ contains a subword $P_0$ such that $\mathcal{F}(\delta_n^2 P_0) \cong E_{m}$. The conclusion of the lemma then follows from Corollary \ref{cor: contains deltan2}. 
\end{proof}

\begin{lemma} 
\label{lemma: no repeated span 2 letters}
Suppose that $n \geq 5$ and $b = \delta_n^2 P$ is a definite BKL-positive $n$-braid. If $P$ contains the square of a letter of span $2$ or $n-2$, then $n \leq 7$ and $\mathcal{F}(b) \cong E_6, E_7$, or $E_8$. 
\end{lemma}

\begin{proof}
Conjugating by a suitable power of $\delta_n$ we may assume that $P$ contains $a_{13}^2$. Then by Lemma \ref{lemma: bound on spans}(3), $n \leq 7$ and $\mathcal{F}(b)$ contains an $E_m$ sublattice for $m = 6, 7$, or $8$. Inspection of the proof of this lemma shows that $P$ contains a subword $P_0$ such that $\mathcal{F}(\delta_n^2 P_0) \cong E_{m}$. The conclusion of the lemma now follows from Corollary \ref{cor: contains deltan2}. 
\end{proof}

An {\it ascending} $(2, 2)$ pair is a BKL-positive word of the form $a_{r, r+2} a_{r+2, r+4}$.

\begin{lemma} 
\label{lemma: no ascending stairs of total length 3 or 4}
Suppose that $n \geq 5$ and $b = \delta_n^2 P$ is a definite BKL-positive $n$-braid. If $P$ contains an ascending $(2,2)$ pair, then $n \leq 7$ and $\mathcal{F}(b) \cong E_6, E_7$, or $E_8$.
\end{lemma} 

\begin{proof}
Without loss of generality, we can suppose that $P$ contains $a_{13} a_{35}$. Since
$$\delta_n^2 a_{13} a_{35} \sim  a_{35} \delta_n^2 a_{13} =  \delta_n^2 a_{13}^2,$$
we see that up to conjugation, $b = \delta_n^2 P$ where $P$ contains the square of a letter of span $2$. An appeal to Lemma \ref{lemma: no repeated span 2 letters} then 
completes the proof.
\end{proof} 

We call a product  
$$a_{r_0 r_1} a_{r_1 r_2} \ldots a_{r_{k-1} r_k}$$ 
where $1 \leq r_0 < r_1 < r_2 < \ldots < r_k \leq n$ a {\it $k$-step ascending staircase} and the letters $a_{r_{i-1} r_i}$ are the {\it steps} of the staircase. It follows from (\ref{eqn: non-commuting generators}) that for each $1 \leq i < j \leq k$, 
$$a_{r_{i-1} r_i} a_{r_i r_{i+1}} \ldots a_{r_{j-1} r_j} = a_{r_{i-1} r_j} a_{r_{i-1} r_{j-1}}  \ldots a_{r_{i-1} r_i}$$
so that up to rewriting, $P$ contains a letter of span $r_j - r_{i-1} \geq j - i$ for each $1 \leq i < j \leq k$. 

\begin{lemma} 
\label{lemma: $3$-step ascending staircase e6 definite}
Suppose that $b = \delta_n^2 P$ is a BKL-positive definite $n$-braid. If $P$ contains a $3$-step ascending staircase $a_{r_0 r_1} a_{r_1 r_2} a_{r_2 r_3}$, then $\mathcal{F}(b) \cong E_6, E_7$, or $E_8$. 
\end{lemma}

\begin{proof}
The existence of a $3$-step ascending staircase implies that $n \geq 4$ and if $n = 4$, the staircase must be $\sigma_1 \sigma_2 \sigma_3 = \delta_4$. In this case $b = \delta_4^2 P \supseteq \delta_4^2 \delta_4 = \delta_4^3$ whose closure is $T(3,4)$. Thus $\mathcal{F}(b)$ contains $\mathcal{F}(T(3,4)) \cong E_6$. The proof now follows from Corollary \ref{cor: contains deltan2}.   

A similar argument shows that the lemma holds for $n \geq 5$ if each of the steps in the staircase has span $1$: Up to conjugation by a power of $\delta_n$, $P$ contains $P_0 = \sigma_1 \sigma_2 \sigma_3 = \delta_4$ so that $b = \delta_n^2 P \supseteq \delta_4^3$. 
Then $\mathcal{F}(b) \supset \mathcal{F}(\delta_4^3) \cong E_6$. Let $F_0$ be the quasipositive surface of the $4$-braid $\delta_4^3$. The reader will verify that $F_1 = F(\delta_n \delta_4^2) \cong F_0$ and the inclusion $F_0 \to F_1$ is a homotopy equivalence. Then $\mathcal{F}(\delta_n \delta_4^2) \cong E_6$. Adding the twisted band associated to the letter $\sigma_4$ in the second $\delta_n$ factor of $b$ to $F_1$ yields the surface $F(\delta_n \delta_5 \delta_4)$. Let $\alpha \in H_1(F(\delta_n \delta_5 \delta_4))$ be the class associated to this band and $\beta$ the class associated to the $\sigma_3$ in the second $\delta_n$ factor of $b$. Then $\mathcal{F}(b)(\alpha, \beta) = \pm 1$ by Lemma \ref{lemma: linking info}. Lemma \ref{lemma: e_n} then implies that $\mathcal{F}(\delta_n \delta_5 \delta_4) \cong E_7$ or $E_8$. We can obtain the quasipositive surface $F(\delta_n^2 \delta_3)$ by successively adding the twisted bands of the letters letters $\sigma_5, \ldots , \sigma_{n-1}$ to $F(\delta_n \delta_5 \delta_4)$ and inductively applying Lemma \ref{lemma: e_n} shows that $\mathcal{F}(\delta_n^2 \delta_4)$ is $E_7$ or $E_8$.
Corollary \ref{cor: contains deltan2} then shows that $\mathcal{F}(b) \cong E_7$ or $E_8$.

Assume next that $r_i - r_{i-1} \geq 2$ for some $i$. Then $n \geq 5$. We saw above that up to rewriting, $P \supseteq P_0$ where $P_0$ has a letter of span $r_j - r_{i-1}$ for each $1 \leq i < j \leq 3$. Lemma \ref{lemma: bound on spans}(1), Lemma \ref{lemma: small span}, and Corollary \ref{cor: contains deltan2} imply that we are done if there are $i$ and $j$ such that $3 \leq r_j - r_{i-1} \leq n-3$. 

Suppose then that $r_j - r_{i-1} \in \{1, 2, n-2, n-1\}$ for each $i < j$. In this case $r_i - r_{i-1} \leq 2$ for each $i$ as otherwise $r_3 - r_0 = (r_3 - r_2) + (r_2 - r_1) + (r_1 - r_0) \geq n$. 

We can assume that $r_i - r_{i-1} = 2$ does not occur for successive values of $i$ as then $b$ contains an ascending $(2, 2)$ pair, so we are done by Lemma \ref{lemma: no ascending stairs of total length 3 or 4}. Thus $P$ contains a product $\sigma_k a_{k+1, k+3} = a_{k, k+3} \sigma_k$ or $a_{k, k+2} \sigma_{k+2} = a_{k, k+3} a_{k, k+2}$. In either case, after rewriting $P$ contains a letter of span $3$ so if $n \geq 6$, Lemma \ref{lemma: small span} and Corollary \ref{cor: contains deltan2} imply that $\mathcal{F}(b) \cong E_6, E_7$ or $E_8$. On the other hand, if $n = 5$ then $a_{r_0 r_1} a_{r_1 r_2} a_{r_2 r_3}$ is one of $a_{13} \sigma_3 \sigma_4, \sigma_1 a_{24} \sigma_4$, or $\sigma_1 \sigma_2 a_{35}$. In each case we are done by Lemma \ref{lemma: no ascending stairs of total length 3 or 4} and Corollary \ref{cor: contains deltan2}: 
\vspace{-.2cm} 
\begin{itemize}

\item $\delta_5^2 a_{13} \sigma_3 \sigma_4 = \delta_5^2 a_{13} a_{35} \sigma_3  \supset \delta_5^2 a_{13}  a_{35}$

\vspace{.2cm} \item $\delta_5^2 \sigma_1 a_{24} \sigma_4 = \delta_5^2 \sigma_1 a_{25} a_{24}  \sim \delta_5^2 \sigma_2 a_{13} a_{35} \supset \delta_5^2 a_{13}  a_{35}$

\vspace{.2cm} \item $\delta_5^2 \sigma_1 \sigma_2 a_{35} = \delta_5^2 a_{13} \sigma_1 a_{35} \supset \delta_5^2 a_{13}  a_{35}$

\end{itemize}
This completes the proof. 
\end{proof}

\section{The proof of Theorem \ref{thm: def baskets} when $n = 4$} 
\label{sec: n = 4}

In this section we prove Theorem \ref{thm: def baskets} for strongly quasipositive $4$-braids. Given Baader's theorem (Theorem \ref{thm: baader}), it suffices to prove the following proposition. 

\begin{prop} 
\label{prop: definite strongly quasipositive 4-braids} 
If $L$ is a definite basket link which is the closure of a BKL-positive word of the form $\delta_4^2 P \in B_4$, then $\delta_4^2 P$ is conjugate to a positive braid. 
\end{prop}

Recall the BKL-positive $4$-braid letters 
\begin{equation} 
\label{eqn: script L} 
\mathcal{L} = \{\sigma_1, \sigma_2, \sigma_3, a_{14}, a_{13}, a_{24}\}
\end{equation}
Conjugation by $\delta_4$ determines a permutation of $\mathcal{L}$ with two orbits: 
$$\sigma_1 \mapsto \sigma_2 \mapsto \sigma_3 \mapsto a_{14} \mapsto \sigma_1$$
and 
$$a_{13} \mapsto a_{24} \mapsto a_{13}$$

Throughout we consider braids of the form 
\begin{equation} 
\label{basic form} 
\delta_4^k P
\end{equation}
where $k \geq 0$ and $P$ is BKL-positive. Given such a braid $b$, choose from among all of its conjugates and their BKL-positive rewritings which have the form (\ref{basic form}), a braid $\delta_4^k P$ for which $k$ is maximal. Then $P$ does not contain a subword equalling $\delta_4$. 

Recall the integer $r(P) \geq 0$ defined by $P \supseteq \delta_4^{r(P)}$ but not $\delta_4^{r(P)+1}$ (cf. \S \ref{subsec: symm seifert forms of baskets}).

\begin{lemma}
If $b = \delta_4^2 P \in B_4$ is strongly quasipositive and definite, then $r(P) \leq 1$.
\end{lemma}

\begin{proof}
If $r(P) \geq 2$ then $b = \delta_4^2 P \supseteq \delta_4^4$ so that $\mathcal{F}(b) \supseteq \mathcal{F}(T(4,4))$, which is indefinite (Lemma \ref{restrictions on krn}), contrary to our hypotheses. 
\end{proof}

\begin{lemma}
If $r(P) = 0$, then $\delta_4^2 P$ definite implies that it is conjugate to a positive braid. 
\end{lemma}

\begin{proof}
We divide the proof into several cases.
\setcounter{case}{0}

\begin{case}
{\it $P$ contains neither $a_{13}$ nor $a_{24}$. }
\end{case}

We can suppose that $P \ne 1$ is not a positive braid so that up to conjugation,
$$b = \delta_4^2 a_{14}^{d_1} p_1 a_{14}^{d_2} p_2 \ldots a_{14}^{d_t} p_t$$
where $t \geq 1$, $p_1, \ldots, p_t$ are positive braids, and $d_1, \ldots, d_t > 0$. If $p_1 \cdots p_r \not \supseteq \sigma_3$, conjugating by $\delta_4$ changes $b$ into a positive braid, so we are done. Assume then that $p_1 \cdots p_t \supseteq \sigma_3$ and let $i_0$ be the smallest index $i$ for which $p_i \supseteq \sigma_3$. 
Since $\sigma_3 a_{14} \sigma_1 = \delta_4$, $p_i \not \supseteq \sigma_1$ for $i > i_0$. Consequently,
$$p_i = \left\{ 
\begin{array}{ll}
p_i(\sigma_1, \sigma_2) & \mbox{ if } i < i_0 \\ 
p_i(\sigma_2, \sigma_3) & \mbox{ if } i > i_0
\end{array} \right.$$
Then pushing one $\delta_4$ forward through $P$ till just before $p_{i_0}$ and conjugating the other to the end of $P$ and then pushing it backward till just after $p_{i_0}$ we obtain
\begin{eqnarray} 
b & = & \delta_4^2 a_{14}^{d_1} p_1(\sigma_1, \sigma_2) \cdots a_{14}^{d_{i_0 - 1}} p_{i_0 - 1}(\sigma_1, \sigma_2) a_{14}^{d_{i_0}} p_{i_0} 
                       a_{14}^{d_{i_0 + 1}} p_{i_0 + 1}(\sigma_2, \sigma_3) \cdots a_{14}^{d_t} p_t (\sigma_2, \sigma_3) \nonumber \\ 
& \sim &   \sigma_1^{d_1} p_1(\sigma_2, \sigma_3) \cdots \sigma_1^{d_{i_0 - 1}} p_{i_0 - 1}(\sigma_2, \sigma_3) \sigma_1^{d_{i_0}} \delta_4 p_{i_0} \delta_4
                       \sigma_3^{d_{i_0 + 1}} p_{i_0 + 1}(\sigma_1, \sigma_2) \cdots \sigma_3^{d_t} p_t (\sigma_1, \sigma_2), \nonumber 
\end{eqnarray}
which is a positive braid. 

\begin{case}
{\it $P$ contains at least one of $a_{13}$ or $a_{24}$.}
\end{case}

Since $\delta_4^2 a_{13} a_{24}$ and $\delta_4^2 a_{24} a_{13}$ are indefinite (Lemma \ref{lemma: no linking}), $P$ contains either no $a_{13}$ or no $a_{24}$. Then up to conjugation we can write
$$b = \delta_4^2 a_{13}^{d_1} w_1 a_{13}^{d_2} w_2  \cdots a_{13}^{d_t} w_t$$
where $w_i = w_i(\sigma_1, \sigma_2, \sigma_3, a_{14})$ and $t$ is minimal among all such BKL-positive expressions for conjugates of $b$. It follows that $d_i > 0$ and $w_i \ne 1$ for each $i$. Since $a_{13} \sigma_1 \sigma_3 = \delta_4$, the assumption that $r(P) = 0$ implies that $w_1 w_2 \cdots w_t \not \supseteq \sigma_1 \sigma_3$. Up to conjugating by $\delta_4^2$ we can suppose that $w_1 w_2 \cdots w_t \not \supseteq \sigma_3$. After a further conjugation by $\delta_4$ we have 
$$b \sim \delta_4^2 a_{24}^{d_1} p_1 a_{24}^{d_2} p_2  \cdots a_{24}^{d_t} p_t$$
where $p_i = p_i(\sigma_1, \sigma_2, \sigma_3) \ne 1$ for each $i$. If $t = 0$, $b \sim \delta_4^2$, a positive braid, and if $t = 1$, $b \sim \delta_4^2 a_{24}^{d_1} p_1(\sigma_1, \sigma_2, \sigma_3) =  \delta_4 (\sigma_1 \sigma_2 \sigma_3) (\sigma_3^{-1} \sigma_2^{d_1}  \sigma_3) p_1(\sigma_1, \sigma_2, \sigma_3) = \delta_4 \sigma_1 \sigma_2^{d_1 + 1} \sigma_3 p_1(\sigma_1, \sigma_2, \sigma_3)$, again a positive braid. We assume below that $t \geq 2$.

Since $\sigma_1 \sigma_3 a_{24} = \delta_4$, $p_1 \cdots p_{t-1} \not \supseteq \sigma_1 \sigma_3$. There are two cases to consider. 

\begin{subcase}
$p_1 \cdots p_{t-1} \not \supseteq \sigma_3$. 
\end{subcase}

If $p_i \not \supseteq \sigma_1$ for some $i < t$, then $p_i$ is a non-zero power of $\sigma_2$, say $p_i = \sigma_2^m$ where $m > 0$. Then $a_{24}^{d_i} p_i= \sigma_2 \sigma_3^{d_i} \sigma_2^{m -1}$, contrary to the minimality of $t$. (When $i = 1$, the condition implies that $b \sim \delta_4^2 a_{13}^{d_2} v_1 a_{13}^{d_3} v_2  \cdots a_{13}^{d_t} v_{t-1}$ where each $v_i = v_i(\sigma_1, \sigma_2, \sigma_3, a_{14})$ is a BKL-positive word.) Thus $p_i \supseteq \sigma_1$ for each $i < t$. 

If $p_2 p_3 \cdots p_t \supseteq \sigma_2$, then $P \supseteq \sigma_1 a_{24} \sigma_2 = \delta_4$, a contradiction. Thus $p_2 p_3 \cdots p_t \not \supseteq \sigma_2$. It follows that 
$$p_i = \left\{ 
\begin{array}{ll}
p_1(\sigma_1, \sigma_2) & \mbox{ if } i = 1 \\ 
\sigma_1^{m_i} & \mbox{ if } 1 < i < t \\ 
p_3(\sigma_1, \sigma_3) & \mbox{ if } i = t
\end{array} \right.$$
As above, the minimality of $t$ can be used to see that $p_1(\sigma_1, \sigma_2) = \sigma_1 q_1(\sigma_1, \sigma_2)$ where $q_1(\sigma_1, \sigma_2)$ is a positive braid. Also, $p_t(\sigma_1, \sigma_3) = \sigma_1^{m_t} \sigma_3^{m_0}$. 
Then 
\begin{eqnarray}
b & \sim & \delta_4^2 a_{24}^{d_1} p_1(\sigma_1, \sigma_2)  a_{24}^{d_2} \sigma_1^{m_2}  \cdots a_{24}^{d_{t-1}} \sigma_1^{m_{t-1}} a_{24}^{d_t} \sigma_1^{m_t} \sigma_3^{m_0} \nonumber \\ 
& \sim & \delta_4^2 \sigma_1^{m_0} \sigma_3^{m_t} a_{24}^{d_1} p_1(\sigma_1, \sigma_2)  a_{24}^{d_2} \sigma_1^{m_2}  \cdots a_{24}^{d_{t-1}} \sigma_1^{m_{t-1}} a_{24}^{d_t}   \nonumber
\end{eqnarray}
If $m_0, m_t > 0$, then $P \supseteq \sigma_1^{m_0} \sigma_3^{m_t} a_{24}^{d_1} \supseteq \sigma_1 \sigma_3 a_{24} = \delta_4$, contradicting our assumption that $r(P) = 0$. Thus one of $m_0, m_t$ is zero. If $m_t > 0$, the fact that $t \geq 2$ implies that $P \supseteq \sigma_3^{m_t} a_{24}^{d_1} p_1(\sigma_1, \sigma_2)  a_{24}^{d_2} = \sigma_3^{m_t} a_{24}^{d_1} \sigma_1 q_1(\sigma_1, \sigma_2) a_{24}^{d_2} \supseteq \sigma_3 \sigma_1 a_{24} = \delta_4$, again a contradiction. Hence $m_t = 0$. Then $m_0 > 0$ and 
\begin{eqnarray}
b & \sim & \delta_4^2 a_{24}^{d_1} p_1(\sigma_1, \sigma_2)  a_{24}^{d_2} \sigma_1^{m_2}  \cdots a_{24}^{d_{t-1}} \sigma_3^{m_0} \nonumber \\ 
& \sim & \delta_4^2 \sigma_3^{m_{t-1}} a_{24}^{d_t}  \sigma_1^{m_0} a_{24}^{d_1} p_1(\sigma_1, \sigma_2)  a_{24}^{d_2} \sigma_1^{m_2}  \cdots a_{24}^{d_{t-1}} \  \nonumber \\ 
& \supseteq & \delta_4^2 \sigma_3  \sigma_1 a_{24} \nonumber \\ 
& = & \delta_4^3 \nonumber
\end{eqnarray}
a contradiction, which completes the proof when $p_1 \cdots p_{t-1} \not \supseteq \sigma_3$.

\begin{subcase}
$p_1 \cdots p_{t-1} \not \supseteq \sigma_1$. 
\end{subcase}

Then $b \sim \delta_4^2 a_{24}^{d_1} p_1(\sigma_2, \sigma_3)  a_{24}^{d_2} p_1(\sigma_2, \sigma_3) \cdots a_{24}^{d_{t-1}} p_{t-1}(\sigma_2, \sigma_3) a_{24}^{d_t} p_t$. By the minimality of $t$, $p_i = \sigma_3 q_i \sigma_2$ where each $q_i$ is a positive braid for $i < t$. Thus 
$$b \sim \delta_4^2 a_{24}^{d_1} \sigma_3 q_1 \sigma_2  a_{24}^{d_2} \sigma_3 q_2 \sigma_2 \cdots a_{24}^{d_{t-1}} \sigma_3 q_{t-1} \sigma_2 a_{24}^{d_t} p_t$$
As $t \geq 2$, $b \supseteq (\sigma_2 \sigma_3)^2 \sigma_2 a_{24} \sigma_3 \sigma_2$, which is an indefinite form by Proposition \ref{prop: when definite}(6), so this case does not occur. 

This completes the proof of the lemma. 
\end{proof}

\begin{lemma}
If $r(P) = 1$, then $\delta_4^2 P$ definite implies that it is conjugate to a positive braid. 
\end{lemma}

\begin{proof}
In this case $P \supseteq \delta_4$ and so $b \supseteq b_0 = \delta_4^3$. Then $\mathcal{F}(b) \supseteq \mathcal{F}(b_0)  = \mathcal{F}(\delta_4^3) \cong E_6$. Hence $\mathcal{F}(b)$ is either $E_6, E_7$, or $E_8$ by Corollary \ref{cor: contains deltan2}.  Further, by (\ref{eqn: l(P) bound}) we see that  
$$3 + l_{BKL}(P) = \mbox{{\rm rank}} \; H_1(F(b)) = \left\{ 
\begin{array}{ll}
6 & \mbox{if } \mathcal{F}(b) \cong E_6 \\ 
7 & \mbox{if } \mathcal{F}(b) \cong E_7 \\ 
8 & \mbox{if } \mathcal{F}(b) \cong E_8 
\end{array} \right.$$
Since $r(P) = 1$, $l_{BKL}(P) \geq 3$. Thus $l_{BKL}(P)$ is $3$ if $\mathcal{F}(b) \cong E_6$, $4$ if $\mathcal{F}(b) \cong E_7$, and $5$ if $\mathcal{F}(b) \cong E_8$. It follows that $P = \delta_4$ if $\mathcal{F}(b) \cong E_6$ and is obtained from $\delta_4$ by adding one, respectively two, BKL-positive letters if $\mathcal{F}(b) \cong E_7$, respectively $\mathcal{F}(b) \cong E_8$.

\setcounter{case}{0}
\begin{case} 
{\it $\mathcal{F}(b) \cong E_6$.} 
\end{case} 

Then $b = \delta_4^3$ and therefore $\widehat b = T(3,4)$, so we are done.

\begin{case} 
{\it $\mathcal{F}(b) \cong E_7$.} 
\end{case} 

Then $P$ is obtained from $\delta_4$ by adding one BKL positive letter in $\mathcal{L}$ (cf. (\ref{eqn: script L})). The letter must be added between 
two letters of $\delta_4$ as otherwise $b \sim \delta_4^3 x$, contrary to our assumption that $k = 2$. Every expression for $\delta_4$ as a product of 
BKL-positive words has the form $uvw$ where $u, v, w \in \mathcal{L}$. Hence for such $u, v, w$, $P$ is either $uxvw$ or $uvxw$ for some $x \in \mathcal{L}$. 

The reader will verify that
$$\sigma_3 x \sigma_1 \sigma_2 \;\; \mbox{ and } \;\; \sigma_2 \sigma_3 x \sigma_1$$
are positive braids for all $x \in \mathcal{L}$, as are
$$\sigma_3 \delta_4 x \;\; \mbox{ and } \;\; x \delta_4 \sigma_1$$

\begin{subcase}
{\it $P = uxvw$ for some $x \in \mathcal{L}$.}
\end{subcase}

Then $b = \delta_4^2 u x u^{-1} \delta_4$. Up to conjugation by a power of $\delta_4$ and changing $x$, $b$ is either $\delta_4^2 \sigma_1 x \sigma_1^{-1} \delta_4$ or $\delta_4^2 a_{13} x a_{13}^{-1} \delta_4 = \delta_4^2 a_{13} x \sigma_1 \sigma_3$. 

In the first case, $b = \sigma_3 (\delta_4^2 x) (\sigma_2 \sigma_3)$ and as $\delta_4^2 x$ is a positive braid for each $x \in \mathcal{L}$, we are done. 

In the second, $b = \delta_4^2 a_{13} x \sigma_1 \sigma_3 =  \delta_4 a_{24} \delta_4 x\sigma_1 \sigma_3 = \sigma_1 \sigma_2^2 \sigma_3 (\delta_4 x \sigma_1) \sigma_3$. Since $\delta_4 x \sigma_1 = x' \delta_4 \sigma_1$ is always a positive braid, we are done.

\begin{subcase}
{\it $P = uvxw$ for some $x \in \mathcal{L}$.}
\end{subcase}

Then $b = \delta_4^3 (w^{-1} x w)$. Up to conjugation by a power of $\delta_4$, $b$ is either $\delta_4^3 \sigma_3^{-1} x \sigma_3$ or $\delta_4^3 a_{24}^{-1} x a_{24}$. 

If the former occurs, then $b \sim \delta_4^2 \sigma_1 \sigma_2 x \sigma_3 = \delta_4 \sigma_2 (\sigma_3 x' \delta_4) \sigma_3$ where $x' = \delta_4 x \delta_4^{-1} \in \mathcal{L}$. Since $\sigma_3 x' \delta_4 = \sigma_3 \delta_4 x''$ is a positive braid for each $x' \in \mathcal{L}$, we are done. 

Otherwise,  $b = \delta_4^3 a_{24}^{-1} x a_{24} = \delta_4^2 \sigma_1 \sigma_3 x a_{24} \sim  \sigma_1 (\sigma_3 x \delta_4) (a_{13} \delta_4) = \sigma_1 (\sigma_3 x \delta_4) \sigma_1 \sigma_2^2 \sigma_3$, so we are done as $\sigma_3 x \delta_4 = \sigma_3 \delta_4 x'$ is a positive braid. 

This completes the proof when $\mathcal{F}(b) \cong E_7$. 

\begin{case} 
{\it $\mathcal{F}(b) \cong E_8$.} 
\end{case} 

Then $P$ is obtained from $\delta_4$ by adding two BKL positive letters. Since $P$ does not contain $\delta_4$ as a subword, for any expression $\delta_4 = uvw$ where $u, v, w \in \mathcal{L}$, there are seven possibilities for $P$ : 
$$P = \left\{
\begin{array}{l} 
(1) \; x u y vw \\
(2) \; x uvy w \\
(3) \; u x y vw \\ 
(4) \; u x v y w \\ 
(5) \; u x vw y \\
(6) \; u v xy w \\
(7) \; u v x w y
\end{array} \right. $$
Since $\delta_4^2 u x vw y \sim \delta_4^2 y' u x vw$ and  $\delta_4^2 uv x w y \sim \delta_4^2 y' u x vw$ where $y' = \delta_4^{-2} y \delta_4^2 \in \mathcal{L}$, we need only deal with cases (1), (2), (3), (4) and (6).

\begin{subcase}
{\it $P = x u y vw$ for some $x, y \in \mathcal{L}$.}
\end{subcase}

Then $b \sim \delta_4^2  x u y vw = \delta_4^2  x (u y u^{-1}) \delta_4$. Up to conjugation, we can take $u$ to be either $\sigma_1$ or $a_{13}$. Further, by Lemma \ref{lemma: no linking}, definiteness implies that if $x = a_{13}$, respectively $a_{24}$, then $y \ne a_{24}$, respectively $a_{13}$. 

If $u = \sigma_1$, then $b \sim \delta_4^2  x \sigma_1 y \sigma_2 \sigma_3 \sim (\delta_4  x \sigma_1) (\delta_4 y' \sigma_1) \sigma_2$ where $y' = \delta_4^{-1} y \delta_4$. Since $\delta_4  x \sigma_1$ and $\delta_4 y' \sigma_1$ are positive braids for all $x, y \in \mathcal{L}$, we are done.

If $u = a_{13}$, then $b \sim \delta_4^2  x a_{13} y a_{13}^{-1} \delta_4 \sim \delta_4^2  x a_{13} y \sigma_1 \sigma_3  \sim \sigma_3 x' \delta_4  a_{24}  \delta_4 y \sigma_1 =  (\sigma_3  x' \sigma_1 \sigma_2) \sigma_2 
\sigma_3 (\delta_4 y \sigma_1)$, which is a positive braid. 

\begin{subcase}
{\it $P = x uv y w $ for some $x, y \in \mathcal{L}$.}
\end{subcase}

Then $b \sim \delta_4^2  x uv y w = \delta_4^2  x \delta_4 (w^{-1} y w)$. Up to conjugation, we can take $w$ to be either $\sigma_3$ or $a_{24}$.

If $w = \sigma_3$, then $b \sim \delta_4^2 x \delta_4 \sigma_3^{-1} y \sigma_3 \sim (\delta_4 x \sigma_1) \sigma_2 (y \delta_4) \sigma_2$. This is a positive braid unless $y = a_{24}$. But in this case $x \ne a_{13}$ and so $\sigma_3 x \sigma_1$ is a positive braid. Hence $b \sim \sigma_1 \sigma_2 (\sigma_3 x \sigma_1 \sigma_2) a_{24} \delta_4 \sigma_2 \sim (\sigma_3 x \sigma_1 \sigma_2) a_{24} \delta_4 (\sigma_2 \sigma_1 \sigma_2) = (\sigma_3 x \sigma_1 \sigma_2) \delta_4 a_{13}  (\sigma_1 \sigma_2 \sigma_1) = (\sigma_3 x \sigma_1 \sigma_2) \delta_4 \sigma_1 \sigma_2^2 \sigma_1)$, which is a positive braid. 

If $w = a_{24}$, then $b \sim \delta_4^2  x \delta_4 a_{24}^{-1} y a_{24} \sim \delta_4^2  x \sigma_1 \sigma_3 y a_{24} = (\sigma_2 \sigma_3 x \sigma_1) \sigma_3 y a_{24} \delta_4 \sigma_1 = (\sigma_2 \sigma_3 x \sigma_1) (\sigma_3 y  \delta_4) a_{24}\sigma_1  = (\sigma_2 \sigma_3 x \sigma_1) (\sigma_3 y  \delta_4) \sigma_1 \sigma_2$, which is a positive braid. 

\begin{subcase}
{\it $P = u x y vw $ for some $x, y \in \mathcal{L}$.}
\end{subcase}

Then $b \sim \delta_4^2  (a  x y a^{-1}) \delta_4  = \delta_4^2  (axa^{-1}) \delta_4 y$. Up to conjugation, we can take $u$ to be either $\sigma_1$ or $a_{24}$.

If $u = \sigma_1$, then $b \sim \delta_4^2 \sigma_1 xy \sigma_2 \sigma_3 \delta_4 = \sigma_3 \delta_4 (\delta_4 x) \delta_4 y' \sigma_1 \sigma_2 = \sigma_3 \delta_4 (\delta_4 x \sigma_1) (\sigma_2 \sigma_3 y' \sigma_1) \sigma_2$, which is a positive braid. 

 If $u = a_{13}$, then $b \sim \delta_4^2  a_{13}  x y a_{13}^{-1}  \delta_4 =  \delta_4  a_{24}  \delta_4 x y \sigma_1 \sigma_3 = \sigma_1 \sigma_2^2  \sigma_3  \delta_4 x y \sigma_1 \sigma_3 = \sigma_1 \sigma_2   (\sigma_2 \sigma_3  x' \delta_4) y \sigma_1 \sigma_3$ where $x' = \delta_4 x \delta_4^{-1}$.  Hence $b \sim \sigma_1 \sigma_2   (\sigma_2 \sigma_3  x' \sigma_1) (\sigma_2 \sigma_3 y \sigma_1) \sigma_3$, which is always a positive braid. 

\begin{subcase}
{\it $P =u x v y w $ for some $x, y \in \mathcal{L}$.}
\end{subcase}

There are four BKL-positive expressions for $\delta_4$ up to conjugacy by a power of $\delta_4$: $\sigma_1 \sigma_2 \sigma_3, \sigma_1 \sigma_3 a_{24}$, $\sigma_1 a_{24} \sigma_2$, and $a_{24} \sigma_1 \sigma_3$. We consider each of these separately. 

First suppose that $b \sim \delta_4^2 \sigma_1 x \sigma_2 y \sigma_3$. Then $b \sim \sigma_3  x' \delta_4   \sigma_3 y' \delta_4 \sigma_3$ where $x' = \delta_4^2 x \delta_4^{-2}$ and $y' = \delta_4 y \delta_4^{-1}$. Hence $b \sim (\sigma_3  x'  \delta_4)  (\sigma_3 y' \delta_4) \sigma_3$, which is a positive braid.  

Next suppose that $b \sim \delta_4^2 \sigma_1 x \sigma_3 y a_{24}$. Then $b \sim \sigma_3  x' \delta_4   \sigma_3 y' \delta_4 a_{24} = (\sigma_3  x' \sigma_1 \sigma_2)   \sigma_3 a_{14} y' \sigma_1 \sigma_2^2 \sigma_3$ where $x' = \delta_4^2 x \delta_4^{-2}$ and $y' = \delta_4 y \delta_4^{-1}$. Hence $b \sim (\sigma_3  x' \sigma_1) \sigma_1 (\sigma_2 \sigma_3 y' \sigma_1) \sigma_2^2 \sigma_3$. Now $\sigma_2 \sigma_3 y' \sigma_1$ is a positive braid as is $\sigma_3  x' \sigma_1$ as long as $x' \ne a_{14}$. But if $x' = a_{14}$, then $x = \sigma_2$ so that $b \sim \delta_4^2 \sigma_1 \sigma_2 \sigma_3 y a_{24} = \delta_4^3 y a_{24} = \delta_4^2 y' \delta_4 a_{24} = \delta_4 \sigma_1 \sigma_2 (\sigma_3 y' \sigma_1 \sigma_2) (\sigma_3 a_{24}) = \delta_4 \sigma_1 \sigma_2 (\sigma_3 y' \sigma_1 \sigma_2) (\sigma_2 \sigma_3)$, which is a positive braid.  

If $b \sim \delta_4^2 \sigma_1 x a_{24} y \sigma_2$, then $b \sim (\sigma_3 x' \sigma_1 \sigma_2) (\sigma_3 \delta_4 a_{24}) y \sigma_2 = (\sigma_3 x' \sigma_1 \sigma_2) (\sigma_3 \sigma_1 \sigma_2)(\sigma_2 \sigma_3 y) \sigma_2$. Now $\sigma_3 x' \sigma_1 \sigma_2$ is a positive braid while $\sigma_2 \sigma_3 y$ is as long as $y \ne a_{13}$. But $y \ne a_{13}$ as $b \sim \delta_4^2 \sigma_1 x a_{24} y \sigma_2 = \delta_4^2 \sigma_1 x a_{24} a_{13} \sigma_2$ is indefinite (Lemma \ref{lemma: no linking}). 

Finally suppose that $b \sim \delta_4^2 a_{24} x \sigma_1 y \sigma_3$. Then $b \sim \delta_4 \sigma_1 \sigma_2^2 \sigma_3 x \sigma_1 y \sigma_3 \sim (\sigma_2 \sigma_3 x \sigma_1) y \sigma_3 \delta_4 \sigma_1 \sigma_2 = (\sigma_2 \sigma_3 x \sigma_1) (\delta_4 y' \sigma_2 \sigma_1 \sigma_2) = (\sigma_2 \sigma_3 x \sigma_1) (\delta_4 y' \sigma_1) \sigma_2 \sigma_1$, which is a positive braid. 

\begin{subcase}
{\it $P = uv xy w$ for some $x, y \in \mathcal{L}$.}
\end{subcase}

Then $b \sim \delta_4^2   uv xy w   = \delta_4^3 w^{-1} xy w$. Up to conjugation, we can take $w$ to be either $\sigma_3$ or $a_{24}$.

If $w = \sigma_3$, then $b \sim \delta_4^3 \sigma_3^{-1} xy \sigma_3  = \delta_4^2  \sigma_1 \sigma_2 xy \sigma_3 = (\sigma_2 \sigma_3 x' \delta_4) (y' \delta_4) \sigma_3 = (\sigma_2 \sigma_3 x' \sigma_1) (\sigma_2 \sigma_3 y' \sigma_1) \sigma_2 \sigma_3^2$ where $x' = \delta_4^2 x \delta_4^{-2}$ and $y' = \delta_4 y \delta_4^{-1}$. As this is a positive braid for each $x, y \in \mathcal{L}$, we are done. 

Finally suppose that  $w = a_{24}$. Then $b \sim \delta_4^3 a_{24}^{-1} xy a_{24} = 
\delta_4^3 (\sigma_3^{-1} \sigma_2^{-1} \sigma_3) xy a_{24} = \delta_4^2 \sigma_1 \sigma_3 xy a_{24} \sim \sigma_1 \sigma_3 xy a_{24} \delta_4^2 
= \sigma_1 \sigma_3 x \delta_4 y' a_{13} \delta_4 = \sigma_1 (\sigma_3 x \sigma_1 \sigma_2) (\sigma_3 y'  \sigma_1 \sigma_2) \sigma_2 \sigma_3$ where 
$y' = \delta_4 y \delta_4^{-1}$. Thus $b$ is conjugate to a positive braid. 
\end{proof}

\section{The proof of Theorem \ref{thm: def baskets} when $n \geq 5$} 
\label{sec: n = 5}

In this section we prove 

\begin{prop} 
\label{prop: definite strongly quasipositive n-braids} 
If $L$ is a definite basket link which is the closure of a BKL-positive word of the form $\delta_n^2 P \in B_n$ where $n \geq 5$, then $b$ is conjugate to a positive braid. 
\end{prop}
We assume throughout this section that $n \geq 5$ and that 
\begin{equation}
b = \delta_n^k P
\end{equation}
is a definite braid where $P \in B_n$ is a BKL-positive word and $k \geq 2$. Given Theorem \ref{thm: definite 3-braids} and Proposition \ref{prop: definite strongly quasipositive 4-braids}, we can suppose that $n = n(\widehat b)$ (cf. \S \ref{subsec: reducing indices}) without loss of generality.

\begin{lemma} 
\label{lemma: wlog r = 0}
Proposition \ref{prop: definite strongly quasipositive n-braids} holds if $r(P) > 0$.
\end{lemma}

\begin{proof}
If $r(P) > 0$, then $\mathcal{F}(b) \supseteq \mathcal{F}(\delta_n^{2 + r(P)}) = \mathcal{F}(T(n, 2 + r(P))$. Since $b$ is definite, we have $n = 5$ and $r(P) = 1$ (Lemma \ref{restrictions on krn}). Then $\mathcal{F}(b) \supseteq \mathcal{F}(\delta_5^3) = \mathcal{F}(T(3,5) \cong E_8$. Corollary \ref{cor: contains deltan2} now implies that $b$ is the positive braid $\delta_5^3$.  
\end{proof}

We assume below that $r(P) = 0$.   

\begin{lemma} 
\label{lemma: the case e6}
Suppose that $b = \delta_n^2 P$ is a definite BKL-positive braid where $n \geq 5$. If $\mathcal{F}(b)$ is congruent to $E_6, E_7$, or $E_8$, then $b$ is conjugate to a positive braid.  
\end{lemma}

\begin{proof}
By (\ref{eqn: l(P) bound}), 
$$(n - 1) + l_{BKL}(P) = \mbox{{\rm rank}} \; H_1(F(b)) = \left\{ 
\begin{array}{ll}
6 & \mbox{if } \mathcal{F}(b) \cong E_6 \\ 
7 & \mbox{if } \mathcal{F}(b) \cong E_7 \\ 
8 & \mbox{if } \mathcal{F}(b) \cong E_8 
\end{array} \right.$$
In particular, $l_{BKL}(P) \leq 4$ and equals $4$ if and only if $n = 5$ and $\mathcal{F}(b) \cong E_8$.  

If $l_{BKL}(P) = 0$, then $b$ is the positive braid $\delta_n^k$. 

If $l_{BKL}(P) = 1$, we can conjugate $b$ by a power of $\delta_n$ to see that $b \sim \delta_n^2 a_{1s} \sim \delta_n a_{1s} \delta_n = \delta_n \delta_{s-1} \delta_{s-2}^{-1} \delta_n = \delta_n \delta_{s-1} \sigma_{s-1} \sigma_{s} \ldots \sigma_{n-1}$, which is a positive braid. 

If $l_{BKL}(P) = 2$, then Lemmas \ref{lemma: not minimal} and \ref{lemma: contains sigma1} imply that 
up to conjugation by a power of $\delta_n$, the two letters of $P$ are $\sigma_1$ and $a_{rn}$. Now $r$ is either $1$ or $2$ as otherwise Lemma \ref{lemma: innermost commute} implies that $n(\widehat b) < n$. But $r = 1$ implies the two letters of $\delta_n b \delta_n^{-1}$ are $\sigma_1$ and $\sigma_2$, so $\sigma_{n-1}$ is not covered, contrary to Lemma \ref{lemma: not minimal}. And if $r = 2$, $P$ is either $\sigma_1 a_{2n} = a_{1n} \sigma_1$, which is impossible, or $a_{2n} \sigma_1$. In the latter case, $b = \delta_n^2 a_{2n} \sigma_1 = \delta_n \sigma_1 \sigma_2^2 \sigma_3 \sigma_4 \cdots \sigma_{n-1} \sigma_1$, a positive braid.

Next suppose that $l_{BKL}(P) = 3$. By Lemma \ref{lemma: not minimal} we can assume that the letters of $P$ are $\sigma_1, a_{rs}$, and $a_{tn}$. Lemma \ref{lemma: innermost commute} implies that $\min \{r, t\} \in \{1, 2\}$.  

 If $t = 1$, then conjugation by $\delta_n$ yields the letters $\sigma_1, \sigma_2, a_{r+1, s+1}$. Hence either $r = n-1$, so $a_{rs} = \sigma_{n-1}$, or $s = n-1$. The former is impossible as it implies that the letters of $\delta_n^2 P \delta_n^{-2}$ are $\sigma_1, \sigma_2$, and $\sigma_3$, which doesn't cover $\sigma_{n-1}$. Thus $s = n-1$. But then $r = 1$ or $2$ by Lemma \ref{lemma: innermost commute} applied to $\delta_n P \delta_n^{-1}$. In either case, the letters of $\delta_n^2 P \delta_n^{-2}$ are $\sigma_1, \sigma_3$, and $a_{1,r+2}$ and as $r+2 \leq 4 < n$, $\sigma_{n-1}$ is not covered, a contradiction. 

Similarly if $t = 2$, conjugation by $\delta_n$ yields the letters $\sigma_1, a_{13}, a_{r+1, s+1}$. Hence either $r = n-1$, so $a_{rs} = \sigma_{n-1}$, or $s = n-1$. The former is impossible as it implies that the letters of $\delta_n^2 P \delta_n^{-2}$ are $\sigma_1, a_{24}$, and $\sigma_3$. Thus $s = n-1$. But then $r = 2$ (Lemma \ref{lemma: innermost commute} and Lemma \ref{lemma: no linking}), and in this case the letters of $\delta_n^2 bP \delta_n^{-2}$ are $\sigma_3, a_{14}$, and $a_{24}$, a contradiction. 

Assume that $t > 2$. Then $r \in \{1, 2\}$ and $s \in \{t, n\}$ (Lemma \ref{lemma: innermost commute} and Lemma \ref{lemma: no linking}). If $s = n$, the letters of $\delta_n P \delta_n^{-1}$ are $\sigma_2, a_{1, r+1}$, and $a_{1, t+1}$. Then $t = n-1$ so that the letters of $\delta_n^2 bP \delta_n^{-2}$ are $\sigma_3, a_{2, r+2},$ and $\sigma_2$, a contradiction. If $s = t$, the letters of $\delta_n P \delta_n^{-1}$ are $\sigma_2, a_{1, t+1}$, and $a_{r+1, t+1}$. Then $t = n-1$ so that the letters of $\delta_n^2 bP \delta_n^{-2}$ are $\sigma_3, a_{1, r+2},$ and $\sigma_1$, a contradiction.

This completes the proof in the case that $l_{BKL}(P) = 3$. In fact it proves more. The same arguments show that the lemma holds if the number of distinct letters in $P$ is at most three. We will use this below. 

Finally suppose that $l_{BKL}(P) = 4$ so as noted above, $n = 5$ and $\mathcal{F}(b) \cong E_8$. From the last paragraph we can suppose that $P$ is made up of four distinct BKL-positive letters. There are two conjugacy classes of BKL-positive letters:
$$A = \{\sigma_1, \sigma_2, \sigma_3, \sigma_4, a_{15}\}$$
and 
$$B = \{a_{13}, a_{24}, a_{35}, a_{14}, a_{25}\}$$
By Lemma \ref{lemma: no linking}, there are at most two letters of $P$ which come from $B$. Hence at least two come from $A$.

\setcounter{case}{0}
\begin{case}
{\rm {\it Each letter of $P$ comes from $A$.}}
\end{case}

In this case $P$ is a product of four of the letters $\sigma_1, \sigma_2, \sigma_3, \sigma_4, a_{15}$. Hence after conjugation by an appropriate power of $\delta_5$ we have that $P$, and therefore $b$, becomes a positive braid. 

\begin{case}
{\rm {\it Three letters of $P$ come from $A$ and one from $B$.}}
\end{case}

Then up to conjugation, $b = \delta_5^2 a_1 a_2 a_3 a_{1r}$ where $a_1, a_2, a_3 \in A$ and $a_{1r} \in B$ so that $r \in \{3, 4\}$. Then
$$b \sim \delta_5 a_1 a_2 a_3 a_{1r} \delta_5 = \left\{ 
\begin{array}{ll} 
\delta_5 a_1 a_2 a_3 \sigma_1 \sigma_2^2 \sigma_3 \sigma_4 & \mbox{ if } r = 3 \\
\delta_5 a_1 a_2 a_3 \sigma_1 \sigma_2 \sigma_{3}^2 \sigma_4  & \mbox{ if } r = 4 
 \end{array} 
 \right.$$ 
 By hypothesis, the $a_i$ are distinct. If no $a_i$ is $a_{15}$, $b$ is a positive braid. 

If $a_1 = a_{15}$, then as $\delta_5 a_{15} = \sigma_1 \delta_5$ and $a_2, a_3 \in A \setminus \{a_{15}\}$, we are done. 

If $a_3 = a_{15}$ then note that either 
$$b \sim \delta_5 a_1 a_2 a_{15} \sigma_1 \sigma_2 \sigma_3^2 \sigma_4 = \delta_5 a_1 a_2 \delta_5 \sigma_3 \sigma_4,$$
so we are done, or 
\begin{eqnarray} 
b \sim \delta_5 a_1 a_2 a_{15} \sigma_1 \sigma_2^2 \sigma_3 \sigma_4 =  \delta_5 a_1 a_2 \delta_5 \sigma_3^{-1} \sigma_2 \sigma_3 \sigma_4 =  \delta_5 a_1 a_2 \sigma_4^{-1} \delta_5 \sigma_2 \sigma_3 \sigma_4  \nonumber   
\end{eqnarray} 
We're done if $a_2 = \sigma_4$. If $a_2 = \sigma_3$ and $a_1 \ne \sigma_4$, then  $\delta_5 a_1 a_2 \sigma_4^{-1} = 
\delta_5 a_1 \sigma_3 \sigma_4^{-1} = a_1' \sigma_4 \delta_5 \sigma_4^{-1} = a_1' \sigma_4 \sigma_1 \sigma_2 \sigma_3$ where $a_1' \in \{\sigma_2, \sigma_3, \sigma_4\}$, so we're done. If  $a_2 = \sigma_3$ and $a_1 = \sigma_4$, then 
\begin{eqnarray} 
b \sim \delta_5 \sigma_4 \sigma_3 \sigma_4^{-1} \delta_5 \sigma_2 \sigma_3 \sigma_4  & = & \delta_5 \sigma_3^{-1} \sigma_4 \sigma_3 \delta_5 \sigma_2 \sigma_3 \sigma_4  \nonumber   \\ 
& = & \sigma_1 \sigma_2 \sigma_3 \sigma_4 \sigma_3^{-1}  \sigma_4 \sigma_3 \delta_5 \sigma_2 \sigma_3 \sigma_4  \nonumber   \\ 
& = & \sigma_1 \sigma_2 \sigma_4^{-1} \sigma_3 \sigma_4^2 \sigma_3 \delta_5 \sigma_2 \sigma_3 \sigma_4  \nonumber   \\ 
& = & \sigma_4^{-1} \sigma_1 \sigma_2  \sigma_3 \sigma_4^2 \sigma_3 \delta_5 \sigma_2 \sigma_3 \sigma_4  \nonumber  \\ 
& \sim & \sigma_1 \sigma_2  \sigma_3 \sigma_4^2 \sigma_3 \delta_5 \sigma_2 \sigma_3,   \nonumber    
\end{eqnarray} 
a positive braid. 

Suppose $a_2 = \sigma_i \in \{\sigma_1, \sigma_2\}$. Then $b \sim \delta_5 a_1 \sigma_i \sigma_4^{-1}  \delta_5 \sigma_2 \sigma_3 \sigma_4 = \delta_5 a_1 \sigma_4^{-1} \sigma_i \delta_5 \sigma_2 \sigma_3 \sigma_4$. We're done if $a_1 \in \{\sigma_1, \sigma_2, \sigma_4\}$ as then $b \sim \delta_5 \sigma_4^{-1} a_1  \sigma_i \delta_5 \sigma_2 \sigma_3 \sigma_4 = \sigma_1 \sigma_2 \sigma_3 a_1  \sigma_i \delta_5 \sigma_2 \sigma_3 \sigma_4$. If $a_1 = \sigma_3$, then 
\begin{eqnarray}  
b \sim \delta_5 \sigma_3 \sigma_4^{-1} \sigma_i \delta_5 \sigma_2 \sigma_3 \sigma_4  & = & \sigma_1 \sigma_2  \sigma_3 (\sigma_4 \sigma_3 \sigma_4^{-1}) \sigma_i \delta_5 \sigma_2 \sigma_3  \sigma_4   \nonumber   \\ 
& = & \sigma_1 \sigma_2  \sigma_3 (\sigma_3^{-1} \sigma_4 \sigma_3) \sigma_i \delta_5 \sigma_2 \sigma_3  \sigma_4  \nonumber   \\ 
& = & \sigma_1 \sigma_2  \sigma_4 \sigma_3 \sigma_i \delta_5 \sigma_2 \sigma_3  \sigma_4,  \nonumber   
\end{eqnarray} 
a positive braid. This completes the argument when $a_3 = a_{15}$.

Finally, if $a_2 = a_{15}$ we are done if $a_1 \ne \sigma_4$ as $\delta_5 a_1 a_2 = (\delta_5 a_1 \delta_5^{-1}) (\delta_5 a_{15})$. If $a_2 = a_{15}$ and $a_1 =\sigma_4$, then $\delta_5 a_1 a_2 = \delta_5 \sigma_4 (\sigma_4^{-1} \sigma_3^{-1} \sigma_2^{-1}) \delta_5 = \delta_5 \sigma_3^{-1} \sigma_2^{-1} \delta_5 = \sigma_4^{-1} \sigma_3^{-1} \delta_5^{2}$. It is obvious then that $b = \delta_5 a_1 a_2 \sigma_4^{-1} \delta_5 \sigma_2 \sigma_3 \sigma_4 = \sigma_4^{-1} \sigma_3^{-1} \delta_5^{2} \sigma_4^{-1} \delta_5 \sigma_2 \sigma_3 \sigma_4 = (\sigma_3 \sigma_4)^{-1} \delta_5 \sigma_1 \sigma_2 \sigma_3 \delta_5 \sigma_2) (\sigma_3 \sigma_4)$ is conjugate to a positive braid. 

\begin{case}
{\rm {\it  Two letters of $P$ come from $A$ and two from $B$.}}
\end{case}

Lemma \ref{lemma: no linking} implies that up to conjugation by a power of $\delta_5$, the two letters from $B$ are $a_{13}$ and $a_{35}$. Lemma \ref{lemma: 2-step descending stairs of length two or more indefinite} implies that $a_{13}$ occurs before $a_{35}$ in $P$. Note as well that we are reduced to the previous subcase if $P$ contains $a_{13} a_{35} = a_{15} a_{13}$. Thus there are distinct letters $a_1, a_2 \in A$ such that 
\begin{equation} 
\label{eqn: possibilities}
P = \left\{ \begin{array}{l} a_{13} a_1 a_{35} a_2, \mbox{ or}   \\ 
a_{13} a_1 a_2 a_{35}, \mbox{ or}  \\ 
a_1 a_{13} a_2 a_{35} \end{array} \right.
\end{equation} 
In particular, $a_{13}$ is one of the first two letters of $P$ and $a_{35}$ one of the last two letters.

Since $\delta_5 \{a_{13}, a_{35}\} \delta_5^{-1} = \{a_{24}, a_{14}\}$, Lemma \ref{lemma: not minimal} implies that either $\sigma_4$ or $a_{15}$ must be contained in the letters of $\delta_5 P \delta_5^{-1}$. Thus $\sigma_3$ or $\sigma_4$ is a letter of $P$. Next note that as $\delta_5^3 \{a_{13}, a_{35}, \sigma_3\} \delta_5^{-3} = \{a_{13}, a_{14}, \sigma_1\}$ and $\delta_5^3 \{a_{13}, a_{35}, \sigma_4\} \delta_5^{-3} = \{a_{13}, a_{14}, \sigma_2\}$, either $\sigma_4$ or $a_{15}$ must be contained in the letters of $\delta_5^3 P \delta_5^{-3}$. Thus the letters of $P$ are either $a_{13}, a_{35}, \sigma_1, \sigma_3$, or $a_{13}, a_{35}, \sigma_1, \sigma_4$, or $a_{13}, a_{35}, \sigma_2, \sigma_3$, or $a_{13}, a_{35}, \sigma_2,  \sigma_4$. We consider these cases separately. 

\begin{subcase}
{\rm {\it The letters of $P$ are $a_{13}, a_{35}, \sigma_1, \sigma_3$.}}
\end{subcase}

Since $\sigma_1$ commutes with $a_{35}$ and $\sigma_3$, if it occurs after $a_{13}$, $P$ can be be rewritten to contain $a_{13} \sigma_1 = \sigma_1 \sigma_2$ and we can then appeal to the previous subcase. So without loss of generality we can suppose that $\sigma_1$ occurs before $a_{13}$. Then (\ref{eqn: possibilities}) implies that $P = \sigma_1 a_{13} \sigma_3 a_{35} = \sigma_1 a_{14} a_{13} a_{35}$. But this is impossible since the presence of $a_{14}$ and $a_{35}$ in (a rewritten) $P$ forces $\mathcal{F}(b)$ to be indefinite.

\begin{subcase}
{\rm {\it The letters of $P$ are $a_{13}, a_{35}, \sigma_1, \sigma_4$.}}
\end{subcase}

A similar argument  to that used in the previous subcase shows that $P = \sigma_1 a_{13} \sigma_4 a_{35} = \sigma_1 a_{13} \sigma_3 \sigma_4$, so again we are done by the previous case.

\begin{subcase}
{\rm {\it The letters of $P$ are $a_{13}, a_{35}, \sigma_2, \sigma_3$.}}
\end{subcase}

Since $\sigma_2 a_{13} = \sigma_1 \sigma_2$ and $a_{35} \sigma_3 = \sigma_3 \sigma_4$, we are reduced to the previous case if $\sigma_2$ occurs before $a_{13}$ or $\sigma_3$ occurs after $a_{35}$ (cf. (\ref{eqn: possibilities})). Assume that this isn't the case. Lemma \ref{lemma: no linking} rules out the possibility that $P$ contains the subwords $\sigma_2 a_{35} = a_{25} \sigma_2$ or $a_{13} \sigma_3 = a_{14} a_{13}$. From (\ref{eqn: possibilities}), the only possibility is for $P = a_{13} \sigma_2 \sigma_3 a_{35} = a_{13} a_{24} \sigma_2 a_{35}$, which is ruled out by Lemma \ref{lemma: no linking}.

\begin{subcase}
{\rm {\it The letters of $P$ are $a_{13}, a_{35}, \sigma_2, \sigma_4$.}}
\end{subcase}

Since $\sigma_2$ commutes with $\sigma_4$, if it occurs before $a_{13}$, $P$ can be be rewritten to contain $\sigma_2 a_{13}  = \sigma_1 \sigma_2$ and we can then appeal to the previous case. So without loss of generality we can suppose that $\sigma_2$ occurs after $a_{13}$. Similarly $\sigma_4$ must occur after $a_{35}$ so from (\ref{eqn: possibilities}), $P = a_{13} \sigma_2 a_{35} \sigma_4 = a_{13} a_{25} \sigma_2 \sigma_4$. This is ruled out by Lemma \ref{lemma: no linking}. 
\end{proof}

\begin{lemma}
\label{lemma: done if all spans 1 or n-1}
If $b = \delta_n^2 P$ where $P$ is a product of the letters of span $1$ and $n-1$, then $b$ is conjugate to a positive braid. 
\end{lemma}

\begin{proof}
Write $b = \delta_n^2 p_1 a_{1n}^{d_1} p_2 a_{1n}^{d_2}\ldots p_r a_{1n}^{d_r}p_{r+1}$ where each $p_i$ is a positive braid. 

If $P$ contains $\sigma_{n-1} a_{1n} \sigma_1$, then $\delta_n^2 b \delta_n^{-2}$ contains the $3$-step ascending staircase $\sigma_1 \sigma_2 \sigma_3$ and therefore $\mathcal{F}(b) \cong E_6, E_7$, or $E_8$ by Lemma \ref{lemma: $3$-step ascending staircase e6 definite}. In this case Lemma \ref{lemma: the case e6} implies that $b$ is conjugate to a positive braid. 

Similarly if $P$ contains $(a_{1n} \sigma_1)^2$, then $\delta_n P \delta_n^{-1}$ contains $(\sigma_1 \sigma_2)^2 = (a_{13} \sigma_1)^2 \supset a_{13}^2$. Lemma \ref {lemma: no repeated span 2 letters} then implies that $\mathcal{F}(b) \cong E_6, E_7$, or $E_8$ and so $b$ is conjugate to a positive braid by Lemma \ref{lemma: the case e6}. 

Assume below that $P$ contains neither $\sigma_{n-1} a_{1n} \sigma_1$ nor $(a_{1n} \sigma_1)^2$. Consequently, if there is a $1 \leq k \leq r$ such that $p_{k}$ contains $\sigma_1$, then
\vspace{-.2cm} 
\begin{itemize}

\item  if $i < k, p_i$ does not contain $\sigma_{n-1}$; otherwise $P$ contains $\sigma_{n-1} a_{1n} \sigma_1$; 

\vspace{.2cm} \item if $k < i, p_i$ does not contain $\sigma_1$; otherwise $P$ contains $a_{1n} \sigma_1 a_{1n} \sigma_1 = (a_{1n} \sigma_1)^2$. 

\end{itemize} 
Now rewrite $P$ by conjugating one copy of $\delta_n$ forward through $P$ till just after $a_{1n}^{d_{k-1}}$. Since there is no $\sigma_{n-1}$
in $p_1 p_2 \ldots p_{k-1}$ and since $\delta_n a_{1n} \delta_n^{-1} = \sigma_1$, we see that
\begin{equation} 
\label{eqn: rewrite 1}
b = \delta_n q_1 \sigma_1^{d_1} q_2 \sigma_1^{d_2}\ldots \sigma_1^{d_{k-1}} \delta_n p_k a_{1n}^{d_k}p_{k+1} \ldots p_r a_{1n}^{d_r}p_{r+1}
\end{equation} 
where each $q_i$ is a positive braid.

Next conjugate the lead $\delta_n$ in (\ref{eqn: rewrite 1}) to the``back of $b$" and then conjugate it backward through the rewritten $P$ till just before $a_{1n}^{d_k}$. Since
$p_{k+1} p_{k+2} \ldots p_{r+1}$ does not contain $\sigma_1$ and since pushing $\delta_n^{-1} a_{1n} \delta_n = \sigma_{n-1}$, we see that
$$b \sim q_1 \sigma_1^{d_1} q_2 \sigma_1^{d_2}\ldots \sigma_1^{d_{k-1}} \delta_n p_k \delta_n \sigma_{n-1}^{d_k} q_{k+1} \ldots q_r \sigma_{n-1}^{d_r} q_{r+1}$$
where $q_{k+1}, q_{k+2}, \ldots$ are positive braids. This completes the proof when  there is a $1 \leq k \leq r$ such that $p_{k}$ contains $\sigma_1$.

Suppose then that $p_2 p_3 \ldots p_{r}$ does not contain $\sigma_1$. If it is also not contained in $p_{r+1}$, conjugate one copy of $\delta_n$ to the back of $b$ and then backward through $P$ till just before $a_{1n}^{d_1}$. As above, this operation yields a positive braid. On the other hand, if $p_{r+1}$ does contain $\sigma_1$, then $p_1 p_2 \ldots p_{r}$ cannot contain $\sigma_{n-1}$ as otherwise $P$ contains $\sigma_{n-1} a_{1n} \sigma_1$. Then conjugating one copy of $\delta_n$ forward through $P$ till just after $a_{1n}^{d_r}$ rewrites $b$ in the form $b = \delta_n q_1 \sigma_1^{d_1} q_2 \sigma_1^{d_2}\ldots q_r \sigma_1^{d_{r}} \delta_n p_{r+1}$ where each $q_i$ is a positive braid, thus completing the proof. 
\end{proof}

\begin{proof}[Proof of Proposition \ref{prop: definite strongly quasipositive n-braids}] 
We can assume that $P$ contains a letter of span $l \not \in \{1, n-1\}$ by Lemma \ref{lemma: done if all spans 1 or n-1}. We can further suppose that the spans of the letters of $P$ lie in $\{1, 2, n-2, n-1\}$. This is obvious if $n = 5$. If $n \geq 6$ and $P$ contains a letter of span $3 \leq l \leq n-3$, Lemmas \ref{lemma: small span} and \ref{lemma: the case e6} imply that Proposition \ref{prop: definite strongly quasipositive n-braids} holds. 

We can also assume that $P$ contains at most one letter of span $2$ and if one, it occurs exactly once by Lemmas \ref{lemma: no linking}, \ref{lemma: when commuting implies indefinite}, \ref{lemma: 2-step descending stairs of length two or more indefinite}, \ref{lemma: no repeated span 2 letters}, \ref{lemma: no ascending stairs of total length 3 or 4}, and \ref{lemma: the case e6}. 

If $P$ contains two letters (possibly equal) of span $n - 2$ they cannot be distinct as otherwise the assumption that $n \geq 5$ implies that they are linked, contrary to Lemma \ref{lemma: no linking}. Thus they coincide. But then, we could conjugate by an appropriate power of $\delta_n$ to obtain $b' = \delta_n^2 P'$ where $P'$ contains the square of a span $2$ letter. Then by Lemmas \ref{lemma: no repeated span 2 letters} and \ref{lemma: the case e6} we are done. Assume then that $P$ does not contain two letters (possibly equal) of span $n - 2$. 

If $P$ contains a letter of span $2$ and one of span $n-2$, conjugate $b$ by an appropriate power of $\delta_n$ so that the span $n-2$ letter is $a_{1, n-1}$. The span $2$ letter now has been converted to a letter $a_{rs}$ of span either $2$ or $n-2$. We have just seen that the latter is impossible so $a_{rs}$ has span $2$. Lemma \ref{lemma: no linking} now implies that $s \leq n-1$. The reader will verify that conjugating by either $\delta_n$ or $\delta_n^{-1}$ yields $P'$ with distinct letters of span $2$, so we are done as in the second paragraph of this proof. 

Assume then that $P$ does not contain both a letter of span $2$ and one of span $n-2$. By assumption, it contains at least one such letter, so after conjugating by a suitable power of $\delta_n$ we can suppose that $P$ contains $a_{13}$. All other letters of $P$ have span $1$ or $n-1$. Thus $b = \delta_n^2 P_1 a_{13} P_2$ where the letters of $P_1$ and $P_2$ have span $1$ or $n-1$. Then 
$$b = \delta_n^2 P_1 a_{13} P_2 \sim  P_1 a_{13} P_2 \delta_n^2 = P_1 a_{13} \delta_n^2 P_2' \sim  \delta_n^2 P_2' P_1 a_{13} =  \delta_n^2 P' a_{13}$$
where the letters of $P'$ are of span $1$ or $n-1$. If $P'$ contains $a_{1n}$, then it contains $a_{1n} a_{13} = a_{13} a_{3n}$. Since the span of $a_{3n}$ is $n-3$, Lemmas \ref{lemma: small span} and \ref{lemma: the case e6} imply that we are done when $n \geq 6$. On the other hand, if $n = 5$, $a_{13} a_{3n}$ is an ascending $(2,2)$ pair. Hence we are done by Lemma \ref{lemma: no ascending stairs of total length 3 or 4} and \ref{lemma: the case e6}. Thus without loss of generality, $P'$ is a positive braid. But then 
$$b \sim \delta_n^2 P' a_{13} \sim \delta_n P' a_{13} \delta_n = \delta_n P' \sigma_1 \sigma_2^2 \sigma_3 \cdots \sigma_{n-1},$$ 
a positive braid. This completes the proof.
\end{proof}

\section{Cyclic basket links} \label{sec: cycle baskets}

A basket surface $F$ is determined by a finite set $\mathcal{A}$ of arcs properly embedded in the disk $D^2$, together with an ordering $\omega$ of $\mathcal{A}$. Let the arcs be ordered $\alpha(1), \alpha(2),...,\alpha(m)$. Then $F = F(\mathcal{A}, \omega)$ is obtained by successively plumbing positive Hopf bands $b(1), b(2),...,b(m)$, where $b(i)$ is plumbed along a neighborhood of the arc $\alpha(i), 1 \le i \le m$. We adopt the convention that each plumbing is a {\it bottom plumbing} \cite{Ru1}, so $b(i)$ lies underneath $D^2 \cup b(1) \cup ... \cup b(i-1)$.

Baskets are also considered in \cite{Hir}; we will adopt the terminology of \cite{Hir} and call the collection of arcs $\mathcal{A}$ in the disk a {\it chord diagram}. These are considered up to the equivalence relation generated by isotopy of an arc $\alpha$, keeping $\partial \alpha$ in $S^1$ and disjoint from the endpoints of the other arcs. In particular we may assume that any two arcs are either disjoint or intersect transversely in a single point. As in \cite{Hir}, $\mathcal{A}$ determines a graph $\Gamma = \Gamma(\mathcal{A})$, the {\it incidence graph} of $\mathcal{A}$: the vertices of $\Gamma$ correspond bijectively to the arcs in $\mathcal{A}$, and two vertices are joined by an edge if and only if the corresponding arcs intersect.

If $\Gamma$ is a tree then a planar embedding of $\Gamma$ determines a unique chord diagram $\mathcal{A}$, and $F(\mathcal{A})$ is the corresponding {\it arborescent plumbing}; this is clearly independent of the ordering $\omega$.

Recall the simply laced arborescent links $L(A_m), L(D_m), L(E_6), L(E_7)$, and $L(E_8)$ from the introduction. We note that 
$$L(D_4) = P(-2,2,2) = T(3,3)$$ 
Also, the expression for $L(D_m)$ continues to hold for $m = 3$: $D_3 = A_3$ and 
$$L(D_3) = P(-2,2,1) = T(2,4) = L(A_3)$$

\begin{problem}
{\rm Determine the definite baskets.}
\end{problem} 

We will see that there are prime definite basket links that are not simply laced arborescent.

We discuss the case where $\Gamma = C_m$ is an $m$-cycle, $m \ge 3$. As observed in \cite{Hir}, $\Gamma$ has a unique realization as a chord diagram $\mathcal{A}$. Let $\omega$ be an ordering of $\mathcal{A}$. We encode $\omega$ as follows. Number the arcs $\alpha_1,\alpha_2,...,\alpha_m$, and the corresponding bands $b_1,b_2,...,b_m$, in clockwise order on $D^2$, and define $\boldsymbol{\epsilon} = (\epsilon_1,...,\epsilon_m) \in \{\pm 1\}^n$ as illustrated in Figure \ref{fig: cam 1}. Thus $\boldsymbol{\epsilon}$ determines the order $\omega$ in which the bands are plumbed. Denote the corresponding basket by $F(C_m, \boldsymbol{\epsilon})$, the corresponding link by $L(C_m, \boldsymbol{\epsilon})$, and the symmetrised Seifert form of $F(C_m, \boldsymbol{\epsilon})$ by $\mathcal{F}(C_m, \boldsymbol{\epsilon})$.

\begin{figure}[!ht] 
\centering
 \includegraphics[scale=0.75]{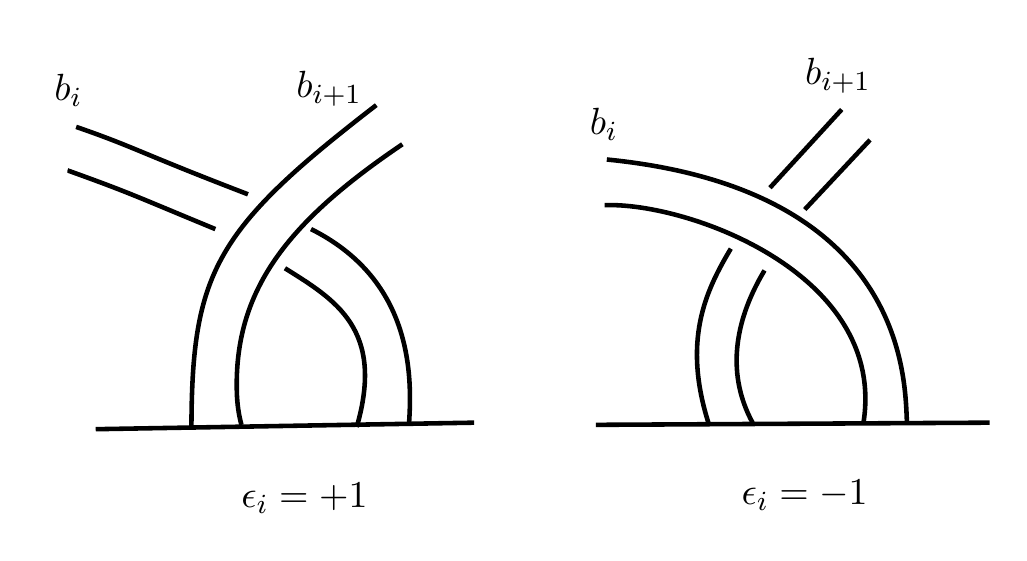}
\caption{} 
\label{fig: cam 1}
\end{figure} 

The link $L(C_m, \boldsymbol{\epsilon})$ has two components if $m$ is odd and three components if $m$ is even.

If $L$ is an oriented link with components $K_1, K_2,..., K_n, n \ge 2$, define 
$$lk(L) = \sum_{1 \le i < j \le n} lk(K_i,K_j)$$
In particular, 
$$lk(L(D_m)) = lk(P(-2,2,m-2)) = \left\{ 
\begin{array}{ll}
2 & \hbox{if $m$ is odd,} \\ 
\frac{m}{2} + 1  & \hbox{if $m$ is even.} 
\end{array} \right.$$ 

Let $p(\boldsymbol{\epsilon})$ be the number of $i$ such that $\epsilon_i = +1$. Note that if $F(C_m, \boldsymbol{\epsilon})$ is a basket then $1 \le p(\boldsymbol{\epsilon}) \le m-1$.

\begin{lemma} \label{lemma: linking}
$lk(L(C_m,\boldsymbol{\epsilon}))  = \left\{ \begin{array}{ll}
2p(\boldsymbol{\epsilon}) & \hbox{if $m$ is odd,} \\ 
\frac{m}{2} + p(\boldsymbol{\epsilon}) & \hbox{if $m$ is even.} 
\end{array} \right.$ 
\end{lemma}

\begin{proof} 
Let $L = L(C_m, \boldsymbol{\epsilon})$. Each band contributes $+1$ to $lk(L)$.

Consider a crossing of a pair of adjacent bands $b_i$ and $b_{i+1}$. First assume that $m$ is odd. Then $L$ has two components, $K_1, K_2$ say, and the crossing contributes $\epsilon_i$ to $lk(L)$; see Figure \ref{fig: cam 2}. Hence
\begin{eqnarray} lk(L) & = & m + (\hbox{number of $i$ such that $\epsilon_i = +1$) - (number of $i$ such that $\epsilon_i = -1$)} \nonumber \\ 
& = & m + p(\boldsymbol{\epsilon}) - (m - p(\boldsymbol{\epsilon})) \nonumber \\ 
& = & 2p(\boldsymbol{\epsilon}) \nonumber 
\end{eqnarray} 

\begin{figure}[!ht]
\centering
 \includegraphics[scale=0.75]{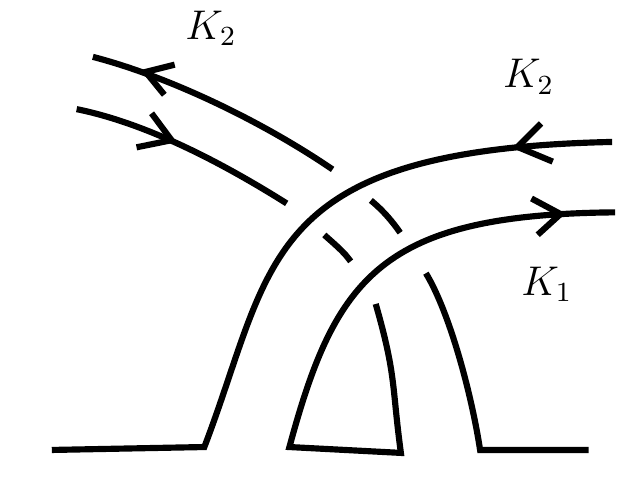}
\caption{} 
\label{fig: cam 2}
\end{figure} 

If $m$ is even then all three components, $K_1, K_2, K_3$ say,  appear at the band crossing and the contribution to $lk(L)$ is $\epsilon_{i}/2$; see Figure \ref{fig: cam 3}. Hence
$$lk(L) = m + (p(\boldsymbol{\epsilon}) - (m - p(\boldsymbol{\epsilon}))/2 = m/2 + p(\boldsymbol{\epsilon}).$$

\begin{figure}[!ht]
\centering
 \includegraphics[scale=0.75]{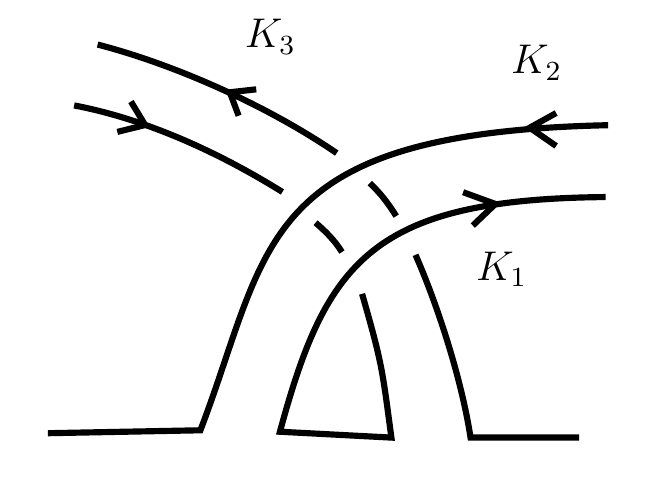}
\caption{} 
\label{fig: cam 3}
\end{figure} 
\end{proof}

\begin{prop} \label{prop: only depends on p}
The baskets $F(C_m, \boldsymbol{\epsilon})$ and $F(C_m, \boldsymbol{\epsilon}')$ are isotopic if and only if $p(\boldsymbol{\epsilon}) = p(\boldsymbol{\epsilon}')$.
\end{prop} 

\begin{proof} The ``only if" direction follows from Lemma \ref{lemma: linking}.

For the``if" direction, note that replacing a top plumbing by a bottom plumbing does not change the isotopy class of the surface \cite[Lemma 3.2.1]{Ru1}. If $b_i$ is top plumbed and is replaced by a bottom plumbing then the effect on $\boldsymbol{\epsilon}$ is that $(\epsilon_{i-1}, \epsilon_i)$ changes from $(+1,-1)$ to $(-1,+1)$. After a sequence of such moves we can bring $\boldsymbol{\epsilon}$ to the form $(-1,...,-1,+1,...,+1)$.
\end{proof}

In view of Proposition \ref{prop: only depends on p} we write $F(C_m,p)$ for $F(C_m, \boldsymbol{\epsilon})$ where $p = p(\boldsymbol{\epsilon})$, and similarly for $L(C_m,p)$ and $\mathcal{F}(C_m, p)$.

We can think of the basket $F(\mathcal{A}, \omega)$ as being obtained from $D^2$ by attaching 1-handles $h_1,...,h_m$, where $h_i = \overline{b_i \setminus D^2}$. We then have the following {\it handle-sliding} move, which does not change the isotopy class of the surface. Let $h$ and $h'$ be 1-handles such that no 1-handle meets the interior of the interval $I$ shown in Figure \ref{fig: cam 4}. 

\begin{figure}[!ht]
\centering
 \includegraphics[scale=0.75]{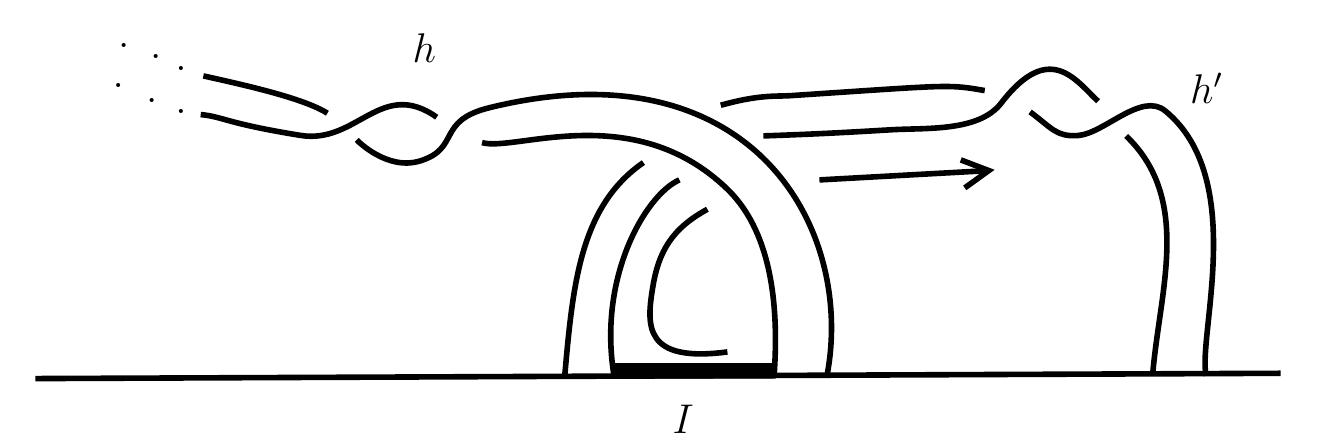}
\caption{} 
\label{fig: cam 4}
\end{figure} 

Then $h$ can be slid over $h'$ as indicated in Figure \ref{fig: cam 4}, resulting in Figure \ref{fig: cam 5}.

\begin{figure}[!ht]
\centering
 \includegraphics[scale=0.75]{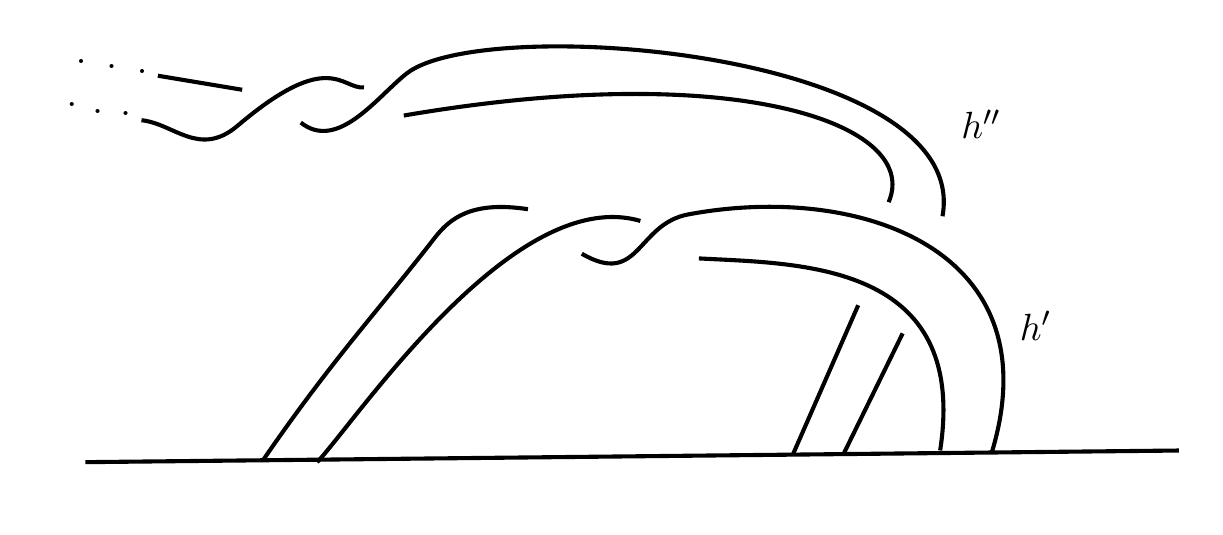}
\caption{} 
\label{fig: cam 5}
\end{figure} 
 
Note that the core of the band corresponding to $h''$ still has framing $-1$. 

\begin{thm} \label{thm: p = 1} 
If $p = 1$ then $F(C_m,p)$ is isotopic to $F(D_m)$.
\end{thm}

\begin{proof} 
Without loss of generality $\epsilon_m = +1, \epsilon_i = -1$ for  $1 \le i \le m-1$. The surface $F(C_m, \boldsymbol{\epsilon})$ is illustrated in Figure \ref{fig: cam 6}, which shows the case $m = 7$.

\begin{figure}[!ht] 
\centering
 \includegraphics[scale=0.75]{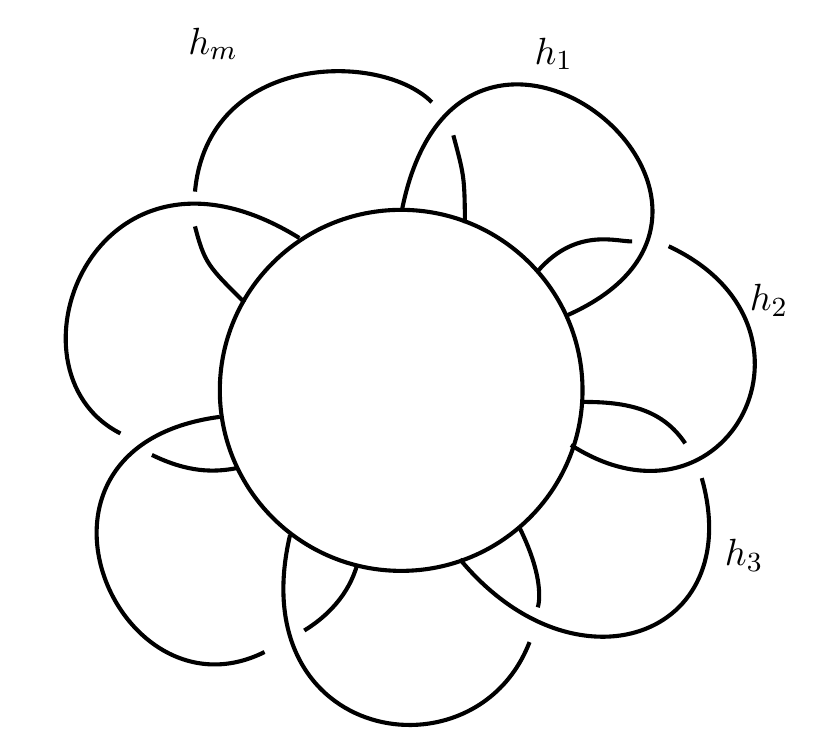}
\caption{} 
\label{fig: cam 6}
\end{figure} 

Slide $h_1$ over $h_2$, then over $h_3$,..., then over $h_{m-1}$. The resulting surface is shown in Figure \ref{fig: cam 7}. 

\begin{figure}[!ht] 
\centering
 \includegraphics[scale=0.75]{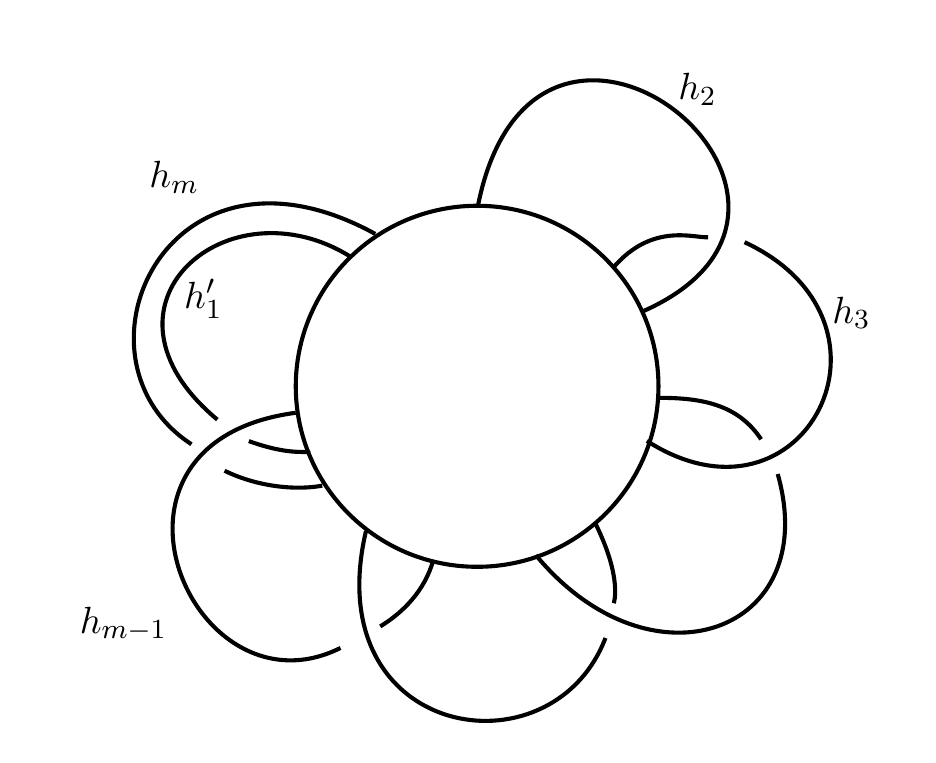}
\caption{} 
\label{fig: cam 7}
\end{figure} 

The corresponding chord diagram has incidence graph $D_m$.
\end{proof}

We now determine which of the links $L(C_m, p)$ are definite (cf. Lemma \ref{lemma: definite iff p odd}). 

Let $x_i$ be the core of $b_i, 1 \le i \le m$, with the anticlockwise orientation, and let $x_i^{+}$ be a copy of $x_i$ pushed slightly off $F = F(C_m, \boldsymbol{\epsilon})$ in the positive normal direction, i.e. the direction from the disk towards the reader. Let $S = (s_{ij})$ be the Seifert matrix of $F$ with respect to the basis $[x_1],...,[x_m]$ for $H_1(F)$. Note that 
$$s_{i, i+1} = lk(x_{i}^+, x_{i+1}) = \left\{ \begin{array}{l} -1 \;\; \hbox{ if } \epsilon_i = +1 \;\;\; \\ 0 \;\; \hbox{ if } \epsilon_i = -1 \end{array} \right.$$
and 
$$s_{i+1, i} = lk(x_{i+1}^+, x_i) = \left\{ \begin{array}{l} 0 \;\; \hbox{ if } \epsilon_i = +1 \\ \;\;\; 1 \;\; \hbox{ if } \epsilon_i = -1 \end{array} \right.$$
Also, $s_{ii} = -1$, and $s_{ij} = 0$ if $|i-j| > 1$ (where subscripts are interpreted modulo $m$).

Let $p = p(\boldsymbol{\epsilon})$, so $0 < p < m$. By Proposition \ref{prop: only depends on p} we may assume that 
$$\epsilon_i = \left\{ \begin{array}{l} -1 \;\; \hbox{ if } 1 \leq i \leq m-p \\ +1 \;\; \hbox{ if } m-p < i \leq m \end{array} \right.$$
Thus the matrix $S$ is 

$$S = \left(
\begin{array}{rrrrrrrr|rrrrrrr} 
-1 & 0 & 0 & 0 & \cdots  & 0 & 0 &  & 0 & 0 & \cdot & \cdot &\cdots & \cdot &   0 \\ 
1 & -1 & 0 & 0 & \cdots & 0 & 0 &  & 0 & 0  & \cdot & \cdot & \cdots &\cdot &    0 \\ 
0 & 1 & -1& 0 & \cdots & 0 & 0 &  & 0 & 0  & \cdot & \cdot & \cdots &\cdot &    0 \\ 
\cdot & \cdot & \cdot &\cdot & \cdots & \cdot & \cdot &  & 0 & 0  & \cdot & \cdot & \cdots &\cdot &    0 \\ 
\cdot & \cdot & \cdot& \cdot& \cdots&\cdot & \cdot  &  & 0 & 0  & \cdot & \cdot & \cdots &\cdot &    0 \\ 
0 & 0 & 0 & 0 &\cdots  & 1 & -1  &  & 0 & 0 & \cdot & \cdot & \cdots &\cdot &    0 \\ 
\hline
0 & 0 & 0 & 0 &\cdots  & 0 & 1  &  & -1 & -1  & 0 & \cdot & \cdots &\cdot &    0 \\ 
0 & 0 & 0 & 0 &\cdots  & 0 & 0  &  & 0 & -1  & -1 & 0 & \cdots &\cdot &    0 \\ 
\cdot & \cdot & \cdot &\cdot & \cdots & \cdot & \cdot &  & \cdot& \cdot  & \cdot & \cdot & \cdots &\cdot &    0 \\ 
\cdot & \cdot & \cdot &\cdot & \cdots & \cdot & \cdot &  & \cdot & \cdot  & \cdot & \cdot & \cdots &\cdot &    0 \\ 
\cdot & \cdot & \cdot& \cdot& \cdots&\cdot & \cdot  &  & \cdot & \cdot & \cdot & \cdot & \cdots &-1 &  -1 \\ 
-1 & 0 & 0 & 0 &\cdots  & 0 & 0  &  & 0 & 0 & \cdot & \cdot & \cdots & 0 &    -1 \\ 
\end{array}\right)$$

Then

$$-(S + S^T) = \left(
\begin{array}{rrrrrrrr|rrrrrrrr}  
2 & -1 & 0 & 0 & \cdots  & 0 & 0 &  & 0 & 0 & \cdot & \cdot &\cdots & \cdot &  \cdot & 1 \\ 
-1 & 2 & -1 & 0 & \cdots & 0 & 0 &  & 0 & 0  & \cdot & \cdot & \cdots &\cdot & \cdot &   0 \\ 
0 & -1 & 2& -1 & \cdots & 0 & 0 &  & 0 & 0  & \cdot & \cdot & \cdots &\cdot &  \cdot &  0 \\ 
\cdot & \cdot & \cdot &\cdot & \cdots & \cdot & \cdot &  & 0 & 0  & \cdot & \cdot & \cdots &\cdot &  \cdot &  0 \\ 
\cdot & \cdot & \cdot& \cdot& \cdots&\cdot & -1 &  & 0 & 0  & \cdot & \cdot & \cdots &\cdot &  \cdot &  0 \\ 
0 & 0 & 0 & 0 &\cdots  & -1 & 2  &  & -1 & 0 & \cdot & \cdot & \cdots &\cdot &  \cdot &  0 \\ 
\hline
0 & 0 & 0 & 0 &\cdots  & 0 & -1  &  & 2 & 1  & 0 & \cdot & \cdots &\cdot &    \cdot &0 \\ 
0 & 0 & 0 & 0 &\cdots  & 0 & 0  &  & 1  & 2  & 1 & 0 & \cdots &\cdot & \cdot &   0 \\ 
\cdot & \cdot & \cdot &\cdot & \cdots & \cdot & \cdot &  & \cdot& \cdot  & \cdot & \cdot & \cdots &\cdot &   \cdot & 0 \\ 
\cdot & \cdot & \cdot &\cdot & \cdots & \cdot & \cdot &  & \cdot & \cdot  & \cdot & \cdot & \cdots &\cdot &  \cdot &  0 \\ 
\cdot & \cdot & \cdot& \cdot& \cdots&\cdot & \cdot  &  & \cdot & \cdot & \cdot &\cdot & \cdots & 1 &2 &  1 \\ 
1 & 0 & 0 & 0 &\cdots  & 0 & 0  &  & 0 & 0 & \cdot & \cdot & \cdots & 0 &1 &    2 \\ 
\end{array}\right)$$

For $\varepsilon = \pm 1, m \ge 3$, let $Q_m(\varepsilon)$ be the $m \times m$ symmetric matrix

$$Q_m(\varepsilon) = \left(\begin{matrix} 2 & 1 & 0 & 0 & \cdots  & 0 & \varepsilon \\ 
1 &2 & 1 & 0 & \cdots & 0 & 0 \\ 
0 & 1 & 2& 1 & \cdots & 0 & 0 \\
\cdot & \cdot & \cdot &\cdot & \cdots& \cdot & \cdot \\ 
\cdot & \cdot & \cdot& \cdot& \cdots&\cdot & \cdot  \\
0 & 0 & 0 & 0 &\cdots  & 2 & 1 \\
\epsilon & 0 & 0 & 0 &\cdots  & 1 & 2  \end{matrix}\right)$$

Successively multiplying the $i^{th}$ row and column of $-(S+S^T)$ by $-1$ for 
$$i = \left\{ \begin{array}{l} m-p, m-p-2, \ldots, 1 \;\; \hbox{ if $m-p$ is odd} \\ m-p, m-p-2, \ldots, 2 \;\; \hbox{ if $m-p$ is even}  \end{array} \right.$$ 
we see that $-(S+S^T)$ is congruent to $Q_m((-1)^{m-p})$.

One easily shows by induction on $r$ that the $r \times r$ leading minor of $Q_m(\varepsilon)$ is $r+1$ for $1 \leq r \leq m-1$. A straightforward calculation also gives the following.

\begin{lemma} \label{lemma: det qme} 
$$\det Q_m(\varepsilon) = \left\{ \begin{array}{l} 2 + 2\varepsilon\;\; \hbox{ if $m$ is odd} \\ 2 - 2\varepsilon \;\; \hbox{ if $m$ is even}  \end{array} \right. \qed$$ 
\end{lemma}

It follows that $\mathcal{F}(C_m, p)$ depends only on the parity of $p$ and is definite if and only if $p$ is odd. Recalling Theorem \ref{thm: p = 1} we therefore have the following.

\begin{lemma} \label  {lemma: definite iff p odd} $\;$ 

$(1)$ If $p$ is odd then $\mathcal{F}(C_m, p) \cong D_m$. 

$(2)$  If $p$ is even then $\mathcal{F}(C_m, p)$ is indefinite.
\qed
\end{lemma}

We now prove Theorem \ref{thm: p odd} which we restate here for the reader's convenience. 

{\bf Theorem \ref{thm: p odd}}. {\it Let $m \geq 3$ and $p$ be integers with $p$ odd and $0 < p < m$. 

$(1)$ $L(C_m, p)$ is prime and definite. 

$(2)$ $L(C_m, p)$ is simply laced arborescent if and only if $p = 1$. }

\begin{proof} 
The definiteness of $L(C_m,p)$ in part (1)  follows from Lemma \ref{lemma: definite iff p odd} 
(1), and the primeness of $L(C_m,p)$ follows from Lemma \ref{lemma: definite iff p odd}(1) and
the fact that $D_m$ is indecomposable. To prove part (2), we have that 
$L(C_m,1) = L(D_m)$ by Theorem \ref{thm: p = 1}. Conversely, if $L(C_m,p)$ is simply 
laced arborescent then it is definite, so $p$ is odd by Lemma \ref{lemma: definite iff p odd}, and 
therefore $F(C_m,p)$ is isotopic to $F(D_m)$. By Lemma \ref{lemma: linking} this implies 
that $p = 1$.
\end{proof}

Next we prove Theorem \ref{thm: intro not a counterexample}, which shows that the links in Theorem \ref{thm: p odd} do not provide counterexamples to 
Conjecture \ref{conj: branched lspace implies simply laced arborescent}.

\begin{lemma} \label{lemma: alex poly} 
The Alexander polynomial of $L(C_m, p)$ is 
$$\Delta_{L(C_m, p)} = \left\{ \begin{array}{l} (t^p - 1)(t^{(m-p)} + 1) \;\; \hbox{ if $m$ is odd} \\ (t^p - 1)(t^{(m-p)} - 1)  \;\; \hbox{ if $m$ is even}  \end{array} \right.$$ 
\end{lemma}

\begin{proof}
The matrix $S - tS^T$ (immediately below) is a Seifert matrix for $L(C_m,p)$ whose determinant is $\Delta_{L(C_m, p)}$. 

$$\left(
\begin{array}{cccccccc|cccccccc} 
t-1 & -t & 0 & 0 & \cdots  & 0 & 0 &  & 0 & 0 & \cdot & \cdot &\cdots & \cdot& \cdot &   t  \\ 
1 & t-1& -t & 0 & \cdots & 0 & 0 &  & 0 & 0  & \cdot & \cdot & \cdots &\cdot & \cdot&    0 \\ 
0 & 1 & t-1& -t & \cdots & 0 & 0 &  & 0 & 0  & \cdot & \cdot & \cdots &\cdot & \cdot&    0 \\ 
\cdot & \cdot & \cdot &\cdot & \cdots & \cdot & \cdot &  & 0 & 0  & \cdot & \cdot & \cdots &\cdot & \cdot&    0 \\ 
\cdot & \cdot & \cdot& \cdot& \cdots&\cdot & -t &  & 0 & 0  & \cdot & \cdot & \cdots &\cdot & \cdot&    0 \\ 
0 & 0 & 0 & 0 &\cdots  & 1 & t-1  &  & -t & 0 & \cdot & \cdot & \cdots &\cdot & \cdot&    0 \\ 
\hline
0 & 0 & 0 & 0 &\cdot  & 0 & 1  &  & t-1 & -1  & 0 & \cdot & \cdots &\cdot & \cdot&    0 \\ 
0 & 0 & 0 & 0 &\cdot  & 0 & 0  &  & t & t-1  & -1 & 0 & \cdots &\cdot & \cdot&    0 \\ 
\cdot & \cdot & \cdot &\cdot & \cdot & \cdot & \cdot &  & \cdot& \cdot  & \cdot & \cdot & \cdots &\cdot & \cdot&    0 \\ 
\cdot & \cdot & \cdot &\cdot & \cdot & \cdot & \cdot &  & \cdot & \cdot  & \cdot & \cdot & \cdots &\cdot & \cdot&    0 \\ 
\cdot & \cdot & \cdot& \cdot& \cdot&\cdot & \cdot  &  & \cdot & \cdot & \cdot & \cdot & \cdots & t &t-1 &  -1 \\ 
-1 & 0 & 0 & 0 &\cdot  & 0 & 0  &  & 0 & 0 & \cdot & \cdot & \cdots & 0 & t &    t-1 \\ 
\end{array}\right)$$

Let $M_1, M_2$, and $M_3$ be the matrices obtained by removing from $S - tS^T$, respectively, the first row and column, the first two rows and columns, and the first and last rows and columns. Correspondingly, let $F_1, F_2$, and $F_3$ be the surfaces obtained by removing, respectively, the first band $b_1$, the first and second bands $b_1$ and $b_2$, and the first and last bands $b_1$ and $b_m$. Then for $i = 1, 2,3$, $M_i = R_i - tR_i^T$ where $R_i$ is a Seifert matrix for $F_i$. Also, $F_1$ is the arborescent basket $F(A_{m-1})$, and $F_2$ and $F_3$ are the arborescent basket $F(A_{m-2})$, the corresponding links being $T(2,m)$ and $T(2,m-1)$ respectively.

Expanding by the first column we obtain $\Delta_{L(C_m,p)} = \det(S - tS^T) = (t-1) \det M_1 - \det U + (-1)^m \det V$. Here $U = \left(\begin{matrix}   -t & C_1  \\ C_2^T & M_2  \end{matrix}\right)$ where $C_1 = (0, 0, \ldots, 0, t)$ and $C_2 = (1, 0, \ldots , 0)$, and  $V = \left(\begin{matrix}  D_1 & t  \\ M_3 & D_2^T  \end{matrix}\right)$ where $D_1 = (-t, 0, \ldots, 0, 0)$ and $D_2 = (0, 0, \ldots , 0, -1)$

Expanding $\det U$ by the first row we get 
$$\det U = (-t) \det M_2 + (-1)^{m} t^p$$ 
Similarly, expanding $\det V$ by the last column we get 
$$\det V = (-1)^m t \det M_3 + (-1)^{m-1} t^{m-p}$$ 
Since $\det M_1= \Delta_{T(2,m)}$ and $\det M_2 = \det M_3 = \Delta_{T(2,m-1)}$, we obtain 
$$\Delta_{L(C_m, p)} = (t-1) \Delta_{T(2,m)} + 2t \Delta_{T(2,m-1)} + (-1)^m t^p - t^{m-p}$$
Now 
$$\Delta_{T(2,r)} = \left\{ \begin{array}{l} (t^r + 1)/(t + 1) \;\; \hbox{ if $r$ is odd} \\ (t^r - 1)/(t + 1)  \;\; \hbox{ if $r$ is even}  \end{array} \right.$$ 
See for example \cite[Exercise 7.3.1]{Mu2}. 

First suppose that $m$ is even. Then $(t-1)\Delta_{T(2,m)} + 2t\Delta_{T(2,m-1)} = \big((t-1)(t^m - 1) + 2t(t^{m-1}+ 1)\big)/(t+1) = (t^{m+1} + t^m + t + 1)/(t+1) = t^m + 1$. Hence $\Delta_{L(C_m,p)} = t^m + 1 - t^p - t^{m-p} = (t^p - 1)(t^{m-p} - 1)$.

The case that $m$ is odd is similar; we leave the details to the reader.
\end{proof}

\begin{proof}[Proof of Theorem \ref{thm: intro not a counterexample}]
For each $\zeta \in S^1$ we define $I_+(\zeta)$, to be the open subarc of the circle with endpoints $\zeta, \bar \zeta$ which contains $+1$.  

If $\Sigma_n(L(C_m, p))$ is an L-space, then
\begin{equation} \label{eqn: roots} 
\hbox{{\it All the roots of $\Delta_{L(C_m, p)}$ lie in $I_+(\zeta_n)$}}
\end{equation}
by Theorem \ref{thm: bbg definite}.

If $p > 1$ and $n \geq 3$, the fact that $t^p - 1$ divides $\Delta_{L(C_m, p)}$ (cf. Lemma \ref{lemma: alex poly}) shows that (\ref{eqn: roots}) is violated.

Next suppose that $\Sigma_n(L(C_m, 1))$ is an L-space and recall that by Theorem \ref{thm: p = 1}, $L(C_m, 1) = P(-2, 2, m-2)$. Hence $\Sigma_2(L(C_m, 1)$ is an L-space. Also, if $m = 3$ then $L(C_3, 1) = P(-2, 2, 1)) = T(2,4)$, and so $\Sigma_3(L(C_3, 1))$ is also an L-space.

Conversely, first suppose that $m$ is even; then the roots of $\Delta_{L(C_m,1)}$ are the $(m-1)^{st}$ roots of unity. Since $m \geq 4$, (\ref{eqn: roots}) fails for $n \geq 3$. If $m$ is odd then the roots of $\Delta_{L(C_m, 1)}$ are the complex numbers $\exp(\frac{\pi i r}{m-1})$, $r$ odd, together with $1$. Therefore, if $m \geq 5$ then (\ref{eqn: roots}) fails  for $n \geq 3$, and if $m = 3$, it fails for $n \geq 4$. 

Finally suppose that $p > 1$ and $n = 2$.  

\begin{lemma} \label{lem: lcm braid}
As an unoriented link, $-L(C_m, p)$ is the closure of the $3$-braid $\gamma^m \delta^{3p}$, where $\delta = \sigma_2 \sigma_1, \gamma = \sigma_2 \delta^{-1}$. 
\end{lemma}

\begin{proof}
We can express $L(C_m, p)$ as a circular concatenation of tangles corresponding to pairs of adjacent bands, as in Figure \ref{fig: cam 14}. These tangles are the braids $\beta_{\pm}$ shown in Figure \ref{fig: cam 15}. 
\begin{figure}[!ht]
\centering
 \includegraphics[scale=0.75]{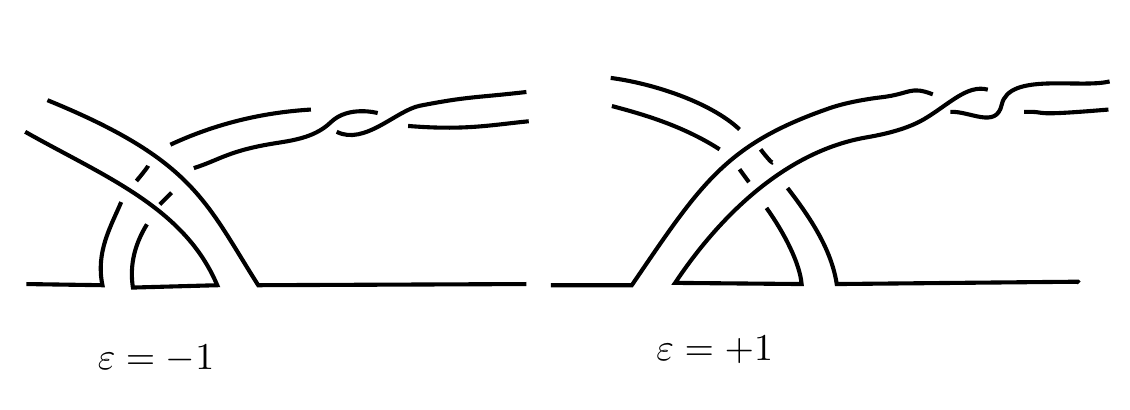}
\caption{} 
\label{fig: cam 14}
\end{figure} 

\begin{figure}[!ht] 
\centering
 \includegraphics[scale=0.75]{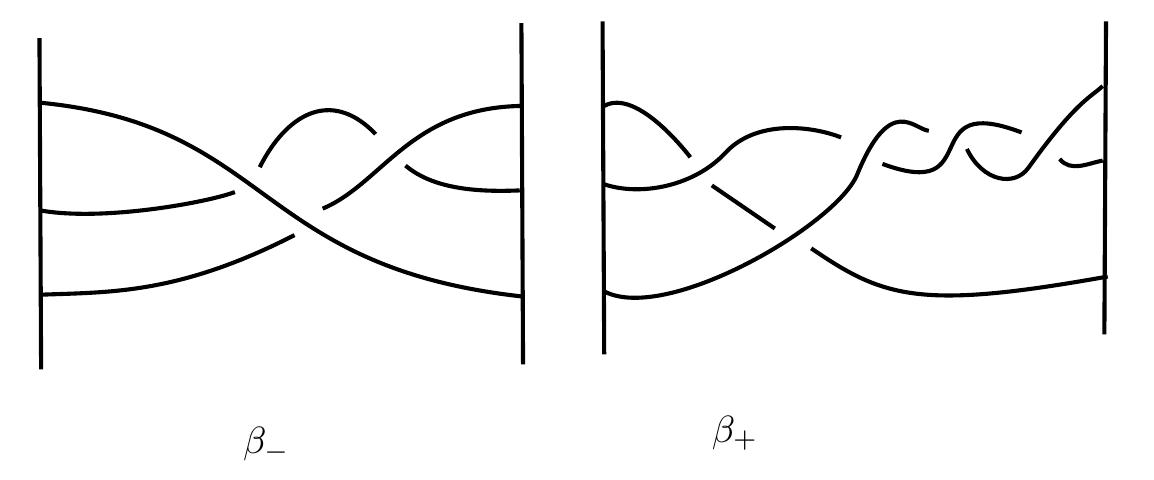}
\caption{} 
\label{fig: cam 15}
\end{figure} 

Reading from right to left in Figure \ref{fig: cam 15}, we see that
$$\beta_- = \sigma_2 \sigma_1 \sigma_2^{-1}, \;\;\;\;\; \beta_+ = \sigma_2^{-1} \sigma_1^{-1} \sigma_2^{-3}$$
Taking inverses for convenience we have 
$$\beta_-^{-1} = \sigma_2 \sigma_1^{-1} \sigma_2^{-1} = \sigma_2 \delta^{-1} = \gamma$$
$$\beta_+^{-1} = \sigma_2^3 \sigma_1 \sigma_2 = \sigma_2^2 \sigma_1 \sigma_2 \sigma_1 = \sigma_2 \delta^2 = \gamma \delta^{3}$$
Since $\delta^3$ is central in $B_3$, the result follows.
\end{proof}

First note that for any link $L$, $\Sigma_2(L)$ is independent of the orientations of the components of $L$. 

By Lemma \ref{lem: lcm braid}, as an unoriented link, $-L(C_m, p)$ is the closure of the $3$-braid $\beta_p = \gamma^m \delta^{3p}$. Let $p = 2r+1, r>0$; then $\beta_p = \beta_1 \delta^{6r}$. 

The inverse image in $\Sigma_2(-L(C_m, p))$ of the braid axis of $\widehat \beta_p$ in $S^3$ is a simple closed curve whose exterior is fibred over $S^1$ with fibre a once-punctured torus $T$, the double branched cover of $(D^2, 3 \hbox{ points})$. 

Recall the natural identification of $B_3$ with a subgroup of the group of homeomorphisms of $(D^2, 3 \hbox{ points})$ which restrict to the identity on $\partial D^2$. Then there are simple closed curves $\alpha_1$ and $\alpha_2$ in Int($T$), meeting transversely in a single point, such that $\sigma_i$ lifts to the positive Dehn twist $\tau_i$ along $\alpha_i$, $i = 1, 2$. Thus $\beta_p$ lifts to $(\tau_2 \tau_1^{-1} \tau_2^{-1})^m (\tau_2 \tau_1)^p = \varphi_p$, say, the monodromy of the $T$-bundle in $\Sigma_2(-L(C_m, p))$. Note that $(\tau_2 \tau_1)^6$ is isotopic (rel $\partial T$) to $\tau_\partial$, the positive Dehn twist along $\partial T$. 

For a homeomorphism $\psi: T \to T$ such that $\psi|\partial T$ is the identity, let $c(\psi)$ denote the fractional Dehn twist coefficient of $\psi$; see \cite[\S 3]{HKM1}. 

\begin{claim}
$c(\varphi_1) = \frac12$.
\end{claim}

\begin{proof}
Since $\delta^3$ is central in $B_3$, $\varphi_1^2$ is the lift of $\gamma^{2m} \delta^6$, which is $(\tau_2 \tau_1^{-1} \tau_2^{-1})^{2m} \tau_\partial$. 

There is an essential properly embedded arc $a$ in $T$ such that $\tau_1(a) = a$. It follows that $c(\tau_1^{-2m}) = 0$ (see \cite{HKM1}). Hence $c((\tau_2 \tau_1^{-1} \tau_2^{-1})^{2m}) = c(\tau_1^{-2m}) = 0$. Therefore $c(\varphi_1^2) = 1$, implying that $c(\varphi_1) = \frac12$ (see \cite{HKM1}). 
\end{proof}

Since $\varphi_p = \varphi_1 \tau_\partial^r$, we have that $c(\varphi_p) = \frac12 + r > 1$. Therefore by \cite{Ro} (see \cite[Theorem 4.1]{HKM2}), $\Sigma_2(L(C_m, p))$ admits a co-oriented taut foliation.
\end{proof}

Having given a detailed analysis of the cyclic case we will not attempt a complete classification of definite baskets, but will content ourselves with a few remarks. 

Recall that a subgraph $\Gamma_0$ of a graph $\Gamma$ is {\it full} if $\Gamma_0$ contains every edge of $\Gamma$ whose endpoints lie in $\Gamma_0$.

Let $\mathcal{A}$ be a chord diagram, with ordering $\omega$, and $\mathcal{A}_0 \subset \mathcal{A}$ a subchord diagram, with ordering $\omega_0$ induced by $\omega$. If $\Gamma(\mathcal{A}_0)$ is a full subgraph of $\Gamma(\mathcal{A})$ then $F(\mathcal{A}_0, \omega_0)$ is a homologically injective subsurface of $F(\mathcal{A}, \omega)$. Hence, if $F(\mathcal{A}, \omega)$ is definite then so is $F(\mathcal{A}_0, \omega_0)$. This can be used to show that many baskets are indefinite. For example, if $F(\mathcal{A}, \omega)$ is definite then any full subtree of $\Gamma(\mathcal{A})$ must be simply laced arborescent.

As a simple application of this, let $C_{m,l}$ be an $m$-cycle with a leg of length $l$, $l \ge 1, m \ge 3$. See Figure \ref{fig: cam 8}, which shows $C_{5,2}$. 

\begin{figure}[!ht]
\centering
 \includegraphics[scale=0.75]{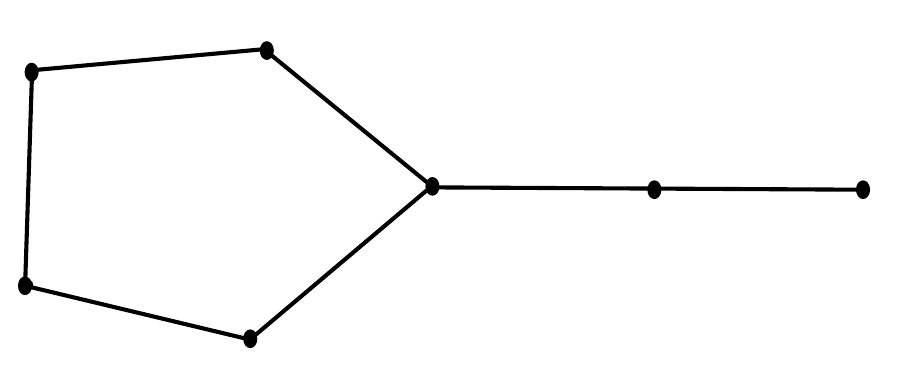}
\caption{} 
\label{fig: cam 8}
\end{figure} 

It is easy to see that any chord diagram realising $C_{m,l}$ is of the form shown in Figure \ref{fig: cam 20}.

\begin{figure}[!ht]
\centering
 \includegraphics[scale=0.75]{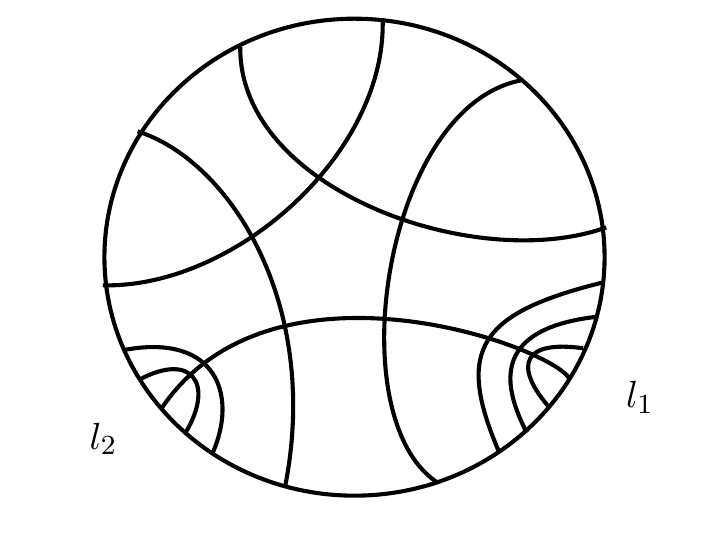}
\caption{} 
\label{fig: cam 20}
\end{figure} 

Thus if $l > 1$, $C_{m,l}$ does not have a unique realisation as a chord diagram. However, given an ordering of the arcs of the $m$-cycle $C_m$, the corresponding basket $F(C_{m,l})$ is clearly unique up to isotopy. 

\begin{thm}
Assume that $m \geq 3$ and $l \geq 1$. There is a definite basket with incidence graph $C_{m,l}$ if and only if $(m,l)$ is one of the following: $(3,l)$ or $(4,l)$ for $l \geq 1$; $(5,l)$ for $l = 1, 2$, or $3$; $(6, 1)$ or $(7, 1)$.
\end{thm} 

\begin{proof}
For the ``only if" direction, note that if $(m,l)$ is not one of the pairs listed then either $(m,l) = (5,4)$ or $C_{m,l}$ contains a full subtree that is not simply laced arborescent. In the case $(m,l) = (5,4)$, the corresponding basket $F$ is obtained by plumbing a positive Hopf band to $F(E_8)$ along a non-separating arc. Hence $F$ is indefinite by Lemma \ref{lemma: e* case}(3). 

For the ``if" direction, order the arcs of the $m$-cycle $C_m$ so that $p = 1$. By Theorem \ref{thm: p = 1}, $F(C_m, 1)$ is isotopic to $F(D_m)$. This isotopy is effected by sliding $h_1$ successively over $h_2, \ldots, h_{m-1}$. After these handle-slides, the handles $h_1$ and $h_m$ correspond to the vertices $v_1$ and $v_m$ of $D_m$ as shown in Figure \ref{fig: cam 21}.
\begin{figure}[!ht] 
\centering
 \includegraphics[scale=0.75]{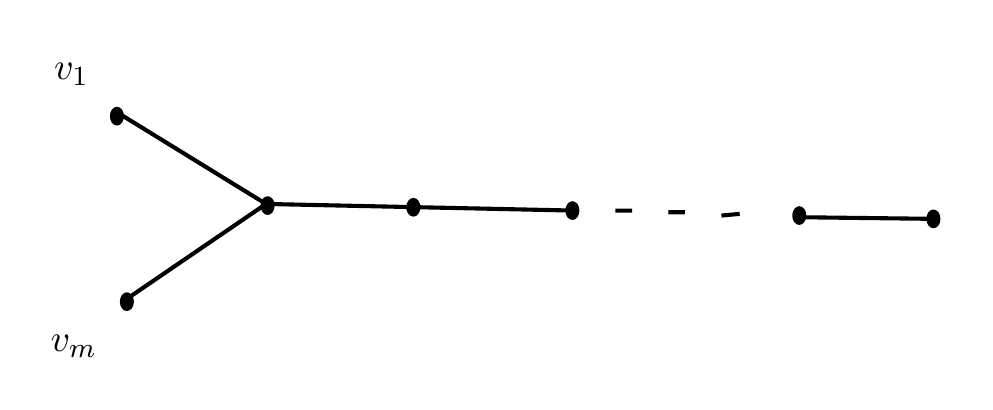}
\caption{} 
\label{fig: cam 21}
\end{figure}

Also, the transverse arcs (co-cores) of $h_1$ and $h_m$ are unaffected by the handle-slides. It follows that if we plumb a row of positive Hopf bands to $h_1$ (say), then the resulting surface is $F(\Gamma)$, where $\Gamma$ is the tree obtained by adjoining a line of $l$ edges to $D_m$ at $v_1$.

If $m = 3$ or $4$ then $\Gamma \cong D_{m+l}$, so $F(\Gamma)$ is definite. 

If $m = 5$ and $l = 1, 2,$ or $3$ then $\Gamma \cong E_6, E_7$, or $E_8$ respectively. Similarly if $m = 6$ or $7$ and $l = 1$ then $\Gamma \cong E_7$ or $E_8$. In all cases, $F(\Gamma)$ is definite. 
\end{proof}

We conclude with one more example. Let $K_m$ be the complete graph on $m$ vertices. It is easy to see that $K_m$ has a unique realization as a chord diagram $\mathcal{A}$, namely $m$ diameters of $D^2$.

\begin{thm} 
If $F(K_m, \omega)$ is definite then it is isotopic to $F(A_m)$.
\end{thm} 

\begin{proof}
If $m \leq 2$ then $K_m = A_m$ and there is nothing to prove. 

Suppose $m \ge 3$. Number the arcs in $\mathcal{A}, \alpha_1, \alpha_2,...,\alpha_m,$ in anticlockwise order around the disk. Let $\omega$ be an ordering of $\mathcal{A}$ such that $F(\mathcal{A}, \omega)$ is definite. Any 3-cycle in $K_m$ is full, and so by Lemma \ref{lemma: definite iff p odd}, up to cyclic renumbering of the $\alpha_i$'s, we must have $\alpha_1 > \alpha_2 > ... > \alpha_m$. The corresponding basket $F(K_m, \omega)$ is illustrated in Figure \ref{fig: cam 9}, which shows the case $m = 6$.

\begin{figure}[!ht]
\centering
 \includegraphics[scale=0.75]{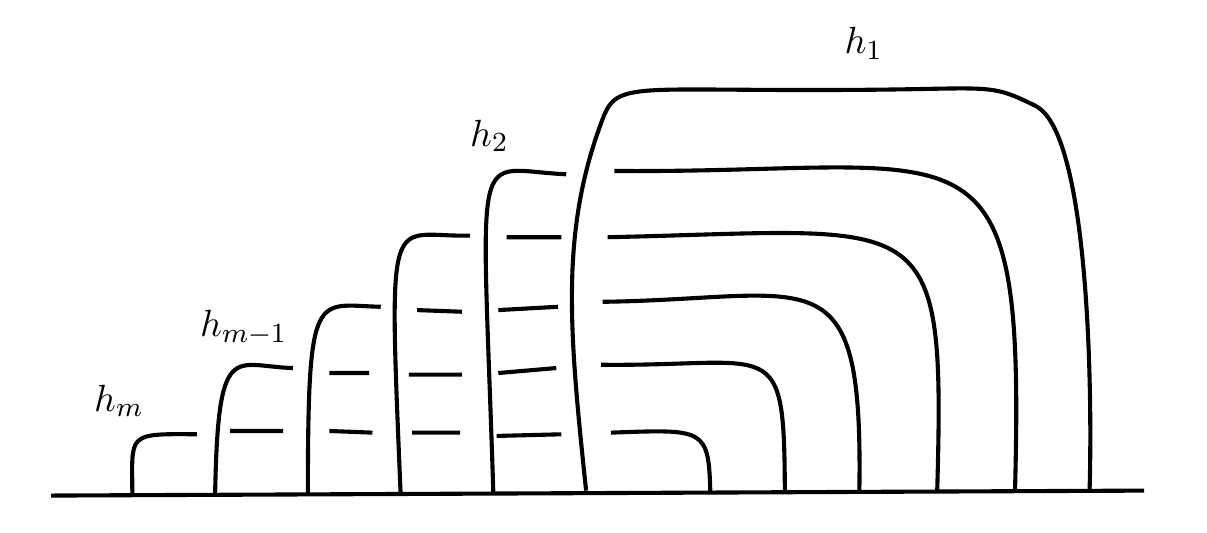}
\caption{} 
\label{fig: cam 9}
\end{figure} 

Slide $h_{m-1}$ over $h_m$. The resulting chord diagram has incidence graph $K_{m-1}$ with a leg of length 1; see Figure \ref{fig: cam 10}. 

\begin{figure}[!ht]
\centering
 \includegraphics[scale=0.75]{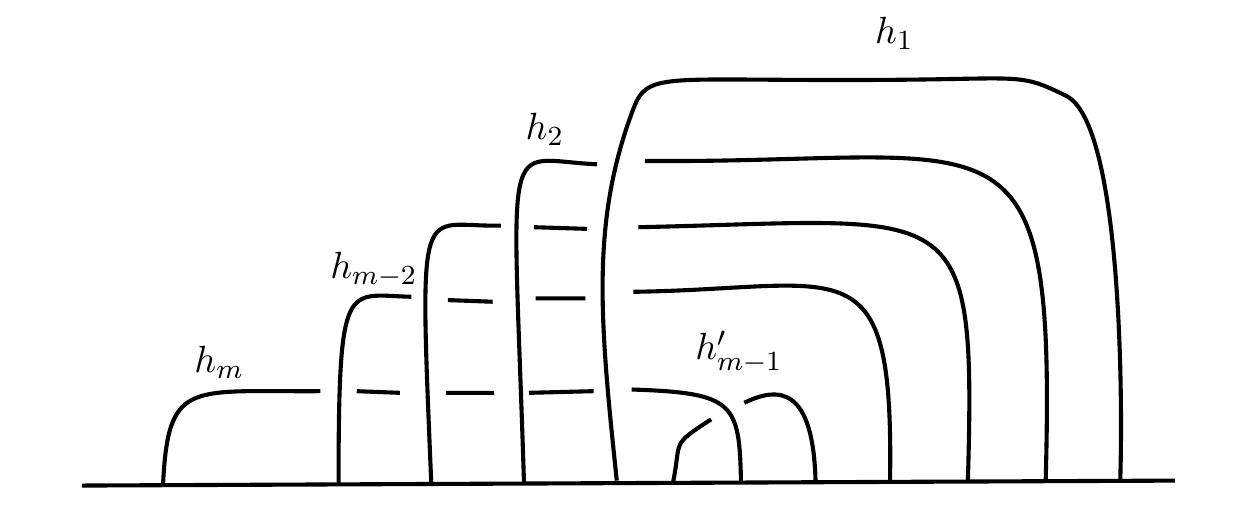}
\caption{} 
\label{fig: cam 10}
\end{figure} 

Now slide $h_{m-2}$ over $h_m$, and then slide $h_m$ over $h'_{m-2}$; see Figures  \ref{fig: cam 11} and \ref{fig: cam 12}. 

\begin{figure}[!ht]
\centering
 \includegraphics[scale=0.75]{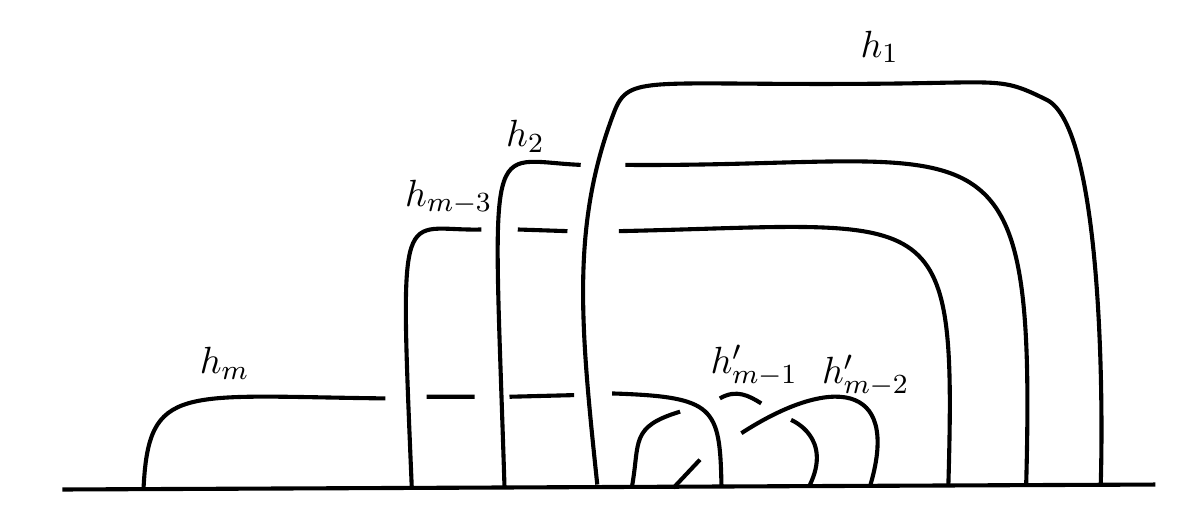}
\caption{} 
\label{fig: cam 11}
\end{figure} 

\begin{figure}[!ht]
\centering
 \includegraphics[scale=0.75]{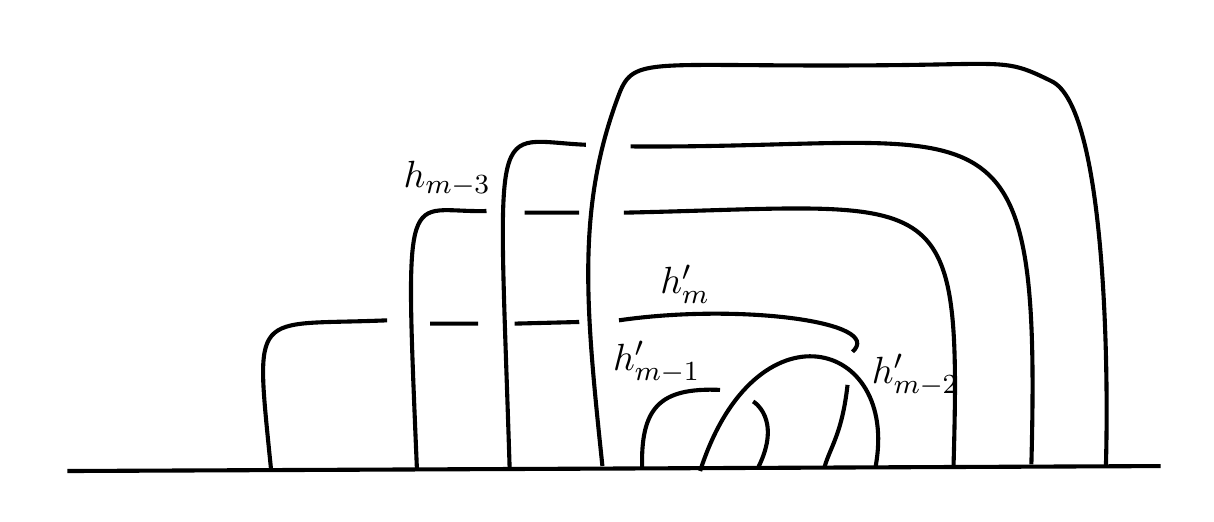}
\caption{} 
\label{fig: cam 12}
\end{figure}

The incidence graph is now $K_{m-2}$ with a leg of length 2. Continue this procedure: slide $h_{m-3}$ over $h'_m$ and then slide $h'_m$ over $h'_{m-3}$; the incidence graph is now $K_{m-3}$ with a leg of length 3; see Figure \ref{fig: cam 13}. 

\begin{figure}[!ht] 
\centering
 \includegraphics[scale=0.75]{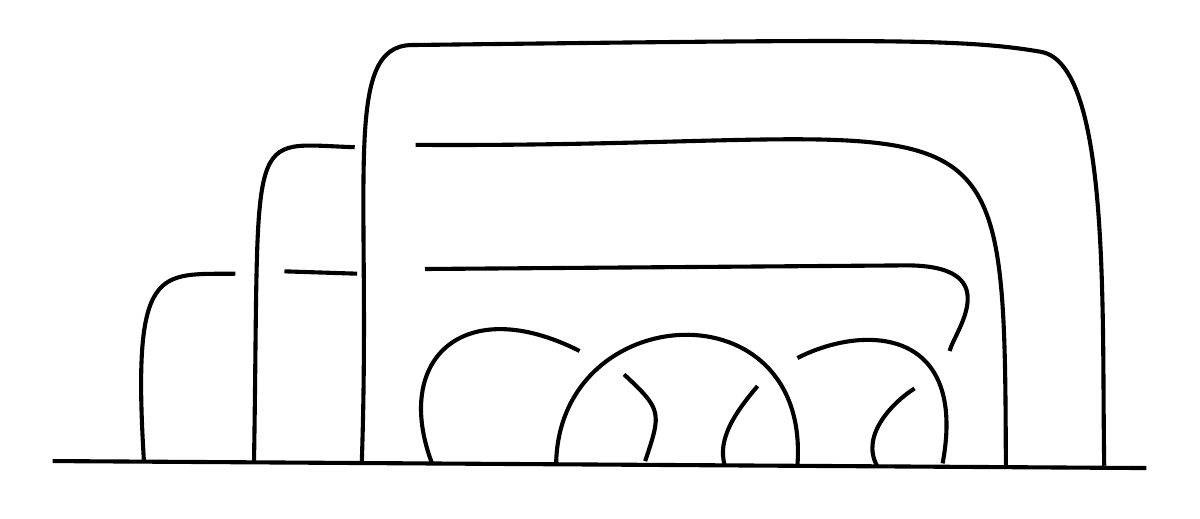}
\caption{} 
\label{fig: cam 13}
\end{figure} 

Eventually we obtain the incidence graph $A_m$.
\end{proof}

\section{L-space branched covers of strongly quasipositive links of braid index $3$}  
\label{sec: lspace branched covers 3-braids}

In this section we investigate the pairs $(L, n)$ where $L$ is a non-split strongly quasipositive link of braid index $3$ 
and $\Sigma_n(L)$ is an L-space. This will lead to a proof of Proposition \ref{prop: 1 < r < n}. We assume that $(L, n)$ is such a pair throughout. By Theorem \ref{thm: bbg definite}, $L$ is definite and therefore Theorem \ref{thm: definite 3-braids} implies that 
$\Sigma_2(L)$ is an L-space. 

Stoimenow has shown that strongly quasipositive links of 
braid index $3$ are the closures of strongly quasipositive $3$-braids (\cite[Theorem 1.1]{Stoi}), so we can write 
$$L = \widehat b$$
where 
$$b = \delta_3^k P$$ 
is a minimal representative of a conjugacy class of 
strongly quasipositive $3$-braids (cf.~\S \ref{subsec: minimal rep}). We analyse the cases $k > 0$, $k = 0$ separately. 

\subsection{The case $k > 0$}
\label{subsec: k > 0}

Proposition \ref{prop: when definite} implies that the minimal form of $b$ is  
one of the following braids: 
\vspace{-.2cm} 
\begin{enumerate}

\item $\delta_3^k \sigma_1^p \mbox{ where either } \left\{ 
\begin{array}{l}
p = 0   \mbox{ and } 1 \leq k \leq 5 \\ 
p > 0  \mbox{ and } 1 \leq k \leq 3 \\ 
p = \{1, 2\}  \mbox{ and }  k = 4
\end{array} \right.$ 

\vspace{.2cm} \item $\delta_3 \sigma_1^p a_{13}^q$ where $p, q \geq 1$.  

\end{enumerate}
The closures of these braids are non-split and the only one whose closure is trivial is $\delta_3$. In the latter case, $\Sigma_n(\widehat b)$ is an L-space for each $n$, and so Proposition \ref{prop: 1 < r < n} holds. Without loss of generality we assume $\widehat b$ is non-trivial below. 

First consider the braids listed in (2) and note that the closure of $b = \delta_3 \sigma_1^p a_{13}^q$ is the composite link $\widehat b = T(2, p+1) \# T(2, q+1)$. 

Since $\Sigma_n(T(2, p+1) \# T(2, q+1)) \cong \Sigma_n(T(2, p+1)) \# \Sigma_n(T(2, q+1))$ is an L-space if and only if both $\Sigma_n(T(2, p+1))$ and $\Sigma_n(T(2, q+1))$ are L-spaces, it follows from \cite{GLid} and \cite[Example 1.11]{Nem} that Proposition \ref{prop: 1 < r < n} holds for the closures of the braids in family (2). 

More precisely, if $p = q = 1$, $\widehat b = T(2, 2) \# T(2, 2)$ where $T(2, 2)$ is the Hopf link. Since each $\Sigma_n(T(2,2))$ is a lens space, $\Sigma_n(T(2, 2) \# T(2, 2))$ is an L-space for each $n \geq 2$. Otherwise, $\Sigma_n(\widehat b)$ is an L-space if and only if  
$$n \leq \left\{ 
\begin{array}{l} 
2 \mbox{ when } \max\{p, q\} \geq  5 \\
3 \mbox{ when } \max\{p, q\}  = 3, 4 \\   
5 \mbox{ when } \max\{p, q\} = 2
\end{array} \right.$$

Next consider the braids listed in (1). These are clearly positive and their closures are prime. They are also definite, so by Theorem \ref{thm: baader}, $\hat b$ is simply laced arborescent and therefore conjugate to one of the following positive braids:   
{\small 
\begin{center}
\begin{tabular}{|c||c|c|c|c|c|} \hline 
$b$ & $\sigma_1^m \sigma_2$ ($m \geq 2$)& $\sigma_1^m \sigma_2 \sigma_1^2 \sigma_2$ ($m \geq 2$)& $\sigma_1^3 \sigma_2 \sigma_1^3 \sigma_2 $ & $\sigma_1^4 \sigma_2 \sigma_1^3 \sigma_2$ & $\sigma_1^5 \sigma_2 \sigma_1^3 \sigma_2$  \\  \hline 
\footnotesize{$\widehat b$} & $T(2,m)$ & $P(-2,2,m) $ & $P(-2,3,3) = T(3,4)$ & $P(-2,3,4)$ & $P(-2,3,5) = T(3,5)$   \\  \hline 
\footnotesize{$\mathcal{F}(\widehat b)$} & $A_{m-1}$ & $D_{m+2}$ & $E_6$ & $E_7$ & $E_8$   \\  \hline 
\end{tabular}
\end{center}}
To see that Proposition \ref{prop: 1 < r < n}  holds for these braids first note that $\Sigma_n(T(2,2))$ is a lens space for each $n \geq 2$, and therefore an L-space. Also, it follows from \cite{GLid} and \cite[Example 1.11]{Nem} that 
 $$\Sigma_n(\widehat b) \mbox{ is an L-space if and only if } n \leq \left\{ 
\begin{array}{ll} 
2 & \mbox{ if } \widehat b = T(2,m) \mbox{ where } m \geq 6  \\
2 & \mbox{ if } \widehat b = T(3,3), T(3,4) \mbox{ or }  T(3,5)  \\
3 & \mbox{ if } \widehat b = T(2,4) \mbox{ or } T(2,5) \\ 
5 & \mbox{ if } \widehat b = T(2,3)   
\end{array} \right.$$ 
Thus we are left with considering the cases that $b$ is either $\sigma_1^m \sigma_2 \sigma_1^2 \sigma_2$ ($m \geq 2$) or $\sigma_1^4 \sigma_2 \sigma_1^3 \sigma_2$. We begin with some remarks. 

For each $\zeta \in S^1 \setminus \{-1\}$ we define $\overline{I}_-(\zeta)$, to be the closed subarc of the circle with endpoints $\zeta, \bar \zeta$  which contains $-1$. 

Let $\mathcal{S}_{F(b)}: H_1(F(b)) \times H_1(F(b)) \to \mathbb Z$ be the Seifert form of the surface $F(b)$ and recall the Hermitian form $\mathcal{S}_{F(b)}(\zeta) = (1 - \zeta)\mathcal{S}_{F(b)} + (1 - \bar \zeta) \mathcal{S}_{F(b)}^T$ defined for $\zeta \in S^1$. It follows from \cite[Theorem 1.1]{BBG} that if $\Sigma_n(\widehat b)$ is an L-space, then $\mathcal{S}_{F(b)}(\zeta)$ is definite for $\zeta \in \overline{I}_-(\zeta_n)$. If $b_0$ is contained in $b$, $F(b_0) \subseteq F(b)$ and therefore $\mathcal{S}_{F(b_0)}(\zeta)$ is definite for $\zeta \in \overline{I}_-(\zeta_n)$ as well. In particular, $\Delta_{\widehat b_0}(t)$ has no zero in this interval. 

Suppose that $b = \sigma_1^4 \sigma_2 \sigma_1^3 \sigma_2$ and that $\Sigma_n(\widehat b)$ is an L-space for some $n \geq 2$. Then $\mathcal{S}_{F(b)}(\zeta)$ is definite for $\zeta \in \overline{I}_-(\zeta_n)$. As $b$ contains $b_0 = \sigma_1^7 \sigma_2$, the same is true for $\mathcal{S}_{F(b_0)}(\zeta)$. In particular, the Alexander polynomial of $\widehat{b_0} = T(2,7)$ has no root in $\overline{I}_-(\zeta_n)$. Since 
$$\Delta_{\widehat b_0}(t) = \frac{t^7 + 1}{t+1},$$
$\Delta_{\widehat b_0}^{-1}(0)$ is the set of primitive $14^{th}$ roots of $1$. It follows that $\zeta_{14}^{5}$ is not contained in $\overline{I}_-(\zeta_n)$ and therefore $\frac{5}{14}  < \frac{1}{n}$. Equivalently, $n < \frac{14}{5} < 3$. Thus $n = 2$ so that Proposition \ref{prop: 1 < r < n} holds for these braids.   

Finally suppose that $b = \sigma_1^m \sigma_2 \sigma_1^2 \sigma_2$ where $m \geq 2$. Then $b$ contains the sub-braid $b_0 = \sigma_1^2 \sigma_2 \sigma_1^2 \sigma_2 = \sigma_1 (\sigma_2 \sigma_1 \sigma_2) \sigma_1 \sigma_2 = \delta_3^3$. It is simple to verify that $\widehat{b_0} = T(3,3)$ has Conway polynomial $z^2(z^2 + 3)$, and as $\Delta_{\widehat{b_0}}(t^2) = \nabla_{\widehat{b_0}}(t - t^{-1})$, the Alexander polynomial of $\widehat{b_0}$ is $(t-1)^2(t^2 + t + 1)$. Hence $\Delta_{\widehat{b_0}}(\zeta_3) = 0$. It follows that $\mathcal{S}_{F(b_0)}(\zeta_3)$, and therefore $\mathcal{S}_{F(b)}(\zeta_3)$, is indefinite. In particular, $\Sigma_n(\widehat b)$ cannot be an L-space for $n \geq 3$. 

Summarising, when $k > 0$ and $b = \delta_3^k P$ is a BKL-positive $3$-braid for which $\widehat b$ is a non-trivial prime link and some $\Sigma_n(\widehat b)$ is an L-space, then $b$ is conjugate to one of the positive braids listed in the table above. If $b$ is $\sigma_1^2 \sigma_2$, then $\Sigma_n(\widehat b)$ is an L-space for each $n \geq 2$. Otherwise, $\Sigma_n(\widehat b)$ is an L-space if and only if  
$$n \leq \left\{ 
\begin{array}{l} 
2 \mbox{ and } b = \sigma_1^m \sigma_2 \mbox{ ($m \geq 6$)}, \sigma_1^m \sigma_2 \sigma_1^2 \sigma_2 \mbox{ ($m \geq 2$)}, \mbox{or } \sigma_1^m \sigma_2 \sigma_1^3 \sigma_2 \mbox{ ($m = 3,4,5$)} \\
3 \mbox{ and } b = \sigma_1^4 \sigma_2 \mbox{ or } \sigma_1^5 \sigma_2   \\   
5 \mbox{ and } b = \sigma_1^3 \sigma_2 
\end{array} \right.$$ 
Thus Proposition \ref{prop: 1 < r < n} holds in the case that $k > 0$

\subsection{The case $k = 0$}
\label{subsec: k = 0}
In this case $b$ is conjugate to $b(p,q,r) = \sigma_1^p a_{13}^q \sigma_2^r$ where $p, q, r \geq 1$ (Corollary \ref{cor: definiteness, fibredness, and forms}). 
As we observed above, $\Sigma_2(\widehat{b})$ is an L-space. 

\subsubsection{Restricting the values of $p, q, r$} 
\label{subsubsec: bounding pqr} 
The identity
$$\delta_3 b(p,q,r) \delta_3^{-1} = \sigma_2^p \sigma_1^q a_{13}^r \sim b(q,r,p)$$
implies that $\widehat{b(p,q,r)} = \widehat{b(q,r,p)} = \widehat{b(r, p,q)}$. Further, $\widehat{b(p,q,r)}$ and $\widehat{b(q,p,r)}$ differ by a flype (\cite[The Classification Theorem, page 27]{BiM}), so coincide. Thus

\begin{lemma} 
\label{lemma: invariant under permutation} 
$\widehat{b(p,q,r)}$ is invariant under arbitrary permutations of $(p,q,r)$. 
\qed
\end{lemma}
As such, we assume below that 
$$p \leq q \leq r$$ 
We noted in \S \ref{subsec: k > 0}  that  if $b_0$ is contained in $b$, 
then $F(b_0) \subseteq F(b)$ and therefore $\mathcal{S}_{F(b_0)}(\zeta)$ is definite for $\zeta \in \overline{I}_-(\zeta_n)$. In particular, $\Delta_{\widehat b_0}(t)$ 
has no zero in this interval. For instance, taking $b_0 = \sigma_1 \sigma_2^r$ we have $\widehat b_0 = T(2, r)$ and therefore 
$$\Delta_{\widehat b_0}(t) = \frac{t^r + (-1)^{r+1}}{t+1}$$
It follows that
$$\Delta_{\widehat b_0}^{-1}(0) = \left\{ 
\begin{array}{l} 
\mbox{ the set of primitive $2r^{th}$ roots of $1$ other than $-1$ if $r$ is odd} \\
\mbox{ the set of $r^{th}$ roots of $1$ other than $-1$ if $r$ is even} 
\end{array} \right.$$
If $r \geq 3$ is odd, then $\zeta_{2r}^{r - 2} \ne -1$ and is a primitive $2r^{th}$ root of unity, so it is not contained in $\overline{I}_-(\zeta_n)$. Therefore $\frac{1}{n} > \frac{r-2}{2r}$. Equivalently, $r < 2 + \frac{4}{n-2}$. Similarly if $r \geq 2$ is even, then $\zeta_r^{\frac{r}{2} - 1} \not \in \overline{I}_-(\zeta_n)$ so that $\frac{\frac{r}{2}-1}{r} = \frac{r-2}{2r} < \frac{1}{n}$. Again we deduce that $r < 2 + \frac{4}{n-2}$. Thus
$$\left\{ \begin{array}{lll} 
 n = 3 & \Rightarrow & 1 \leq p \leq q \leq r \leq 5 \\
n = 4, 5 & \Rightarrow & 1 \leq p \leq q \leq r  \leq 3 \\ 
n \geq 6 & \Rightarrow & 1 \leq p \leq q \leq r \leq 2
\end{array} \right.$$

In what follows, we study whether or not $\Sigma_n(\widehat{b(p,q,r)})$ is an L-space  
by analysing the zeros of the Alexander polynomials of the links $\widehat{b(p,q,r)}$ for each of the thirty-five values of $(p, q, r)$ with $1 \leq p \leq q \leq r \leq 5$ .  

\subsubsection{The Alexander polynomials of  $\widehat{b(p,q,r)}$ where $1 \leq p \leq q \leq r \leq 5$}
\label{subsec: calculating alex polys} 
Let $\nabla(p,q, r)(z)$ denote the Conway polynomial of $\widehat b(p,q,r)$. 
The skein relation (\cite[\S 2]{Kf}) implies that
\begin{equation}
\label{eqn: skein} 
\nabla(p,q, r)(z) = \left\{
\begin{array}{l} z \nabla(p-1,q, r)(z) + \nabla(p-2,q, r)(z) \\  
z \nabla(p,q-1, r)(z) + \nabla(p,q-2, r)(z)\\ 
z \nabla(p,q, r-1)(z) + \nabla(p,q, r-2)(z) 
 \end{array} \right.
 \end{equation} 
Since $T(2, p) = \widehat{\sigma_1^p\sigma_2}$, the reader will verify using (\ref{eqn: skein}) that $\nabla_{T(2, p)}(z) = f_p(z)$ where $f_n(z)$ is defined 
recursively by 
$$f_0(z) = 0, f_1(z) = 1, f_n(z) = zf_{n-1}(z) + f_{n-2}(z) \mbox{ when } n \geq 2$$
For instance, 
{\small 
\begin{center}
\begin{tabular}{|c||c|c|c|c|c|c|c|c|} \hline 
$n$ & $0$ & $1$ & $2$ & $3$ & $4$ & $5$ & $6$ & $7$  \\  \hline 
$f_n(z)$ & $0$ & $1$ & $z$ & $z^2 + 1$ & $z(z^2 + 2)$ & $z^4 + 3z^2 + 1$ & $z(z^2 + 1)(z^2 + 3)$  & $z^6 + 5z^4 + 6z^2 + 1$ \\  \hline 
\end{tabular}
\end{center}}
Since $b(p, q, 0) \cong T(2, p) \# T(2, q)$, we have 
$$\nabla(p, q, 0) = f_p(z) f_q(z)$$
for $p, q \geq 0$. An inductive argument then shows that 
\begin{eqnarray} 
\label{eqn: r} 
\nabla(1,1,r) = f_r(z) \nabla(1,1,1) + f_{r-1}(z) \nabla(1,0,1) = f_r(z) \nabla(1,1,1) + f_{r-1}(z) \nonumber 
\end{eqnarray} 
Similarly we have 
\begin{eqnarray} 
\label{eqn: qr} 
\nabla(1,q,r) & = & f_q(z) \nabla(1,1,r) + f_{q-1}(z) \nabla(1,0,r) \nonumber \\
& = & f_q(z) f_r(z) \nabla(1,1,1) + f_q(z) f_{r-1}(z) + f_{q-1}(z) f_{r}(z) \nonumber 
\end{eqnarray} 
Finally,
\begin{eqnarray} 
\label{eqn: pqr} 
\nabla(p,q,r) & = & f_p(z) \nabla(1,q,r) + f_{p-1}(z) \nabla(0,q,r) \nonumber  \\
& = & f_p(z) f_q(z) f_r(z) \nabla(1,1,1) + f_p(z) f_q(z) f_{r-1}(z) + f_p(z) f_{q-1}(z) f_{r}(z) \nonumber \\
&   & \hspace{.5cm} + f_{p-1}(z) f_{q}(z) f_{r}(z) \nonumber
\end{eqnarray} 
Since $\nabla(1,1,1) = 2z$, we have  
\begin{eqnarray} 
\nabla(p,q,r) = 2z f_p(z) f_q(z) f_r(z) + f_p(z)f_q(z)f_{r-1}(z) + f_p(z)f_{q-1}(z)f_r(z) + f_{p-1}(z)f_q(z)f_r(z) \nonumber 
\end{eqnarray} 
This identity combined with the skein relation (\ref{eqn: skein}) yields the following table. 
{\footnotesize
\begin{center}
\begin{tabular}{|c|c||c|c|} \hline
$(p,q,r)$ & $\nabla(p,q,r)$ & $(p,q,r)$ & $\nabla(p,q,r)$      \\  \hline   \hline  
$(1,1,1)$   &  $2z$  &  $(2,2,5)$ &  $z(2z^6 + 9 z^4 +10z^2 + 2)$       \\  \hline 
$(1,1,2)$   &  $2z^2  + 1$  &  $(2,3,3)$  &  $(z^2 + 1)(2z^4 + 5z^2 + 1)$        \\  \hline 
$(1,1,3)$   & $z(2z^2 + 3)$  &  $(2,3,4)$   &  $z(2z^6 + 9z^4 + 11z^2 + 3)$    \\   
&&& $= z(2 z^2 + 3) (z^4 + 3 z^2 + 1)$\\  \hline 
$(1,1,4)$   &  $2z^4 + 5z^2 + 1$  &  $(2,3,5)$ & $2z^8 + 11z^6 + 18z^4 + 9z^2 + 1$         \\  \hline 
$(1,1,5)$   &  $z(2z^4 + 7z^2 + 4)$  & $(2,4,4)$     &   $z^2(z^2+2)(2z^4 + 7z^2+4)$        \\  \hline 
$(1,2,2)$   &  $2z(z^2 + 1)$  &  $(2,4,5)$  & $z(2z^8 + 13z^6 + 27z^4 + 19z^2 + 3)$    \\   
&&& $ = z(z^2 + 1) (z^2 + 3) (2 z^4 + 5 z^2 + 1)$ \\  \hline 
$(1,2,3)$   &  $2z^4 + 4z^2 + 1$  &  $(2,5,5)$  &  $(z^4 + 3z^2 + 1)(2z^6 + 9z^4 + 9z^2 + 1)$       \\  \hline 
$(1,2,4)$   &  $z(2z^4 + 6z^2 + 3)$  &  $(3,3,3)$  &  $z(z^2 + 1)^2 (2z^2 + 5)$          \\  \hline 
$(1,2,5)$   &  $(z^2 + 1) (2 z^4 + 6 z^2 + 1)$  &  $(3,3,4)$  & $(z^2+1)(2z^6 + 9z^4+10z^2+1)$         \\ \hline 
$(1,3,3)$   &  $2z(z^2 + 1)(z^2 + 2)$  &   $(3,3,5)$ & $z(z^2 + 1) (z^2 + 2)(2z^4 + 7z^2 + 3)$ \\   
&&&$   = z(z^2 + 1) (z^2 + 2)(2 z^2 + 1) (z^2 + 3)$          \\  \hline 
$(1,3,4)$   &  $2z^6 + 8z^4 + 8z^2 + 1$  &  $(3,4,4)$  &  $z(z^2 + 2)(2z^6 + 9z^4 + 10z^2 + 2)$         \\  \hline 
$(1,3,5)$   & $z(2z^6 + 10z^4 + 14z^2 + 5)$  & $(3,4,5)$  &  $2z^{10} + 15z^8 + 39z^6 + 41z^4 + 15z^2 + 1$          \\  \hline 
$(1,4,4)$ & $2z(z^2 + 2)(z^4 + 3z^2 + 1)$ &  $(3,5,5)$   &  $z(z^4 + 3z^2 + 1)(2z^6 + 11z^4 + 17z^2 + 7)$        \\  \hline 
$(1,4,5)$   &  $2z^8 + 12 z^6 + 22 z^4 + 12 z^2 + 1$  &  $(4,4,4)$ &  $z^2(z^2+2)^2(2z^4 + 7z^2 + 3)$ \\ 
&&& $ = z^2(z^2+2)^2(2z^2 + 1)(z^2 + 3)$  \\  \hline 
$(1,5,5)$   &  $2z(z^2+1)(z^2+3)(z^4 + 3z^2 + 1)$  &  $(4,4,5)$ & $z(z^2+2)(2z^8 + 13z^6 + 26z^4 + 16 z^2 + 2)$            \\ \hline 
$(2,2,2)$   &  $z^2(2z^2 + 3)$  &  $(4,5,5)$   & $(z^2+1)(z^4+3z^2+1)(2z^6 + 11z^4 + 15z^2 + 1)$          \\  \hline 
$(2,2,3)$   &  $z(2z^4 + 5z^2 + 2) = z(2 z^2 + 1) (z^2 + 2)$  &  $(5,5,5)$  &  $z(z^4 + 3z^2 + 1)^2(2z^4 + 9z^2 + 8)$      \\  \hline 
$(2,2,4)$   &  $z^2(2z^4 + 7z^2 + 5)$ &    &          \\  \hline 
\end{tabular}
\end{center}}
The Alexander polynomials of these links are obtained using the identity 
$\nabla(p,q,r)(t - t^{-1}) = \Delta_{\widehat{b(p,q,r)}}(t^2)$. 
The result is listed in the second column of the following table. The third column lists the values of $n$ for which the position of the 
roots of the Alexander polynomial obstructs $\Sigma_n(\widehat{b(p,q,r)})$ from being an L-space (cf. \S \ref{subsubsec: bounding pqr}). The last 
column lists what is known about whether or not $\Sigma_n(\widehat{b(p,q,r)})$ is an L-space for the remaining values of $n \geq 3$ 
as well as those values for which the answer is unknown to us.

{\footnotesize
\begin{center}
\begin{tabular}{|c|c|c|c|} \hline
& & $\Delta_{\widehat b}(t)$ implies   & $\Sigma_n(\widehat b)$ is an L-space ($\checkmark$)  \\ 
$(p,q,r)$  & $\Delta_{\widehat b}(t)$  & $\Sigma_n(\widehat b)$ not  & $\Sigma_n(\widehat b)$ is not an L-space (x) \\  
&& an L-space&  Unknown (?) \\      \hline   \hline 
$(1,1,1)$   & 2(t-1)   &     & all $n$ $\checkmark$ \\  \hline 
$(1,1,2)$   & $2t^2 - 3t + 2$ &$n \geq 9$    & $n = 3,4,5 \checkmark; 6,7,8$  (x)\\  \hline 
$(1,1,3)$   &$(t - 1)(2t^2 - t + 2)$ & $n \geq 5$   & $n = 3, 4$ (?)\\  \hline 
$(1,1,4)$   & $2t^4 - 3t^3 + 3t^2 - 3t + 2$& $n \geq 4$   & $n = 3$ $\checkmark$ \\  \hline 
$(1,1,5)$   & $(t-1)(2t^4 - t^3 + 2t^2 - t + 2)$ &  $n \geq 4$ &  $n = 3$ (?) \\  \hline 
$(1,2,2)$   &$2(t - 1)(t^2 - t + 1)$ & $n \geq  6$  & $n = 3, 4, 5$ (?) \\  \hline 
$(1,2,3)$   & $2t^4 - 4t^3 + 5t^2 - 4t + 2$ & $n \geq 5$  & $n = 3$ $\checkmark$; $4$ (?)  \\  \hline 
$(1,2,4)$   & $(t -1)(2t^4 - 2t^3 + 3t^2 - 3t + 2)$ & $n \geq 4$  &   $n = 3$  (?)  \\  \hline 
$(1,2,5)$   & $(t^2 - t + 1)(2t^4 - 2t^3 + t^2 - 3t + 2)$ & $n \geq 4$&  $n = 3$ (?)\\ \hline 
$(1,3,3)$   & $2(t-1)(t^2 - t + 1)(t^2 + 1)$ & $n \geq 4$ & $n = 3$ (?)\\  \hline 
$(1,3,4)$   & $2t^6 - 4t^5 + 6t^4 - 7t^3 + 6t^2 - 4t + 2$  & $n \geq 4$&  $n = 3$ (?)\\  \hline 
$(1,3,5)$   & $(t - 1)(2t^6 - 2t^5 + 4t^4 - 3t^3 + 4t^2 - 2t + 2)$ & $n \geq 4$&  $n = 3$ (?)  \\  \hline 
$(1,4,4)$ & $2(t-1)(t^2 + 1)(t^4 - t^3 + 5t^2 - t + 1)$ &$n \geq 4$ &  $n = 3$ (?)\\  \hline 
$(1,4,5)$   &$2t^8 - 4t^7 + 6t^6 - 8t^5 + 9t^4 - 8t^3 + 6t^2 - 4t + 2$ &  $n \geq 4$ & $n = 3$ (?)\\  \hline 
$(1,5,5)$   &$2(t-1)(t^2 - t + 1)(t^2+t+1)(t^4 - t^3 + t^2 - t + 1) $ & $n \geq 3$   &  $*$ \\ \hline 
$(2,2,2)$   & $(t-1)^2(2t^2 - t + 2)$ & $n \geq 5$ & $n = 3, 4$ (?)  \\  \hline 
$(2,2,3)$   & $(t-1)(t^2 + 1)(2t^2 - 3t + 2)$& $n \geq 4$  & $n = 3$ (?)\\  \hline 
$(2,2,4)$   & $(t-1)^2(2t^4 - t^3 + 3t^2 - t + 2)$ & $n \geq 4$   & $n = 3$ (?)  \\  \hline 
$(2,2,5)$ & $(t-1)(2t^6 - 3t^5 + 4t^4 - 4t^3 + 4t^2 - 3t +2)$ & $n \geq 4$ & $n = 3$ (?)  \\  \hline 
$(2,3,3)$  & $(t^2 - t +1)(2t^4 - 3t^3 + 3t^2 - 3t + 2)$   &  $n \geq 4$ & $n = 3$ (?)  \\  \hline 
$(2,3,4)$   & $(t-1)(2t^2 - t + 2)(t^4 - t^3 + 5t^2 - t + 1)$    &  $n \geq 4$ & $n = 3$ (?)\\  \hline 
$(2,3,5)$ &  $2t^8 - 5t^7 + 8t^6 - 19t^5 + 29t^4 - 19t^3 + 8t^2 - 5t + 2$ &  $n \geq 3$  & $*$\\  \hline 
$(2,4,4)$    & $(t-1)^2(t^2 + 1)(2t^4 - t^3 + 2t^2 - t + 2)$ & $n \geq 4$ & $n = 3$ (?)\\  \hline 
$(2,4,5)$  & $(t-1)(t^2 - t +1)(t^2 + t +1)(2t^4 - 3t^3 + 3t^2 - 3t + 2)$ & $n \geq 3$ & $*$\\  \hline 
$(2,5,5)$  & $(t^4 - t^3 + 5t^2 - t + 1)(2t^6 - 3t^5 + 3t^4 - 3t^3 + 3t^2 - 3t +2)$  &    $n \geq 3$ & $*$\\  \hline 
$(3,3,3)$  & $(t-1)(t^2 - t + 1)^2(2t^2 + t + 2)$ &   $n \geq 4$   & $n = 3$ (?) \\  \hline 
$(3,3,4)$  &  $(t^2 - t + 1)(2t^6 - 3t^5 + 4t^4 - 4t^3 + 4t^2 - 3t + 2)$ & $n \geq 4$ & $n = 3$ (?)\\ \hline 
$(3,3,5)$ & $(t-1)(t^2 - t + 1)(t^2 + 1)(t^2 + t + 1)(2t^2 - 3t + 2)$ & $n \geq 3$ &$*$\\  \hline 
$(3,4,4)$  & $(t-1)(t^2 + 1)(2t^6 - 3t^5 + 4t^4 - 4t^3 + 4t^2 - 3t +2)$  & $n \geq 4$ & $n = 3$ (?)\\  \hline 
$(3,4,5)$  &  $2t^{10} - 5t^9 + 9t^8 - 13t^7 + 16t^6 -17 t^5 + 16t^4 -13t^3 + 9t^2 - 5t + 2$ &   $n \geq 3$ &$*$ \\  \hline 
$(3,5,5)$   &  $(t-1)(t^4 - t^3 + 5t^2 - t + 1)(2t^6 - t^5 + 3t^4 - t^3 + 3t^2 - t + 2)$   &    $n \geq 3$ &$*$\\  \hline 
$(4,4,4)$  & $(t-1)^2(t^2 + 1)^2(t^2 + t + 1)(2t^2 - 3t + 2)$ & $n \geq 3$ & $*$\\  \hline 
$(4,4,5)$ & $(t-1)(t^2 + 1)(2t^8 - 3t^7 + 4t^6 - 5t^5 + 6t^4 - 5t^3 + 4t^2 - 3t + 2)$ &   $n \geq 3$ & $*$\\ \hline 
$(4,5,5)$   & $(t^2 - t + 1)(t^4 - t^3 + 5t^2 - t + 1)(2t^6 - t^5 + t^4 - 3t^3 + t^2 - t + 2)$  &$n \geq 3$ &$*$\\  \hline 
$(5,5,5)$  & $(t-1)(t^4 - t^3 + 5t^2 - t + 1)^2(2t^4 + t^3 + 2t^2 + t + 2)$ & $n \geq 3$ &$*$\\  \hline 
\end{tabular}
\end{center}}

\subsubsection{The proof of Proposition \ref{prop: 1 < r < n} when $k = 0$} 
\label{subsubsec: proof of proposition}

Assume that $\Sigma_n(\widehat{b(p,q,r)})$ is an L-space for some $ 1 \leq p \leq q \leq r \leq 5$ and $n \geq 2$. Since 
$\Sigma_2(\widehat{b(p,q,r)})$ is an L-space, Proposition \ref{prop: 1 < r < n} holds if $n = 2$ or $3$. We suppose that $n \geq 4$ below. 

\begin{case}
$n \geq 6$ 
\end{case} 

We saw in \S \ref{subsubsec: bounding pqr} that $\max \{p, q, r\} \leq 2$ when $n \geq 6$, so $(p,q,r)$ is either $(1,1,1), (1,1,2), (1,2,2)$, or $(2,2,2)$. We 
also saw that there are no roots of $\Delta_{\widehat b}(t)$ in $\overline{I}_-(\zeta_n)$. Equivalently, if $\zeta \in S^1$ is a root of $\Delta_{\widehat b}(t)$, 
then $\mathfrak{Re}(\zeta) > \cos(2\pi/n) \geq \cos(2\pi/6) = \frac12$. The Alexander polynomials of these links are listed in the third table in 
\S \ref{subsec: calculating alex polys} and their roots are  
$$\left\{ 
\begin{array}{cl} 
\{1\} & \mbox{ if } (p,q,r) = (1,1,1)\\ 
\{\frac{3 + \sqrt{7} i}{4}, \frac{3 - \sqrt{7} i}{4}\} & \mbox{ if } (p,q,r) = (1,1,2) \\
\{1, \frac{1 + \sqrt{3} i}{2}, \frac{1 - \sqrt{3} i}{2}\} & \mbox{ if } (p,q,r) = (1,2,2) \\
\{1, \frac{1 + \sqrt{15} i}{4}, \frac{1 - \sqrt{15} i}{4}\}  & \mbox{ if } (p,q,r) = (2,2,2) 
\end{array} \right.$$
Since these roots $\zeta$ must satisfy $\mathfrak{Re}(\zeta) >  \frac12$, the only possibilities are $b(1,1,1)$ (any $n$) or $b(1,1,2)$ ($n = 6, 7, 8$). 

\begin{subcase}
$(p,q,r) = (1,1,1)$
\end{subcase} 

\begin{lemma}
The link $\widehat{b(1,1,1)}$ is $T(2, -4)$ oriented so that its components have linking number $2$. Further, $\Sigma_n(\widehat{b(1,1,1)})$ is an L-space for each $n$. In particular, Proposition \ref{prop: 1 < r < n} holds when $(p,q,r) = (1,1,1)$. 
\end{lemma}

\begin{proof}
The first claim is easily verified. For the second, note that the exterior $X$ of $\widehat{b(1,1,1)}$ is Seifert fibred with base orbifold $A(2)$ where $A$ is an annulus. The homomorphism $\pi_1(X) \to \mathbb Z/n$ associated to the $n$-fold cyclic cover $X_n$ of $X$ determined by $\widehat{b(1,1,1)}$ factors through $H_1(A(2)) \cong \mathbb Z \oplus \mathbb Z/2$ (and therefore $\pi_1(A(2))$) in such a way that the $\mathbb Z/2$ factor is killed. Thus the base orbifold 
of $X_n$ is an annulus with cone points of order $2$. Further, the Seifert fibre of $X_n$ has distance $1$ from the lifts of the meridians, so $\Sigma_n(\widehat{b(1,1,1)}$ has base orbifold a $2$-sphere with $n$ cone points of order $2$. It's also a rational homology $3$-sphere since its Alexander polynomial is $\Delta_{\widehat b}(t) = 2(t - 1)$ and therefore $|H_1(\Sigma_n(\widehat{b(1,1,1)})| = 2^{n-1} \prod_{j=1}^{n-1} |\zeta_n^j - 1| = 2^{n-1} |\prod_{j=1}^{n-1} (\zeta_n^j - 1)| = 2^{n-1}|t^{n-1} + t^{n-2} + \cdots + t + 1|_{t=1} = 2^{n-1} n$ (\cite[Theorem 1]{HK}).  

The normalised Seifert invariants of $\Sigma_n(\widehat{b(1,1,1)}$ (\cite[\S 3]{EHN}) are of the form 
$$(0; b_0, \frac12, \frac12, \ldots, \frac12)$$ 
where there are $n$ ``$\frac12$"'s. Therefore $e(\Sigma_n(\widehat{b(1,1,1)})) = -b_0 - \frac{n}{2}$. On the other hand, it follows from \cite[Corollary 6,2]{JN} that 
$|e(\Sigma_n(\widehat{b(1,1,1)})| = (\frac{1}{\;\;2^n})|H_1(\Sigma_n(\widehat{b(1,1,1)})| = n/2$. Thus 
$$-b_0 - \frac{n}{2} = e(\Sigma_n(\widehat{b(1,1,1)})) = \epsilon n/2$$ 
for some $\epsilon = \pm 1$. It follows that $b_0 + \frac{n}{2} \ne 0$ and further, 
$$b_0 = -(\epsilon + 1)n/2$$
Hence, $|b_0|$ is either $0$ or $n$. If $\Sigma_n(\widehat{b(1,1,1)}))$ supports a co-oriented taut foliation, then by \cite[Theorem 3.3]{EHN}, either $b_0 + \frac{n}{2} = 0$ (which has already been ruled out), or $b_0 \leq -1$ and $b_0 + \frac{n}{2} \geq 1$. The condition that $b_0 \leq -1$ implies that $b_0 = -n$. But then $b_0 + \frac{n}{2} = -\frac{n}{2} < 0$, which contradicts the requirement that $b_0 + \frac{n}{2} \geq 1$. Hence no $\Sigma_n(\widehat{b(1,1,1)}))$ supports a co-oriented taut foliation. Consequently, each is an L-space (\cite{BGW}).  
\end{proof}

\begin{subcase}
$(p,q,r) = (1,1,2)$
\end{subcase} 

\begin{lemma}
$\widehat{b(1,1,2)}$ is the knot $5_2$. Consequently, $\Sigma_n(\widehat{b(1,1,2)})$ is an L-space if and only if $n = 2,3,4, 5$. In particular, Proposition \ref{prop: 1 < r < n} holds when $(p,q,r) = (1,1,2)$. 
\end{lemma}

\begin{proof}
The first claim is easily verified. For the second, note that it follows from the third table in \S \ref{subsec: calculating alex polys} that $\Sigma_n(\widehat{b(1,1,2)})$ is not an L-space for $n \geq 9$. Robert Lipshitz has shown by computer calculation that $\Sigma_n(K)$ is an L-space for $n = 5$, and is not an L-space for $n = 6, 7$, and $8$ (private communication). 
The fact that $\Sigma_5(K)$ is an L-space was also proved by Mitsunori Hori. See \cite{Te2}. We already know that $\Sigma_2(K)$ is an L-space. This is also true for 
$\Sigma_3(\widehat{b(1,1,2)})$ by \cite{Pe} and $\Sigma_4(\widehat{b(1,1,2)})$ by \cite{Te1}. 
\end{proof}

\begin{case}
$n \in \{4, 5\}$ 
\end{case} 

In this case, $\max \{p, q, r\} \leq 3$ so $(p,q,r)$ is one of $(1,1,1), (1,1,2), (1,1,3), (1,2,2), (1,2,3)$, $(1,3,3), (2,2,2), (2,2,3), (2,3,3)$, $(3,3,3)$. Similar to the last case, the condition that $\Sigma_n(\widehat{b(p,q,r)})$ is an L-space implies that for each root $\zeta \in S^1$ of $\Delta_{\widehat b}(t)$ we have 
$$\mathfrak{Re}(\zeta) > \left\{ \begin{array}{ll} \cos(2\pi/4) = 0 & \mbox{ if } n = 4 \\ \cos(2\pi/5) > 0.3 & \mbox{ if } n = 5 \end{array} \right.$$

We saw above that both $\Sigma_4(\widehat{b(p,q,r)})$ and $\Sigma_5(\widehat{b(p,q,r)})$ are L-spaces when $(p,q,r)$ is $(1,1,1)$ or $(1,1,2)$, which verifies Proposition \ref{prop: 1 < r < n} in these cases.

\begin{subcase}
$(p,q,r) = (1,2,2)$
\end{subcase} 

The reader will verify that $\widehat{b(1,2,2)}$ is the link $6_3^2$. See Figure \ref{fig: b(1,2,2)}.  The root condition implies (cf. the third table in \S \ref{sec: lspace branched covers 3-braids}): 

\begin{figure}[!ht]
\centering
 \includegraphics[scale=0.7]{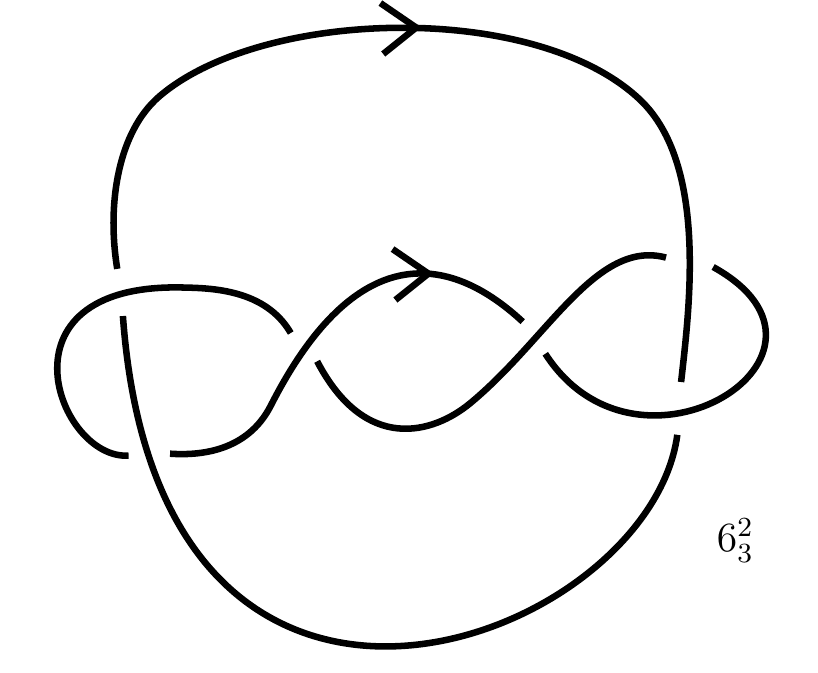} 
\caption{$\widehat{b(1,2,2)} = 6_3^2$} 
\label{fig: b(1,2,2)}
\end{figure} 

\begin{lemma}
$\Sigma_n(\widehat{b(1,2,2)})$ is not an L-space for $n \geq 6$.   
\qed
\end{lemma}
We do not know whether or not $\Sigma_n(\widehat{b(1,2,2)})$ is an L-space for $n = 3, 4, 5$. 

\begin{subcase}
$(p,q,r) = (2,2,2)$
\end{subcase} 

The closure of $b(2,2,2)$ is the link $7_1^3$ (cf. Figure \ref{fig: b(2,2,2)}). From the third table in \S \ref{subsec: calculating alex polys} we see that: 

\begin{figure}[!ht]
\centering
 \includegraphics[scale=0.7]{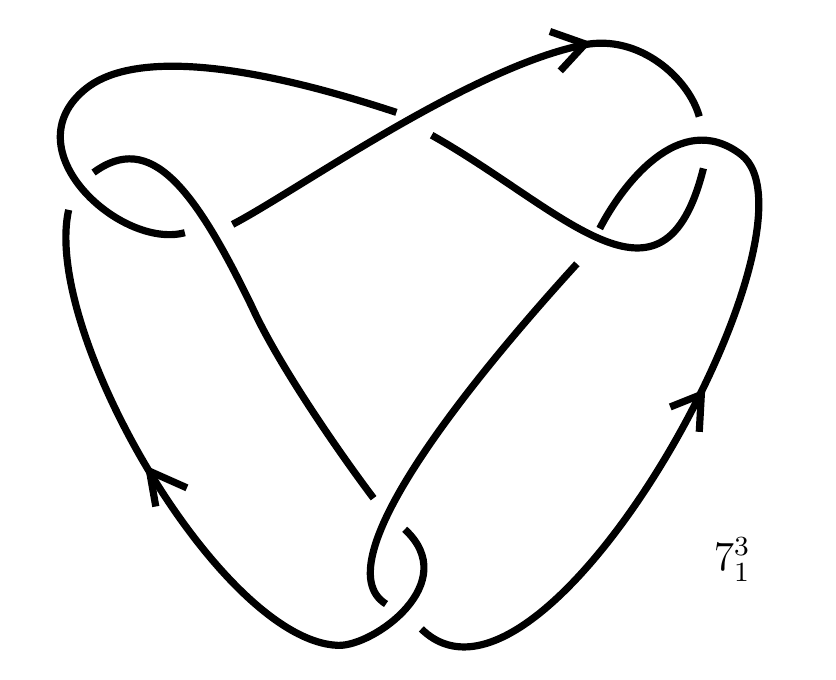} 
\caption{$\widehat{b(2,2,2)} = 7_1^3$} 
\label{fig: b(2,2,2)}
\end{figure} 

\begin{lemma}
$\Sigma_n(\widehat{b(2,2,2)})$ is not an L-space for $n \geq 5$.  
\qed
\end{lemma}

We do not know whether or not $\Sigma_n(\widehat{b(2,2,2)})$ is an L-space for $n = 3, 4$.  

Next suppose that $(p,q,r)$ is one of $(1,1,3), (1,2,3), (1,3,3), (2,2,3), (2,3,3), (3,3,3)$ and consider $b_0 = b(1,3,3)$. From above, $\nabla(1,3,3) =  2z(z^2 + 1)(z^2 + 2)$ and so $\Delta_{\widehat b_0}(t) = 2(t-1)(t^2 - t +1)(t^2 + 1)$. Hence $i$ is a root of $\Delta_{\widehat b_0}(t)$ so that $\mathcal{S}_{F(\widehat{b_0})}(i)$ is indefinite. It follows that $\mathcal{S}_{F(\widehat{b(p,q,r}))}(i)$ is indefinite if $b(p,q,r)$ contains $b_0$. But then neither $\Sigma_4(\widehat{b(p,q,r)})$ nor $\Sigma_5(\widehat{b(p,q,r)})$ is an L-space (\cite[Theorem 1.1]{BBG}), contrary to our assumptions. 

We are left with considering $(p,q,r) = (1,1,3), (1,2,3)$, or $(2,2,3)$. From the third table in \S \ref{subsec: calculating alex polys} we see that the roots of the associated Alexander polynomials are 
$$\left\{ 
\begin{array}{cl} 
\{1, \frac14 \pm \frac{\sqrt{15}}{4}i\} & \mbox{ if } (p,q,r) = (1,1,3)\\ 
\{0.14645 \pm 0.98922i, 0.85355 \pm 0.52101i\} & \mbox{ if } (p,q,r) = (1,2,3) \\
\{1, \pm i, \frac34 \pm \frac14 \sqrt{7} i\} & \mbox{ if } (p,q,r) = (2,2,3) 
\end{array} \right.$$
In each case there is a root $\zeta$ of $\Delta_{\widehat b}(t)$ for which $\mathfrak{Re}(\zeta) <  \frac{3}{10} < \cos(2\pi/5)$. Thus $n \ne 5$. Similarly $n = 4$ is ruled out for $(p,q,r) = (2,2,3)$, contrary to our assumptions. On the other hand, $n = 4$ remains a possibility for $(p,q,r) = (1,1,3)$ or $(1,2,3)$.

\begin{subcase}
$(p,q,r) = (1,1,3)$
\end{subcase} 

The closure of $b(1,1,3)$ is the link $6_2^2$. See Figure \ref{fig: b(1,3,3)}. 

\begin{figure}[!ht]
\centering
 \includegraphics[scale=0.7]{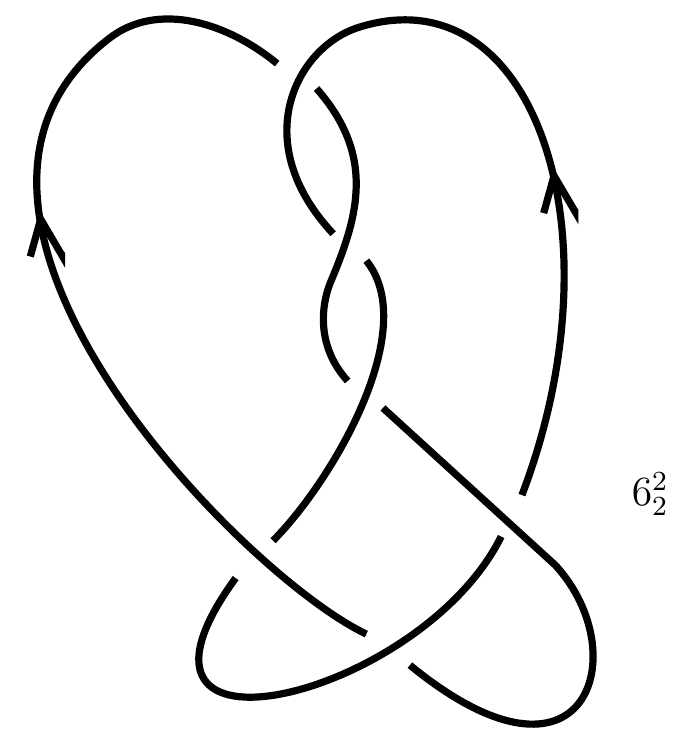} 
\caption{$\widehat{b(1,3,3)} = 6_2^2$} 
\label{fig: b(1,3,3)}
\end{figure} 

From the third table in \S \ref{subsec: calculating alex polys} we see that:

\begin{lemma}
$\Sigma_n(\widehat{b(1,1,3)})$ is not an L-space for $n \geq 5$. 
\qed
\end{lemma}

We do not know whether or not $\Sigma_n(\widehat{b(1,1,3)})$ is an L-space for $n = 3, 4$.  

\begin{subcase}
$(p,q,r) = (1,2,3)$
\end{subcase} 

The closure of $b(1,2,3)$ is the knot $7_5$ as depicted in \cite[Appendix C, page 392]{Rlf}. 

\begin{lemma}
$\Sigma_n(\widehat{b(1,2,3)})$ is an L-space if $n = 3, 4$ and not an L-space for $n \geq 5$. Thus Proposition \ref{prop: 1 < r < n} holds when $(p,q,r) = (1,2,3)$. 
\end{lemma}

\begin{proof}
The third table in \S \ref{subsec: calculating alex polys} shows  that $\Sigma_n(\widehat{b(1,2,3)})$ is not an L-space for $n \geq 5$. Since $7_2$ is a genus $2$ two-bridge knot, \cite[Theorem 1.2]{Ba} shows that $\Sigma_n(\widehat{b(1,2,3)})$ is an L-space for $n = 2$ and $3$.
\end{proof}

\begin{remark} 
\label{rmk: to do}
{\rm We summarise here what is needed to do to extend Proposition \ref{prop: 1 < r < n} to include the three exceptional links $6_2^2, 6_3^2$ and $7_3^1$. 
\vspace{-.2cm}
\begin{itemize}

\item If $L$ is the link $6_2^2 = \widehat{b(1,1,3)}$ (cf. Figure \ref{fig: b(1,3,3)}), it must be shown that either $\Sigma_3(L)$ is an L-space or $\Sigma_4(L)$ is not an L-space.  

\vspace{.2cm} \item If $L$ is the link $6_3^2 = \widehat{b(1,2,2)}$ (cf. Figure \ref{fig: b(1,2,2)}), it must be shown that either $\Sigma_3(L)$ and $\Sigma_4(L)$ are L-spaces or $\Sigma_4(L)$ and $\Sigma_5(L)$ are not L-spaces.  

\vspace{.2cm} \item If $L$ is the link $7_3^1 =  \widehat{b(2,2,2)}$ (cf. Figure \ref{fig: b(2,2,2)}), it must be shown that either $\Sigma_3(L)$ is an L-space or $\Sigma_4(L)$ is not an L-space. 

\end{itemize}
}
\end{remark}

\end{document}